\documentclass{article}

\usepackage{amssymb}
\usepackage{amsmath}
\usepackage[all]{xy}
\usepackage{amsthm}

\theoremstyle{plain}

\theoremstyle{definition}

\newtheorem{definition}{Definition}[section]

\theoremstyle{remark}

\newtheorem{remark}[definition]{Remark}

\theoremstyle{plain}

\newtheorem{theorem}[definition]{Theorem}
\newtheorem{lemma}[definition]{Lemma}
\newtheorem{corollary}[definition]{Corollary}
\newtheorem{proposition}[definition]{Proposition}

\begin{document}
\bibliographystyle{amsplain}
\relax

\def\AW[#1]^#2_#3{\ar@{->}@<.5ex>[#1]^{#2} \ar@{<-}@<-.5ex>[#1]_{#3}}
 \def\NAW[#1]{\ar@{->}@<.5ex>[#1] \ar@{<-}@<-.5ex>[#1]}

 \renewcommand{\arraystretch}{1}
 \renewcommand{\O}{\bigcirc}
 \newcommand{\OX}{\bigotimes}
 \newcommand{\OD}{\bigodot}
 \newcommand{\OV}{\O\llap{v\hspace{.6ex}}}
 \newcommand{\B}{\mbox{\Huge $\bullet$}}
  \newcommand{\D}{$\diamond$}

\xymatrixrowsep{1.4pc} \xymatrixcolsep{1.4pc}

\title{ Regular Kac-Moody superalgebras and\\ integrable highest weight modules}

\author{Crystal~Hoyt$^1$\footnote{Department of Mathematics, Weizmann Institute of Science, Rehovot 76100 Israel; Email: crystal.hoyt@weizmann.ac.il.}}

\footnotetext[1]{Supported in part by NSF grant DMS-0354321 at the University of California, Berkeley
and by ISF grant, no. 1142/07 at the Weizmann Institute of Science.}

\date{}
\maketitle

\begin{abstract}
We define regular Kac-Moody superalgebras and classify them using integrable modules.  We give conditions for irreducible highest weight modules of regular Kac-Moody superalgebras to be integrable.  This paper is a major part of the proof for the classification of finite-growth contragredient Lie superalgebras.
\end{abstract}


\setcounter{section}{-1}

\section{Introduction}
The results of this paper are a crucial part of the proof for the classification of contragredient Lie superalgebras with finite growth, and in particular, for the classification of finite-growth Kac-Moody superalgebras  \cite{H07,HS07}. Previously, such a classification was known only for contragredient Lie superalgebras with either symmetrizable Cartan matrices \cite{L86,L89}, or Cartan matrices with no zeros on the main diagonal, i.e. contragredient Lie superalgebras without simple isotropic roots \cite{K78}.  Several of the results of this paper are surveyed in \cite{S08}.

A contragredient Lie superalgebra $\mathfrak{g}(A)$ is a Lie superalgebra defined by a Cartan matrix $A$ \cite{K77,K90}.  A Lie superalgebra usually has more than one Cartan matrix. However, an odd reflection at a regular simple isotropic root allows one to move from one base to another (see Definitions~\ref{isotropicdef}, \ref{regulardef}) \cite{LSS86}.  An odd reflection yields a new Cartan matrix $A'$ such that $\mathfrak{g}(A')$ and $\mathfrak{g}(A)$ are isomorphic as Lie superalgebras.

A matrix which satisfies certain numerical conditions is called a generalized Cartan matrix (see Definition~\ref{def gcm}).  If $A$ is a generalized Cartan matrix then all simple isotropic roots are regular.
A contragredient Lie superalgebra $\mathfrak{g}(A)$ is said to be regular Kac-Moody if $A$ and any matrix $A'$, obtained by a sequence of odd reflections of $A$, are generalized Cartan matrices.  If $A$ is a generalized Cartan matrix and $\mathfrak{g}(A)$ has no simple isotropic roots, then $\mathfrak{g}(A)$ is regular Kac-Moody by definition.  Hence, we restrict our attention to regular Kac-Moody superalgebras which have a simple isotropic root.  Remarkably, there are only a finite number of such families.

It is shown in \cite{HS07} that if $\mathfrak{g}(A)$ is a finite-growth contragredient Lie superalgebra and the defining matrix $A$ has no zero rows, then simple root vectors of $\mathfrak{g}(A)$ act locally nilpotently on the adjoint module. This implies certain conditions on $A$ which are only slightly weaker than the conditions for the matrix to be a generalized Cartan matrix. For a finite-growth Lie superalgebra, these matrix conditions should still hold after odd reflections, which leads to the definition of a regular Kac-Moody superalgebra.  Remarkably, these superalgebras almost always have finite growth.  The exception is the family: $Q^{\pm}(m,n,t)$ with $m,n,t\in\mathbb{Z}_{\leq -1}$ and not all equal to $-1$.

By comparing the classification of regular Kac-Moody superalgebras given in this paper to the classification of
symmetrizable finite-growth contragredient Lie superalgebras in \cite{L86,L89} we obtain the following formulation of our classification theorem.

{\def\thedefinition{\ref{271}}\addtocounter{definition}{-1}
\begin{theorem}
If $A$ is a symmetrizable indecomposable matrix and the contragredient Lie superalgebra $\mathfrak{g}(A)$ has a simple
isotropic root, then $\mathfrak{g}(A)$ is a regular Kac-Moody superalgebra if and only if it has finite growth.

If $A$ is a non-symmetrizable indecomposable matrix and the contragredient Lie superalgebra $\mathfrak{g}(A)$ has a
simple isotropic root, then $\mathfrak{g}(A)$ is a regular Kac-Moody superalgebra if and only if it belongs to one of the
following three families: $q(n)^{(2)}$, $S(1,2,\alpha)$ with $\alpha\in\mathbb{C}\setminus\mathbb{Z}$, $Q^{\pm}(m,n,t)$ with $m,n,t\in\mathbb{Z}_{\leq -1}$.
\end{theorem}}

The non-symmetrizable contragredient Lie superalgebra $S(1,2,\alpha)$ appears in the list of conformal superalgebras \cite{KL89}.  It has finite growth, but is not regular Kac-Moody when $\alpha \in \mathbb{Z}$.  The non-symmetrizable regular Kac-Moody superalgebra $Q^{\pm}(m,n,t)$ was discovered during this classification (see Section~\ref{s2}).  If $m,n,t=-1$ then $Q^{\pm}(m,n,t)$ is just $q(3)^{2}$, which has finite growth.  Otherwise, $Q^{\pm}(m,n,t)$ has infinite growth, but is hyperbolic for small (absolute) values of $m$, $n$, and $t$.  An explicit realization of $Q^{\pm}(m,n,t)$ is still unknown and would be interesting.

For the proof of Theorem~\ref{271}, we classify the corresponding connected regular Kac-Moody diagrams (see Section~\ref{sec 1.3}) by using induction on the number of vertices (i.e. simple roots). A subdiagram of a regular Kac-Moody diagram is regular Kac-Moody, however if the subdiagram does not have an isotropic vertex then it is not part of the classification.  We work around this difficulty by using odd reflections.

We say that a regular Kac-Moody diagram is {\em subfinite} if it is connected, has an isotropic vertex, and satisfies the condition that all connected proper subdiagrams which have an isotropic vertex are of finite type.
In Section~\ref{s1}, we find all subfinite regular Kac-Moody diagrams by extending connected finite type diagrams which have an isotropic vertex. A diagram of a finite-dimensional or affine Kac-Moody superalgebra is subfinite regular Kac-Moody. In Section~\ref{s3}, we prove that every connected regular Kac-Moody diagram with an isotropic vertex is subfinite by using integrable modules and some explicit computations.

A highest weight module $V$ of a regular Kac-Moody superalgebra $\mathfrak{g}(A)$ is called integrable if for each real root $\alpha$ the element $X_{\alpha} \in \mathfrak{g}(A)_{\alpha}$ is locally nilpotent on $V$.  This is consistent with the original definition of integrable modules for affine Lie superalgebras. It is shown in \cite{KW01} that most
non-twisted affine Lie superalgebras do not have non-trivial irreducible integrable highest weight modules. The only
exceptions are $B(0,n)^{(1)}$, $A(0,m)^{(1)}$ and $C(n)^{(1)}$.  By the same method, one can show that the twisted affine Lie superalgebras, including $q(n)^{(2)}$, but excluding $A(0,2n-1)^{(2)}$, $A(0,2n)^{(4)}$ and $C(n)^{(2)}$, have only trivial irreducible integrable highest weight modules.

A regular Kac-Moody superalgebra is integrable under the adjoint action, hence, so are its submodules.  A non-trivial
extension of a regular Kac-Moody diagram $\Gamma$ yields a non-trivial integrable highest weight module of
$\mathfrak{g}(A_{\Gamma})$. Thus, if the Kac-Moody superalgebra corresponding to $\Gamma$ does not have non-trivial
irreducible integrable highest weight modules, then $\Gamma$ is not extendable. Thus, it remains to show that the
diagrams for $A(0,m)^{(1)}$, $C(n)^{(1)}$, $S(1,2,\alpha)$, and $Q^{\pm}(m,n,t)$ are not extendable, which we do by
direct computation.

Integrable irreducible highest weight modules for affine Lie superalgebras were described in \cite{KW01}. If $L(\lambda)$ is an irreducible highest weight module a regular Kac-Moody superalgebra which does not have an isotropic root, then $L(\lambda)$ is integrable if and only if $\lambda_i\in 2^{p(i)}\mathbb{Z}_{\geq 0}$.  We describe the integrable irreducible highest
weight modules for the remaining regular Kac-Moody superalgebras which have an isotropic root: $S(1,2,\alpha)$ with
$\alpha\in\mathbb{C}\setminus\mathbb{Z}$, and $Q^{\pm}(m,n,t)$ with $m,n,t\in\mathbb{Z}_{\leq -1}$ and not all equal
to $-1$. We show that these superalgebras have non-trivial irreducible integrable highest weight modules, and give
explicit conditions on the weights for an irreducible highest weight module to be integrable.\\

\textbf{Acknowledgement.}  The author would like to express her sincere gratitude to Professor Vera Serganova for helpful discussions and suggestions while supervising this dissertation work.

\section{Preliminaries}
\subsection{Contragredient Lie superalgebras}

Let $A$ be a $n\times n$ matrix over $\mathbb{C}$, $I=\{1,\ldots,n\}$ and $p:I \to {\mathbb Z}_{2} $ be a parity function.  Fix a vector space $\mathfrak{h}$ over $\mathbb{C}$ of dimension $2n-\operatorname{rk}(A)$.  Let $\alpha_{1},\dots ,\alpha_{n}\in{\mathfrak h}^{*}$ and $h_{1},\dots ,h_{n}\in{\mathfrak h}$ be linearly independent elements satisfying $\alpha_{j}\left(h_{i}\right)=a_{ij}$, where $a_{ij}$ is the $ij$-th entry of $A$.
Define a Lie superalgebra $\bar{\mathfrak{g}}(A)$ by generators $X_{1},\dots ,X_{n},Y_{1},\dots ,Y_{n}$ and $\mathfrak{h}$, and by relations
\begin{equation}
\left[h,X_{i}\right]=\alpha_{i}\left(h\right)X_{i}\text{,\hspace{.5cm} }\left[h,Y_{i}\right]=-\alpha_{i}\left(h\right)Y_{i}\text{,\hspace{.5cm}
}\left[X_{i},Y_{j}\right]=\delta_{ij}h_{i}, \label{equ1}\end{equation}
where the parity of $X_i$ and $Y_i$ is $p(i)$, and the elements of $\mathfrak{h}$ are even.

The {\em contragredient Lie superalgebra} given by the matrix $A$ is defined to be the quotient of $\bar{\mathfrak{g}}(A)$ by the unique maximal ideal that intersects $\mathfrak{h}$ trivially, and it is denoted by $\frak{g}(A)$.  We call $A$ the {\em Cartan matrix} of $\mathfrak{g}(A)$.  If $B=DA$ for some invertible diagonal matrix $D$, then $\mathfrak{g}(B)\cong\mathfrak{g}(A)$.  Hence, we may assume without loss of generality that $a_{ii}\in\{0,2\}$ for $i\in I$.

The matrix $A$ is said to be {\em symmetrizable} if there exists an invertible diagonal matrix $D$ such that $B=DA$ is a symmetric matrix, i.e. $b_{ij}=b_{ji}$ for all $i,j\in I$.  In this case, we also say that $\mathfrak{g}(A)$ is symmetrizable. A contragredient Lie superalgebra is symmetrizable if and only if there exists a nondegenerate invariant symmetric bilinear form.  Hence, symmetrizability does not depend on the choice of Cartan matrix.

The matrix $A$ is {\em indecomposable} if the the set $I$ can not be decomposed into the disjoint union of non-empty subsets $J,K$ such that $a_{j,k}=a_{k,j}=0$ whenever $j\in J$ and $k\in K$.  A proof of the following lemma can be found in \cite{HS07}.

\begin{lemma} \label{lm20} For any subset $ J\subset I $ the subalgebra $\mathfrak{a}_{J}$ in $\mathfrak{g}\left(A\right) $ generated by $ \mathfrak{h} $, $ X_{i} $ and $ Y_{i} $, with $ i\in J $, is isomorphic to $ \mathfrak{h}'\oplus{\mathfrak g}\left(A_{J}\right) $, where $ A_{J} $ is the submatrix of $A$
with coefficients $ \left(a_{ij}\right)_{i,j\in J} $ and $\mathfrak{h}' $ is a subspace of $\mathfrak{h} $. More
precisely, $\mathfrak{h}' $ is a maximal subspace in $ \cap_{i\in J} \operatorname{Ker} \alpha_{i} $ which trivially
intersects the span of $ h_{i} $, $ i\in J $.
\end{lemma}

A superalgebra $\mathfrak{g}=\mathfrak{g}(A)$ has a natural $\mathbb{Z}$-grading $\mathfrak{g}=\oplus\mathfrak{g}_{m}$, called the {\em principal grading}, which is defined by  $\mathfrak{g}_{0}=\mathfrak{h} $ and $\mathfrak{g}_{1}=\mathfrak{g}_{\alpha_{1}}\oplus\dots \oplus\mathfrak{g}_{\alpha_{n}} $.  We say that ${\mathfrak g} $ is of {\em finite growth} if $ \dim\mathfrak{g}_{n} $ grows polynomially depending on $n$. This means that the Gelfand-Kirillov dimension of $\mathfrak{g}$ is finite.

We recall a result from \cite{HS07}.

\begin{theorem}[Hoyt, Serganova] \label{th1}\relax  Suppose $A$ is a matrix with no zero rows. If $\mathfrak{g}(A) $ has finite growth, then $\mbox{ad}_{X_{i}}$ are locally nilpotent for all $i\in I$.
\end{theorem}

A proof of the following lemma can be found in \cite{HS07}.

\begin{lemma} \label{lm23}\relax  Let $ {\mathfrak g}\left(A\right) $ be a contragredient Lie superalgebra.
Then $\mbox{ad}_{X_{i}}$ are locally nilpotent for all $i\in I$ if and only if $A$ satisfies the following conditions, (after rescaling the rows of $A$ such that $a_{ii}\in\{0,2\}$ for all $i\in I$):
\begin{enumerate}
\item if $ a_{i i}=0 $ and $p(i)=0$, then $a_{ij}=0$ for all $j\in I$;
\item if $ a_{i i}=2 $, then $ a_{ij}\in2^{p\left(i\right)}{\mathbb Z}_{\leq 0} $ for all $j\in I$;
\item if $ a_{ij}=0 $ and $ a_{ji}\not=0 $, then $ a_{i i}=0 $.
\end{enumerate}
\end{lemma}

\subsection{Roots and Reflections}
The Lie superalgebra $ {\mathfrak g}={\mathfrak g}\left(A\right) $ has a {\em root space decomposition}
\begin{equation}
{\mathfrak g}={\mathfrak h}\oplus\bigoplus_{\alpha\in\Delta}{\mathfrak g}_{\alpha}.
\notag\end{equation}
Every root is either a positive or a negative linear combination of the simple roots, $\alpha_1,\ldots,\alpha_n$.  Accordingly, we  have the decomposition $ \Delta=\Delta^{+}\cup\Delta^{-} $, and we call $\alpha\in\Delta^{+}$ positive and $\alpha\in\Delta^{-}$ negative.
One can define $ p:\Delta \to {\mathbb Z}_{2} $ by letting $ p\left(\alpha\right)=0 $ or 1 whenever $\alpha $ is even or odd, respectively. By $ \Delta_{0}$ (resp. $\Delta_{1}$) we denote the set of even (resp. odd) roots.

Let $\Pi:=\{\alpha_1,\ldots,\alpha_n\}$ be the set of simple roots of $\mathfrak{g}(A)$.
There are four possibilities for each simple root $\alpha_{i} $:
\begin{enumerate}
\item if $ a_{ii}=2 $ and $ p(\alpha_i)=0 $,
then $ X_{i} $, $ Y_{i} $ and $ h_{i} $ generate a subalgebra isomorphic to $ sl\left(2\right) $;

\item if $ a_{ii}=0 $ and $ p(\alpha_i)=0 $,
then $ X_{i} $, $ Y_{i} $ and $ h_{i} $ generate a subalgebra isomorphic to the Heisenberg algebra;

\item if $ a_{ii}=2 $ and $ p(\alpha_i)=1 $, then $ X_{i} $, $ Y_{i} $ and $ h_{i} $ generate a subalgebra isomorphic to $ \operatorname{osp}\left(1|2\right) $, and in this case $2\alpha_i\in\Delta $;

\item if $ a_{ii}=0 $ and $ p(\alpha_i)=1 $, then  $ X_{i} $, $ Y_{i} $ and $
h_{i} $ generate a subalgebra isomorphic to $ sl\left(1|1\right) $.
\end{enumerate}
\begin{definition}\label{isotropicdef}
A simple root $ \alpha_i$ is {\em isotropic} if $ a_{ii}=0 $ and $ p(\alpha_i)=1 $, and otherwise $\alpha_i$ is {\em non-isotropic}.
\end{definition}
\begin{definition}\label{regulardef}
A simple root $\alpha_i$ is {\em regular} if for any other simple root $\alpha_j$, $a_{ij}=0$ implies $a_{ji}=0$.
\end{definition}

If $\alpha_k$ is a simple root with $a_{kk}\neq 0$, we define the (even) reflection $r_{k}$ at $\alpha_k$ by
\begin{equation*}
r_{k}(\alpha_i)=\alpha_i - \alpha_i(h_k)\alpha_k,\hspace{2cm}\alpha_i\in\Pi.
\end{equation*}

If $\alpha_k$ is a regular isotropic root, we define the {\em odd reflection} $r_{k}$ at $\alpha_k$ as follows:
\begin{equation*}r_{k}(\alpha_{i})=\left\{
  \begin{array}{ll}
    -\alpha_{k}, & \hbox{if $i=k$;} \\
    \alpha_{i}, & \hbox{if $a_{ik}=a_{ki}=0$, $i \neq k$;} \\
    \alpha_{i}+\alpha_{k}, & \hbox{if $a_{ik}\neq 0$ or $a_{ki} \neq 0$,  $i \neq k$;}
  \end{array}
\right.\end{equation*}

\begin{equation*}
 X_{i}':=\left\{
  \begin{array}{ll}
     Y_{i}, & \hbox{if $i=k$;}\\
     X_{i}, & \hbox{if $i\neq k$, and $a_{ik}=a_{ki}=0$;}\\
    {[X_{i},X_{k}]}, & \hbox{if $i\neq k$, and $a_{ik}\neq 0$ or $a_{ki}\neq 0$;}
\end{array}\right.
\end{equation*}

\begin{equation*}
 Y_{i}':=\left\{
  \begin{array}{ll}
     X_{i}, & \hbox{if $i=k$;}\\
     Y_{i}, & \hbox{if $i\neq k$, and $a_{ik}=a_{ki}=0$;}\\
    {[Y_{i},Y_{k}]}, & \hbox{if $i\neq k$, and $a_{ik}\neq 0$ or $a_{ki}\neq 0$;}
\end{array}\right.
\end{equation*}
and
\begin{equation*}
h_{i}' := [X_{i}',Y_{i}'].
\end{equation*}
Then $$ h_{i}' = \begin{cases}
    (-1)^{p(\alpha_{i})} (a_{ik}h_{k}+a_{ki}h_{i}) &\text{ if }i \neq k\text{, and }a_{ik} \text{ or }a_{ki} \neq 0, \\
    h_{i} &\text{ if }i \neq k\text{, and }a_{ik}=a_{ki}=0, \\
    h_{k} &\text{ if } i=k. \end{cases} $$
Set $\alpha'_i:=r_{k}(\alpha_i)$ for $i\in I$.

A proof of the following lemma can be found in \cite{HS07}.
\begin{lemma} \label{lm1}\relax
Let $\mathfrak{g}(A)$ be a contragredient Lie superalgebra with base $\Pi=\{\alpha_{1},\ldots,\alpha_{n}\}$. Suppose that $\Pi'=\{ \alpha'_{1},\dots ,\alpha'_{n} \}$ is is obtained from $\Pi$ by an odd reflection with respect to a regular isotropic root. Then $\alpha'_{1},\dots ,\alpha'_{n}$ are linearly independent. The corresponding elements (defined above) $
X'_{1},\dots ,X'_{n}$, $Y'_{1},\dots ,Y'_{n} $ together with $ h'_{1},\dots ,h'_{n} $ satisfy~\eqref{equ1}. Moreover, $ {\mathfrak h} $ and $ X'_{1},\dots ,X'_{n},Y'_{1},\dots ,Y'_{n} $ generate $ {\mathfrak
g}(A) $.
\end{lemma}

It then follows that given a matrix $A$ and a regular isotropic root $\alpha_k$, one can construct a new matrix $A'$ such that $\mathfrak{g}(A')\cong\mathfrak{g}(A)$ as superalgebras.  Explicitly, the entries of $A'$ can be defined by $A'_{ij}=\alpha_{j}'(h_{i}')$.  After possibly rescaling the elements $h_i'$, we find that ($ i \neq k$ and $ j \neq k $):
\begin{equation*}
    a_{kk}' := a_{kk}; \hspace{1cm}
    a_{kj}' := a_{kj}; \hspace{1cm}
    a_{ik}' := -a_{ki}a_{ik}; \end{equation*}\begin{equation*}
    a_{ij}' := \begin{cases}
        a_{ij}, &\text{ if } a_{ik}=a_{ki}=0; \\
        a_{ki}a_{ij}, &\text{ if }a_{ik} \text{ or } a_{ki} \neq 0, \text{ and } a_{kj}=a_{jk}=0; \\
        a_{ki}a_{ij}+a_{ik}a_{kj}+a_{ki}a_{ik}, &\text{ if }a_{ik} \text{ or } a_{ki}\neq 0, \text{ and } a_{jk} \text{ or } a_{kj} \neq 0.
        \end{cases}
\end{equation*}

\begin{remark}\label{remrescale}
We can rescale the rows of $A'$ by rescaling the elements $h_i'$.  So after rescaling, we may assume that $a_{ii}'=\alpha_i'(h_i')=0$ or $2$.  In the case that $a_{ii}'=0$ for some $i$, it is our convention to rescale $A'$ so that $a_{ij}'=1$ for some $j$.
\end{remark}

We say that $A'$ is obtained from $A$ (and $\Pi':=\{\alpha'_1,\ldots,\alpha'_n\}$ is obtained from $\Pi$) by an odd reflection with respect to $\alpha_k$.  If $\Delta'^{+}$ is the set of positive roots with respect to $\Pi'$, then
\begin{equation*}
\Delta'^{+}=\left(\Delta^{+}\setminus\{\alpha_k\}\right)\cup\{-\alpha_k\}.
\end{equation*}
A root $\alpha$ is called {\em real} if $\alpha$ or $\frac{1}{2} \alpha$ is simple in some base $\Pi'$, which is obtained from $\Pi$ by a sequence of even and odd reflections. Otherwise, it is called {\em imaginary} otherwise.
\begin{remark}
A reflection with respect to a regular isotropic simple root $\alpha_k$ is indeed a reflection.  If $A''$ is obtained from $A$ by successively applying the reflection at $\alpha_k$ twice, then there is an invertible diagonal matrix $D$ such that $A''=DA$ and scalars $b_i,c_i$ such that $X''_i=b_i X_i$ and $Y''_i=c_i Y_i$.  It is possible to define an ``odd reflection'' at a simple isotropic root which is not regular, but in this case the subalgebra generated by $ X'_{1},\dots ,X'_{n},Y'_{1},\dots ,Y'_{n} $ and $ {\mathfrak h} $ is necessarily a proper subalgebra of $\mathfrak{g}(A)$.
\end{remark}

The notion of finite growth does not depend on the choice of a base for $\mathfrak{g}(A)$ \cite{HS07}.  It thus follows from Theorem~\ref{th1} that if $\mathfrak{g}(A)$ has finite growth and $A$ has no zero rows, then $A$ and any matrix $A'$ obtained from $A$ by a sequence of odd reflections satisfies the matrix conditions of Lemma~\ref{lm23}.

\begin{definition}\label{def gcm} A matrix $A$ is called a {\em generalized Cartan matrix} if it satisfies the matrix conditions of Lemma~\ref{lm23}, where the third condition is strengthened to the following:
\begin{equation*}
\text{3}'.\ \text{ if } a_{ij}=0 \text{, then } a_{ji}=0.
\end{equation*}\end{definition}
If a matrix satisfies the conditions of Lemma~\ref{lm23}, then condition (3$'$) is equivalent to the condition that all simple isotropic roots are regular.
We call $\mathfrak{g}(A)$ a {\em regular Kac-Moody superalgebra} if $A$ and any matrix obtained from $A$ by a sequence of odd reflections is a generalized Cartan matrix.

We call a root $\alpha\in\Delta$ {\em principal} if either $\alpha$ is even and belongs to some base $\Pi'$ obtained from $\Pi$ be a sequence of odd reflections, or if $\alpha=2\beta$ where $\beta$ is odd and belongs to some base $\Pi'$ obtained from $\Pi$ be a sequence of odd reflections.  For a principal root, the subalgebra generated by $X_{\alpha}$, $Y_{\alpha}$ and $h_{\alpha}:=[X_{\alpha},Y_{\alpha}]$ is isomorphic to $\mathfrak{sl}_2$, and we may choose $X_{\alpha}$, $Y_{\alpha}$ such that $\alpha(h_{\alpha})=2$.  Note that if $\alpha=2\beta$, then $X_{\alpha}=[X_{\beta},X_{\beta}]$.

Let $\Pi_{0}\subset\Delta$ denote the set of principal roots.  It is clear that $\Pi_{0}\subset\Delta_{0}^{+}$.  In general this set can be infinite, but this is not the case whenever $\frak{g}(A)$ has finite growth \cite{HS07}.
For any finite subset $S\subset \Pi_{0}$, we can define a matrix $B$ by setting $b_{ij}=\alpha_j(h_i)$.

A proof of the following lemma can be found in \cite{HS07}.

\begin{lemma}\label{lemma B}   If a Lie superalgebra $ {\mathfrak g}\left(A\right) $ has finite growth, then for any finite subset $S$ of ${\Pi}_0$ the Lie algebra $ {\mathfrak g}\left(B\right) $ also has finite growth. In particular, $ B $ is a generalized Cartan matrix for a finite growth Kac-Moody algebra.
\end{lemma}

\subsection{Matrix diagrams}\label{sec 1.3}
Given a Cartan matrix $A$, we can associate a {\em matrix diagram}, denoted $\Gamma_{A}$ (or simply $\Gamma$ when $A$ is fixed), as follows.  Recall that we may assume that $a_ii=0$ or $2$.  The vertices of $\Gamma_{A}$ correspond to the simple roots of $\mathfrak{g}(A)$ and are given by the following table.
\begin{equation*}
\begin{tabular}{|c|c|c|c|}
\hline
$\mathfrak{g}(A)$ & $A$ & $p(1)$ & Diagram \\
\hline
$A_1$ & (2) & 0 & $\O$ \\
$B(0,1)$ & (2) & 1 & \B \\
$A(0,0)$ & (0) & 1 & $\OX$ \\
Heisenberg & (0) & 0 & $\square$ \\
\hline
\end{tabular}
\end{equation*}

We join vertex $v_i$ to vertex $v_j$ by an arrow if $a_{ij} \neq 0$, and we label this arrow with the number $a_{ij}$.
This correspondence between Cartan matrices and matrix diagrams is a bijection.  The condition that the matrix $A$ is
indecomposable corresponds to the requirement that the diagram $\Gamma_{A}$ is connected.  We say that a vertex is
isotropic if it is of the type $\OX$. We use the notation of the following table to denote the possible vertex types.
\begin{center}
\begin{tabular}{|c|c|}
\hline
Notation & Vertex Types \\
\hline
$\bullet$ & $\O$ \text{ or } $\OX$\\
$\OD$ & $\O$ \text{ or } $\B$ \\
$\OV$ & $\O$ \text{ or } $\OX$ \text{ or } $\B$ \\
\hline
\end{tabular}
\end{center}

Let $A$ be a Cartan matrix with $a_{ii}=0$.  Then rescaling the $i^{th}$ row of the matrix $A$ (i.e. multiplying by a
non-zero constant) corresponds to rescaling all labels of arrows exiting the vertex $v_i$ of the diagram $\Gamma_{A}$.
This defines is an equivalence relation on matrix diagrams, where $\Gamma_{A'}\sim\Gamma_{A}$ if $A'=DA$ for some
invertible diagonal matrix $D$, and in this case $\mathfrak{g}(A')\cong\mathfrak{g}(A)$. We consider matrix diagrams modulo this
equivalence.

A diagram $\Gamma_{A}$ is called a {\em regular Kac-Moody} (or regular Kac-Moody diagram) if the corresponding contragredient Lie superalgebra $\mathfrak{g}(A)$ is regular Kac-Moody.  An odd reflection of a matrix diagram
$\Gamma_{A}$ at an isotropic vertex $v_{i}$ is defined to be the diagram $\Gamma_{A'}$, where $A'$ is obtained from $A$
by an odd reflection at the corresponding isotropic root $\alpha_i$ of $\mathfrak{g}(A)$.  Note that the set of all regular Kac-Moody diagrams is closed under odd reflections. Denote by $C(\Gamma)$ the collection of all diagrams obtained from sequences of odd reflections of a regular Kac-Moody diagram $\Gamma$.  When $\alpha_i$ is a simple odd isotropic root, we denote by $r_{i}$ the odd reflection with respect to $\alpha_i$.

By a {\em subdiagram} $\Gamma'$ of the diagram $\Gamma$, we mean a full subdiagram, i.e. if the vertices $v_i$ and $v_j$ are in  $\Gamma'$ then the arrows with labels $a_{ij}$ and $a_{ji}$ also belong to $\Gamma'$.
We say that a connected regular Kac-Moody diagram is {\em extendable} if it is a proper subdiagram of a connected regular Kac-Moody diagram. For a 3-vertex subdiagram $\Gamma' = \{v_i,v_j,v_k\}$, with $v_j$ isotropic, we refer to the fractions $\frac{a_{ji}}{a_{jk}}$ and $\frac{a_{jk}}{a_{ji}}$ as the {\em ratios} of the isotropic vertex $v_{j}$ in $\Gamma'$.

\section{Classification of connected subfinite regular Kac-Moody diagrams}\label{s1}

In order to classify regular Kac-Moody superalgebras, we classify the corresponding regular Kac-Moody diagrams.  A diagram for a finite-dimensional or affine Kac-Moody superalgebra is regular Kac-Moody.
\begin{definition}\label{defsubfinite}
We call a regular Kac-Moody diagram {\em subfinite} if it is connected, it contains an isotropic vertex, and it satisfies the following condition for all reflected diagrams: all connected proper subdiagrams which contain an isotropic vertex are of finite type.
\end{definition}
In Section~\ref{s35}, we will show that all regular Kac-Moody diagrams are subfinite.

In this section, we classify subfinite regular Kac-Moody diagrams, by using induction on the number of vertices.
A subdiagram of a subfinite regular Kac-Moody diagram is regular Kac-Moody, however if it does not contain an isotropic vertex then it is not part of our classification.  We will work around this difficulty by using odd reflections.

Note that we only need to find the extensions of one diagram $\Gamma$ belonging to a collection of
reflected diagrams $\mathcal{C}(\Gamma)$, and then include all reflections of each extended diagram in our
classification.

\subsection{Regular Kac-Moody diagrams: 2 or 3 vertices}\label{rkm23}
In this section, we find all connected diagrams with two or three vertices which are regular Kac-Moody and contain an isotropic vertex.

We say that an $n$-vertex diagram is a {\em cycle} if it is a connected diagram and each vertex is connected to exactly two other vertices.  We say that an $n$-vertex diagram is a {\em chain} if the vertices can be enumerated by the set
$\{1,2,...,n\}$ such that $a_{ij}=0$ if and only if $j \neq i+1$ and $i \neq j+1$.  A proper connected
subdiagram of a cycle is a chain.

\begin{lemma}
The connected regular Kac-Moody 2-vertex diagrams which contain an isotropic vertex are $A(1,0)$ and $B(1,1)$.
\end{lemma}
\begin{proof}
Recall that in the case that $a_{ii}=0$ for some $i$, it is our convention to rescale so that $a_{ij}=1$ for some $j$.
\begin{equation*}
    \xymatrix{ \OX \AW[r]^{1}_{a} & \OV}
\end{equation*}
\D If $a\neq -1$, then reflecting at $v_{1}$, we have $a, \frac{-a}{a+1} \in {\mathbb Z}_{<0}$ which implies $a=-2$.  Then $\frac{-a}{a+1}=-2$ and this is $B(1,1).$
\begin{equation*}
    \xymatrix{ \OX \AW[r]^{1}_{a} & \OD}
    \text{\hspace{.5cm}} \overrightarrow{r_{1}} \text{\hspace{.5cm}}
    \xymatrix{ \OX \AW[r]^{1}_{- \frac{a}{1+a}} & \OD}
\end{equation*}\\
\D If $a=-1$, then by reflecting at $v_{1}$ we have $A(1,0)$.
\begin{equation*}
    \xymatrix{ \OX \AW[r]^{1}_{-1} & \O }
    \text{\hspace{.5cm}} \overrightarrow{r_{1}} \text{\hspace{.5cm}}
    \xymatrix{ \OX \AW[r]^{1}_{1} & \OX }
    \end{equation*}
\end{proof}
We note that all 2-vertex diagrams of regular Kac-Moody superalgebras are of finite type.

\begin{lemma}
The regular Kac-Moody extensions of $A(1,0)$ to three vertices are the following: $A(0,2)$, $A(1,1)$, $B(1,2)$,
$B(2,1)$, $C(3)$, $G(3)$, $D(2,1,\alpha)$, $A(1,2)^{(2)}$, $A(0,1)^{(1)}$, $B(1,1)^{(1)}$, $S(1,2,\alpha)$, $q(3)^{(2)}$,
$Q^{\pm}(m,n,t)$.
\end{lemma}
\begin{proof} We consider each case for attaching a vertex to an $A(1,0)$ diagram. \\
\textbf{Case 1:}
\begin{equation*}
\xymatrix{ \OX_1 \AW[r]^{1}_{1} & \OX_2 \AW[r]^{a}_{b} & \OV_3} \hspace{1cm} a,b \neq 0
\end{equation*}\\
\D If $v_3$ is $\OX$, then by reflecting at $v_2$ we have $1+a, 1+\frac{1}{a} \in \mathbb{Z}_{\leq 0}$, which implies
$a=-1$. This is $A(1,1)$.
\begin{equation*}
\xymatrix{ \OX \AW[r]^{1}_{1}  & \OX \AW[r]^{a}_{1}  & \OX} \text{ \hspace{.5cm} } \overrightarrow{r_{2}} \text{  }
\xymatrix{& \OX \AW[ldd]^{1}_{-1} \AW[rdd]^{a}_{-1} & \\ & & \\ \O \AW[rr]^{1+a}_{1+\frac{1}{a}} & & \O }
\overset{a=-1}{\Longrightarrow}\text{  } \xymatrix{ \O \AW[r]^{-1}_{1}  & \OX \AW[r]^{-1}_{-1}  & \O}
\end{equation*}\\
\D If $v_3$ is $\OD$ and $b=-2$, then by reflecting at $v_{2}$ we have $1+a, 2+\frac{2}{a} \in \mathbb{Z}_{\leq 0}$,
which implies $a=-1$. This is $B(1,2)$.
\begin{equation*}
\xymatrix{ \OX \AW[r]^{1}_{1}  & \OX \AW[r]^{a}_{-2}  & \OD } \text{ \hspace{.5cm} } \overrightarrow{r_{2}} \text{  }
\xymatrix{& \OX \AW[ldd]^{1}_{-1} \AW[rdd]^{a}_{-2} & \\ & & \\ \O \AW[rr]^{1+a}_{2+\frac{2}{a}} & & \OD }
\overset{a=-1}{\Longrightarrow}\text{  } \xymatrix{ \O \AW[r]^{-1}_{1}  & \OX \AW[r]^{-1}_{-2}  & \OD}
\end{equation*}\\
\D If $v_3$ is $\O$ and $a,b=-1$, then this is $A(0,2)$.
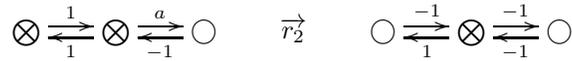
\begin{equation*}
\xymatrix{ \OX \AW[r]^{1}_{1}  & \OX \AW[r]^{a}_{-1}  & \O } \text{ \hspace{.5cm} } \overrightarrow{r_{2}}\text{ \hspace{.5cm} }  \xymatrix{ \O \AW[r]^{-1}_{1}  & \OX \AW[r]^{-1}_{-1}  & \O}
\end{equation*} \\
\D If $v_3$ is $\O$, $b=-1$ and $a\neq -1$, then by reflecting at $v_{2}$ we have $1+a \in \{-1,-2\}$, which implies $a
\in \{-2,-3\}$. If $a=-2$, then $1+a=-1$ and this is $C(3)$. If $a=-3$, then $1+a=-2$ and this is $G(3)$.
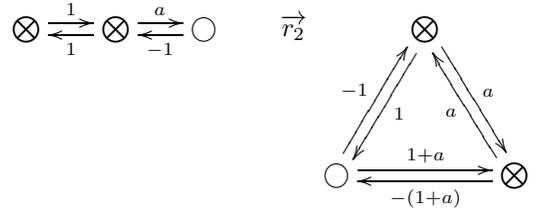
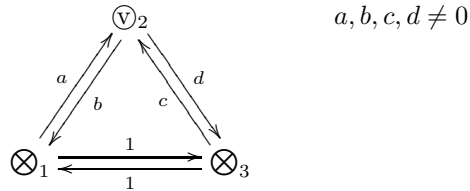
\begin{equation*}
\xymatrix{ \OX \AW[r]^{1}_{1} & \OX \AW[r]^{a}_{-1} & \O } \text{ \hspace{.5cm} } \overrightarrow{r_{2}} \text{  }
\xymatrix{& \OX \AW[ldd]^{1}_{-1} \AW[rdd]^{a}_{a} & \\ & & \\ \O \AW[rr]^{1+a}_{-(1+a)} & & \OX }.
\end{equation*}\\ \\
\textbf{Case 2:}
\begin{equation*}
\xymatrix{& \OV_2 \AW[ldd]^{b}_{a} \AW[rdd]^{d}_{c} & \\ & & \\ \OX_1 \AW[rr]^{1}_{1} & &\OX_3 } \hspace{1cm} a,b,c,d
\neq 0
\end{equation*}\\
\D If $v_2$ is $\OX$, then after rescaling we have
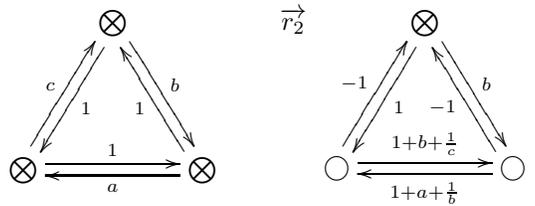
\begin{equation*}
\xymatrix{& \OX \AW[ldd]^{1}_{c} \AW[rdd]^{b}_{1} & \\ & & \\ \OX \AW[rr]^{1}_{a} & &\OX } \text{ \hspace{.5cm} }
\overrightarrow{r_{2}} \text{  }
\xymatrix{& \OX \AW[ldd]^{1}_{-1} \AW[rdd]^{b}_{-1} & \\
& & \\ \O \AW[rr]^{1+b+\frac{1}{c}}_{1+a+\frac{1}{b}} & & \O}.
\end{equation*}
Then $1+a+\frac{1}{b}, 1+b+\frac{1}{c} \in \mathbb{Z}_{<0}$ or both zero. By symmetry it follows that
\begin{equation}\label{e3}
\begin{array}{l}
1+a+\frac{1}{b}=m \\
1+b+\frac{1}{c}=n \\
1+c+\frac{1}{a}=t \\
\end{array}
\in \mathbb{Z}_{<0} \text{, or all equal to zero.}
\end{equation}
If they all equal zero, then this is $D(2,1,\alpha)$.  If they all equal $-1$, then $a,b,c=-1$ and this is $q(3)^{(2)}$.
If they are in $\mathbb{Z}_{<0}$ and not all equal to $-1$, then this is $Q^{\pm}(m,n,t)$. \\ \\
\D If $v_2$ is $\OD$ and $b,d=-2$, then by reflecting at $v_1$ we have $4+\frac{2}{a},4+\frac{2}{c},1+a+c \in
\mathbb{Z}_{<0}$ or all zero, which implies $a,c=- \frac{1}{2}$.
If $v_2$ is $\O$ then this is $A(1,2)^{(2)}$, and if $v_2$ is $\B$ then this is $B(1,1)^{(1)}$.\\
\begin{equation*}
\xymatrix{& \OD  \AW[ldd]^{-2}_{a} \AW[rdd]^{-2}_{c} & \\ & & \\ \OX \AW[rr]^{1}_{1} & & \OX } \text{\hspace{.5cm} }
\overrightarrow{r_{1}} \text{  }
\xymatrix{& \OD \AW[ldd]^{-2}_{a} \AW[rdd]^{4+\frac{2}{a}}_{Q} & \\
& & \\ \OX \AW[rr]^{1}_{-1} & & \O } \text{ } Q=1+a+c.
\end{equation*}\\
\D If $v_2$ is $\O$, $b=-2$ and $d=-1$, then by reflecting at $v_{1}$ we have
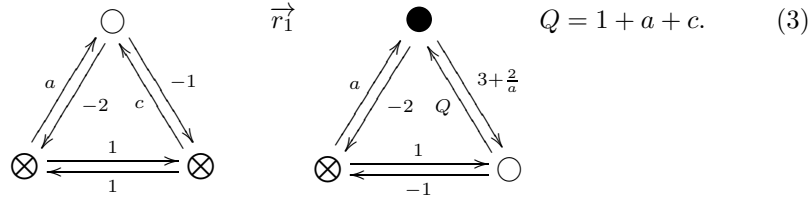
\begin{equation}\label{d2}
\xymatrix{& \O \AW[ldd]^{-2}_{a} \AW[rdd]^{-1}_{c} & \\ & & \\ \OX \AW[rr]^{1}_{1} & & \OX } \text{\hspace{.5cm} }
\overrightarrow{r_{1}} \text{  } \xymatrix{& \B \AW[ldd]^{-2}_{a} \AW[rdd]^{3+\frac{2}{a}}_{Q} & \\ & & \\ \OX
\AW[rr]^{1}_{-1} & & \O} \text{ } Q=1+a+c.
\end{equation}
Then either $1+a+c \in \mathbb{Z}_{<0}$ and $3+\frac{2}{a} \in 2\mathbb{Z}_{<0}$, or both equal zero. If they are both zero, then $a=- \frac{2}{3}, c=- \frac{1}{3}$ and this is G(3).

So now assume that $1+a+c \neq 0$.  Then by reflecting at $v_{3}$ we have
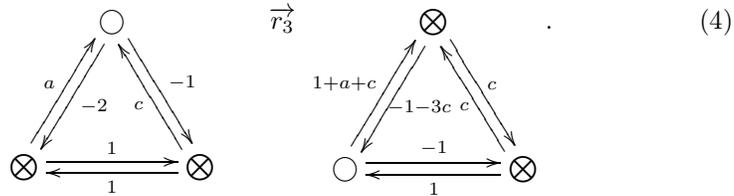
\begin{equation}\label{d1}
\xymatrix{& \O \AW[ldd]^{-2}_{a} \AW[rdd]^{-1}_{c} & \\ & & \\ \OX \AW[rr]^{1}_{1} & & \OX } \text{\hspace{.5cm} }
\overrightarrow{r_{3}} \text{  } \xymatrix{& \OX \AW[ldd]^{\!\!\!\!-1-3c}_{1+a+c} \AW[rdd]^{\!c}_{c\!} & \\ & & \\ \O
\AW[rr]^{-1}_{1} & & \OX }.
\end{equation}
This implies $1+a+c \in \{-1,-2\}$.  Since $3+\frac{2}{a}\leq -2$, it follows that $c\in (-2,-1\frac{3}{5}]$  in the first case ($1+a+c=-1$) and $c\in (-3,-2\frac{3}{5}]$ in the second case.

If $1+a+c=-1$, then by reflecting the last diagram of (\ref{d1}) at $v_2$ we have:
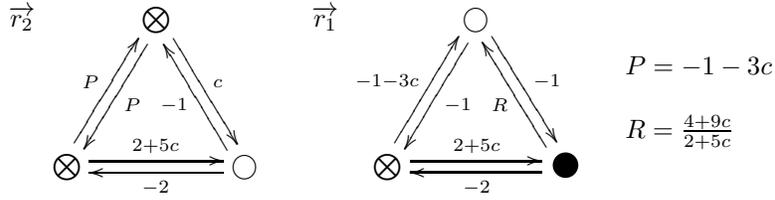
\begin{equation*}
\begin{array}{c}\overrightarrow{r_{2}} \text{  } \xymatrix{& \OX \AW[ldd]^{P}_{P} \AW[rdd]^{c}_{-1} & \\ & & \\ \OX
\AW[rr]^{2+5c}_{-2} & & \O } \text{\hspace{.5cm} } \overrightarrow{r_{1}} \text{  } \xymatrix{& \O \AW[ldd]^{-1}_{-1-3c}
\AW[rdd]^{-1}_{R} & \\ & & \\ \OX \AW[rr]^{2+5c}_{-2} & & \B }\end{array} \text{ } \begin{array}{l}
P=-1-3c\\ \\
R=\frac{4+9c}{2+5c} \end{array}.
\end{equation*}
Then $R=\frac{4+9c}{2+5c} \in 2\mathbb{Z}_{<0}$, which contradicts  $c\in (-2,-1\frac{3}{5}]$.

If $1+a+c=-2$, then by reflecting the last diagram of (\ref{d1}) at $v_{2}$ we have
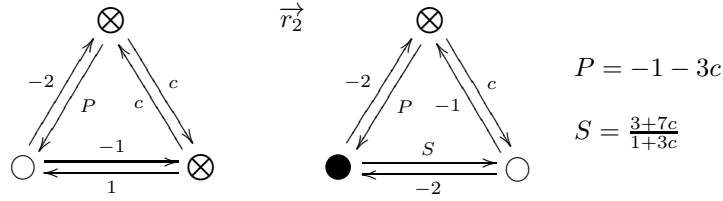
\begin{equation*}
\begin{array}{c}\xymatrix{& \OX \AW[ldd]^{P}_{-2} \AW[rdd]^{c}_{c} & \\ & & \\ \O \AW[rr]^{-1}_{1} & & \OX } \text{ \hspace{.5cm} }
\overrightarrow{r_{2}} \text{  } \xymatrix{& \OX \AW[ldd]^{P}_{-2} \AW[rdd]^{c}_{-1} & \\ & & \\ \B \AW[rr]^{S}_{-2} & &
\O }\end{array} \text{ } \begin{array}{l} P=-1-3c\\ \\S=\frac{3+7c}{1+3c}  \end{array}.
\end{equation*}
Then $S=\frac{3+7c}{1+3c} \in 2\mathbb{Z}_{<0}$, which contradicts $c\in (-3,-2\frac{3}{5}]$.\\

\D If $v_2$ is $\O$ and $b,d=-1$, then by reflecting at $v_{1}$ we have
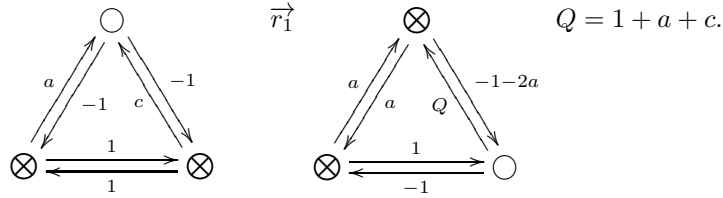
\begin{equation*}
\xymatrix{& \O \AW[ldd]^{-1}_{a} \AW[rdd]^{-1}_{c} & \\ & & \\ \OX \AW[rr]^{1}_{1} & & \OX } \text{\hspace{.5cm} }
\overrightarrow{r_{1}} \text{  } \xymatrix{& \OX \AW[ldd]^{a}_{a} \AW[rdd]^{-1-2a}_{Q} & \\ & & \\ \OX \AW[rr]^{1}_{-1} &
& \O } \text{ } Q=1+a+c.
\end{equation*}
Then either $1+a+c = 0$ and $1+2a = 0$, or $1+a+c \in \{-1,-2\}$. If $1+a+c = 0$, then this is $C(3)$. We have that $1+a+c\neq -2$ by the previous case, since the first diagram is not a diagram for $G(3)$.

Suppose now that $1+a+c=-1$.  If $c=-1$, then this is $A(0,1)^{(1)}$. For $c \neq -1$ the substitution
$c=\frac{1-\alpha}{\alpha}$ is reversible. After substituting and rescaling, we have
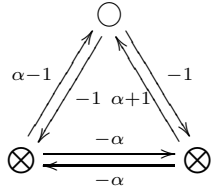
\begin{equation*}
\xymatrix{& \O \AW[ldd]^{\!-1}_{\alpha-1} \AW[rdd]^{-1}_{\alpha+1\!\!\!} & \\ & & \\ \OX \AW[rr]^{-\alpha}_{-\alpha} & & \OX }.
\end{equation*}
This diagram is $S(1,2,\alpha)$, and is regular Kac-Moody precisely when $\alpha$ is not an integer.  Indeed,
by reflecting the following diagram at $v_{1}$ we have:
\begin{equation*}
\begin{array}{c}
\xymatrix{& \O \AW[ldd]^{\!-1}_{Q-1} \AW[rdd]^{Q+1}_{-1\!} & \\ & & \\ \OX \AW[rr]^{-Q}_{-Q}  & &\OX } \\
Q=\alpha+n\\
\end{array}
\overrightarrow{r_{1}}
\begin{array}{c}
\xymatrix{& \OX \AW[ldd]^{\!\!-R}_{-R} \AW[rdd]^{R-1}_{-1\!} & \\ & & \\ \OX \AW[rr]^{R+1}_{-1} & & \O } \\
R=Q-1=\alpha+(n-1)
\end{array}.
\end{equation*}
Reflecting at $v_{3}$ gives us
\begin{equation*}
\begin{array}{c}
\xymatrix{& \O \AW[ldd]^{\!-1}_{Q-1} \AW[rdd]^{Q+1}_{-1\!} & \\ & & \\ \OX \AW[rr]^{-Q}_{-Q}  & &\OX } \\
Q=\alpha+n\\
\end{array}
\overrightarrow{r_{3}}
\begin{array}{c}
\xymatrix{& \OX \AW[ldd]^{\!\!\!R+1}_{-1} \AW[rdd]^{-R}_{-R\!\!} & \\ & & \\ \O \AW[rr]^{-1}_{R-1} & & \OX } \\
R=Q+1=\alpha+(n+1)\\
\end{array}.
\end{equation*}
By induction, two diagrams given by labels $\alpha_{1}$ and $\alpha_{2}$ are connected by a sequence of odd reflections precisely when $\alpha_{1}-\alpha_{2} \in \mathbb{Z}$.  Hence, $S(1,2,\alpha)$ is regular Kac-Moody if and only
if $\alpha$ is not an integer.

We have found all regular Kac-Moody extensions of $A(1,0)$ by one vertex.
\end{proof}

\begin{lemma}
The regular Kac-Moody extensions of $B(1,1)$ to three vertices that are not extensions of $A(1,0)$ are the
following: $A(2,2)^{(4)}$, $D(2,1)^{(2)}$.
\end{lemma}
\begin{proof} We consider each case for attaching a vertex to a $B(1,1)$  diagram. \\
\textbf{Case 1:}
\begin{equation*}
\xymatrix{ \OX_1 \AW[r]^{1}_{-2} & \B_2 \AW[r]^{a}_{b} & \OV_3 }\hspace{1cm}a,b\neq 0
\end{equation*}
Reflecting at $v_{1}$ we have
\begin{equation*}
\xymatrix{ \OX \AW[r]^{1}_{-2} & \B \AW[r]^{a}_{b} & \OV } \text{\hspace{.5cm} } \overrightarrow{r_{1}} \text{
\hspace{.5cm} } \xymatrix{ \OX \AW[r]^{1}_{-2} & \O \AW[r]^{-a}_{b} & \OV}.
\end{equation*}
This implies $a,-a \in \mathbb{Z}_{<0}$, which is a contradiction. \\
\textbf{Case 2:}
\begin{equation*}
\xymatrix{ \OD_1 \ar@<.5ex>[r]^{-2} & \OX_2 \ar@<.5ex>[l]^{a} \ar@<.5ex>[r]^{1} & \B_3 \ar@<.5ex>[l]^{-2}}\hspace{1cm}a\neq 0
\end{equation*}
By reflecting at $v_{2}$ we have
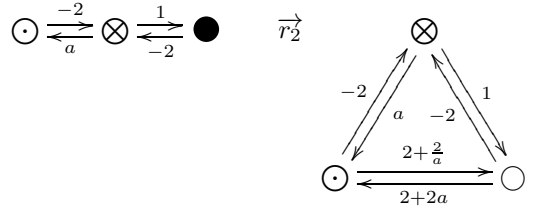
\begin{equation*}
\xymatrix{ \OD \AW[r]^{-2}_{a} & \OX \AW[r]^{1}_{-2} & \B } \text{\hspace{.5cm} } \overrightarrow{r_{2}} \text{  }
\xymatrix{& \OX \AW[ldd]^{a}_{-2} \AW[rdd]^{1}_{-2} & \\ & & \\ \OD \AW[rr]^{2+\frac{2}{a}}_{2+2a} & & \O },
\end{equation*}
which implies $2+2a, 2+\frac{2}{a} \in \mathbb{Z}_{<0}$.  The unique
solution is $a=-1$.  If $v_1$ is $\O$ then this is $A(2,2)^{(4)}$, and if $v_1$ is $\B$ then this is $D(2,1)^{(2)}$.\\
\textbf{Case 3:}
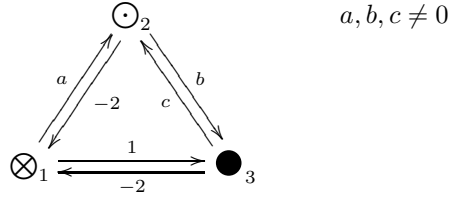
\begin{equation*}
\xymatrix{& \OD_2 \AW[ldd]^{-2}_{a} \AW[rdd]^{b}_{c} & \\ & & \\ \OX_1 \AW[rr]^{1}_{-2} & & \B_3 }\hspace{1cm}a,b,c\neq 0
\end{equation*}
Now $a \neq 0$, $b \in \mathbb{Z}_{<0}$, and $c \in 2\mathbb{Z}_{<0}$.  By reflecting at $v_{1}$ we have
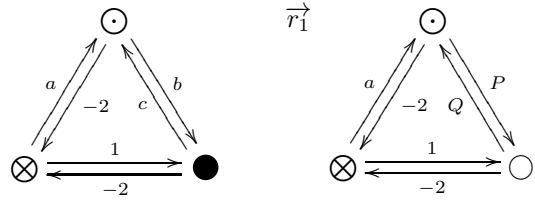
\begin{equation*}
\xymatrix{& \OD \AW[ldd]^{-2}_{a} \AW[rdd]^{b}_{c} & \\ & & \\ \OX \AW[rr]^{1}_{-2} & & \B } \text{ \hspace{.5cm} }
\overrightarrow{r_{1}} \text{  } \xymatrix{& \OD \AW[ldd]^{-2}_{a} \AW[rdd]^{P}_{Q} & \\ & & \\ \OX \AW[rr]^{1}_{-2} & &
\O }
\end{equation*}
where $P = 2-b+\frac{2}{a}$, $Q = 2-c+2a$, and $P,Q \in \mathbb{Z}_{\leq 0}$.  But, the system of equations
\begin{equation*}
b \leq -1, \hspace{1cm} c \leq -2, \hspace{1cm} 2-c+2a \leq 0, \hspace{1cm} 2-b+\frac{2}{a} \leq 0
\end{equation*}
has no real solution.  This is a contradiction.
\end{proof}

We note that the only regular Kac-Moody diagram with three vertices which is not finite or of finite growth
is $Q^{\pm}(m,n,t)$.

\subsection{Preliminaries}

\begin{lemma} \label{lm12.2}\relax
Suppose that $\Gamma$ is a subfinite regular Kac-Moody $n$-vertex chain where the vertex $v_n$ is isotropic and the
vertices $v_i$ for $1<i<n$ are not isotropic.  Then there is a sequence of odd reflections of the vertices $v_i$ with
$3\leq i\leq n$ such that in the reflected diagram $\Gamma'$ the vertex $v_2$ is isotropic.  The vertex $v_1$ is
unchanged and $a_{12}'=a_{12}$.
\end{lemma}

\begin{proof}
If $n=2$, then we are done.  So suppose the lemma holds for such a chain with $n-1$ vertices.  Let $\Gamma$ be a chain
with $n$ vertices satisfying the hypothesis of the lemma.  Now $a_{n,n}=0$ because $v_{n}$ is isotropic, and
$a_{n-1,n-1}=2$ because $v_{n-1}$ is not isotropic.  Since $\Gamma$ is a chain, $a_{n-1,n-2}\neq 0$ and
$a_{n-2,n}=a_{n,n-2}=0$.  Since $\Gamma$ is a subfinite regular Kac-Moody diagram, the subdiagram containing the vertices $v_{n-2},v_{n-1},v_{n}$ is of finite type, which by the three vertex classification implies that
$a_{n-1,n}=-1$ and the vertex $v_{n-1}$ is even.

By reflecting at $v_n$ we obtain a diagram $\Gamma'$ where $v'_{n-1}$ is isotropic.  We observe that $\Gamma'$ is again a chain.  Indeed, $a'_{in}=a'_{ni}=0$ and $a'_{ij}=a_{ij}$ for $i<n-1$, $j<n$, because $a_{in}=a_{ni}=0$ for $i\neq n-1$. Also, $a'_{n-1,i}=a_{n,n-1} a_{n-1,i}=0$ for $i < n-2$. We can apply the induction hypothesis to $\Gamma'-\{v_n\}$ which gives us the desired sequence of odd reflections of $\Gamma'$ and hence of $\Gamma$.
\end{proof}

\begin{lemma} \label{lm12.3}\relax
Let $\Gamma$ be a subfinite regular Kac-Moody diagram.  Suppose $\Gamma'$ is a connected $m$-vertex subdiagram of
$\Gamma$ such that $\Gamma'$ does not contain an isotropic vertex.  Then there is a sequence of odd reflections $R$ of
$\Gamma$ such that $R(\Gamma') = \Gamma'$, and there is a connected $(m+1)$-vertex subfinite regular Kac-Moody subdiagram of $R(\Gamma)$ which contains $\Gamma'=R(\Gamma')$ as a subdiagram and an isotropic vertex.
\end{lemma}
\begin{proof}
First note that $\Gamma$ contains an isotropic vertex because it is subfinite.  So we may take a minimal subdiagram $\Gamma''$ containing $\Gamma'$ and an isotropic vertex.  Then by the minimality of $\Gamma''$ the set of vertices in $\Gamma''-\Gamma'$ form a chain with an isotropic vertex, and only one of these vertices is connected to the subdiagram $\Gamma'$.  Denote this vertex $v_{2}$ and the rest of the vertices in the chain $\Gamma''-\Gamma'$ by $\{v_3,...,v_n\}$ so that $v_k$ is connected to $v_{k-1}$ and $v_{k+1}$.  Then $v_n$ is isotropic by minimality.  Choose any vertex in $\Gamma'$ that is connected to $v_2$ and denote it $v_1$.

Now apply the Lemma~\ref{lm12.2} to the subdiagram given by $\{v_1,v_2,...,v_n\}$ to obtain a sequence of odd reflections $R$ such that $R(v_2)$ is isotropic.  By minimality of $\Gamma''$ each vertex $v_k$ for $k\geq 3$ is not connected to $\Gamma'$, so each reflection does not change the subdiagram $\Gamma'$, ie. $a'_{ij}=a_{ij}$ if $v_i,v_j \in \Gamma'$.  Then the subdiagram $R(\Gamma' \cup \{v_2\})$ of $R(\Gamma)$ is a subfinite regular Kac-Moody $m+1$-vertex diagram containing the diagram $\Gamma'$ and an isotropic vertex.
\end{proof}

\begin{lemma}\label{52} If $\Gamma$ is a subfinite regular Kac-Moody diagram with $n \geq 4$ vertices, then any odd non-isotropic vertex of $\Gamma$ has degree one.
\end{lemma}

\begin{proof}
First note that connected finite type $3$-vertex regular Kac-Moody diagrams satisfy the condition that odd non-isotropic vertices have degree one.  Suppose that the vertex $v_2$ is odd non-isotropic with degree greater than or equal to two.  Let $\Gamma'=\{v_1,v_2,v_3\}$, where $a_{12},a_{21},a_{23},a_{32} \neq 0$.  By the 3-vertex classification, $\Gamma'$ does not contain an isotropic vertex since $\Gamma$ is subfinite.  By Lemma~\ref{lm12.3}, we are reduced to the case when $\Gamma$ has four vertices.

If $v_4$ is connected to an odd non-isotropic vertex $v_i$, then $v_i$ has degree two in a $3$-vertex subdiagram $\{v_4,v_i,v_j\}$, which contradicts the assumption that $\Gamma$ is subfinite.  Thus, $a_{24}=a_{42}=0$.  We may assume $a_{14},a_{41}\neq 0$, which implies that $v_1$ is even. By
reflecting at $v_4$ we obtain a diagram in which $R(v_1)$ is isotropic, $R(v_2)=v_2$ and $a'_{2i}=a_{2i}$ for $i=1,\dots,4$.  But then the subdiagram $R(\{v_1,v_2,v_3\})$ is not of finite type.  This contradicts the assumption that $\Gamma$ is subfinite.
\end{proof}

\begin{lemma}\label{53} If $\Gamma$ is a subfinite regular Kac-Moody diagram with $n \geq 4$ vertices, and $\Gamma'$ is a 3-vertex subdiagram with an isotropic vertex of degree two in $\Gamma'$ such that the ratio of the vertex is not a negative rational number, then $\Gamma'$ is a $D(2,1,\alpha)$ diagram.
\end{lemma}
\begin{proof}
Since $\Gamma'$ is finite type regular Kac-Moody, this follows from the $3$-vertex classification.
\end{proof}

\begin{corollary} \label{54} If $\Gamma$ is a subfinite regular Kac-Moody diagram with $n \geq 4$ vertices and $\Gamma'$ is a 4-vertex subdiagram containing an isotropic vertex of degree three in $\Gamma'$, then $\Gamma'$ contains a $D(2,1,\alpha)$ diagram containing this vertex.
\end{corollary}

\subsection{Subfinite regular Kac-Moody: 4 vertices}
We introduce the following notation in order to simplify the presentation.  If $\Gamma$ is a diagram and $\{v_i\}_{i\in I}$ is a subset of the vertices of $\Gamma$, then we denote by $\Gamma_{\{i\mid i\in I\}}$ the subdiagram of $\Gamma$ obtained by removing the vertices $\{v_i\}_{i\in I}$. For example, $\Gamma_{1,2}$ is the subdiagram of $\Gamma$ obtained by removing the vertices $v_1$ and $v_2$. Let $\mathcal{F}$ denote the set of all connected finite type regular Kac-Moody diagrams which contain an isotropic vertex.

For each finite type $3$-vertex diagram, we will consider each case for attaching an additional vertex to the diagram.   Let $\Gamma$ denote the corresponding extended diagram.

\begin{lemma}
The subfinite regular Kac-Moody extensions of $D(2,1,\alpha)$ to four vertices are $D(3,1)$, $D(2,2)$, $F_{4}$, $B(2,1)^{(2)}$,
$D(2,1,\alpha)^{(1)}$, $G(3)^{(1)}$, $A(1,3)^{(2)}$, and $A(2,3)^{(2)}$.
\end{lemma}

\begin{proof} We consider each case for attaching a vertex to a $D(2,1,\alpha)$ diagram.  \\
\textbf{Case 1:}
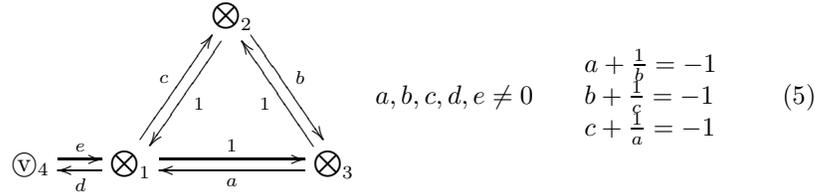
\begin{equation}\label{e1} \begin{array}{c}
\xymatrix{ & & \OX_2 \AW[ldd]^{1}_{c} \AW[rdd]^{b}_{1}  & \\ & & &  \\
  \OV_4 \AW[r]^{e}_{d} &  \OX_1 \AW[rr]^{1}_{a} & & \OX_3} \end{array}
a,b,c,d,e\neq 0\hspace{.5cm}
\begin{array}{lll}
a+\frac{1}{b}=-1 \\
b+\frac{1}{c}=-1 \\
c+\frac{1}{a}=-1 \\
\end{array}
\end{equation}
\D If $v_4$ is $\OX$, then $\Gamma_{2}\in \mathcal{F}$ implies $d=-1$ and $\Gamma_{3}\in \mathcal{F}$ implies $c=-d$.  Then by (\ref{e1}) we have $a=\frac{-1}{2}$, $b=-2$, $c=1$, and $d=-1$.  This is $D(2,2)$. \\
\D If $v_4$ is $\OD$ and $e=-2$, then $\Gamma_2\in\mathcal{F}$ implies $d=-1$ and $\Gamma_{3}\in\mathcal{F}$ implies $c=-d$.  Then by (\ref{e1}) we have $a=\frac{-1}{2}$, $b=-2$, $c=1$, and $d=-1$.  If $v_4$ is $\O$, then this is $B(2,1)^{(1)}$.  If $v_4$ is \B, then this is $A(2,3)^{(2)}$. \\
\D If $v_4$ is $\O$ and $e=-1$, then $\Gamma_{2}\in\mathcal{F}$ implies $d\in \{-1,-2,-3\}$ and $\Gamma_{3}\in\mathcal{F}$ implies $c \in \{-d,\frac{-d}{2},\frac{-d}{3}\}$.  If $d \in \{-1,-2\}$, then $c \in \{-d,\frac{-d}{2}\}$.  This yields three distinct options, and we find that $\Gamma$ is one of the following: $D(3,1)$,  $F_{4}$, $A(1,3)^{(2)}$.  If $d=-3$, then $c \in \{3,\frac{3}{2},1\}$ and by reflecting at $v_{1}$ we obtain
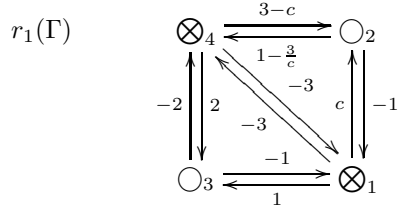
\begin{equation*}r_1(\Gamma)\hspace{1cm}
\xymatrix{ \OX_4 \AW[rr]^{3-c}_{1-\frac{3}{c}} \AW[dd]^{2}_{-2} \AW[rrdd]^{-3}_{-3} & & \O_2 \AW[dd]^{-1}_{c} \\
  & & \\ \O_3 \AW[rr]^{-1}_{1} & & \OX_1  \\ }
\end{equation*}
Then $r_{1}(\Gamma)_{1}\in\mathcal{F}$ implies $\frac{3-c}{2} \leq 0$.  Hence, $c=3$. Then by (\ref{e1}) we have $a=\frac{-1}{4}$, $b=\frac{-4}{3}$. This is $G(3)^{(1)}$. \\ \\
\textbf{Case 2:}
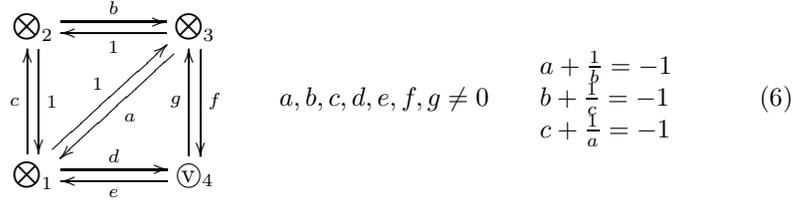
\begin{equation}\label{e2}\begin{array}{c}
\xymatrix{ \OX_2 \AW[rr]^{b}_{1} \AW[dd]^{1}_{c}  & & \OX_3 \AW[dd]^{f}_{g} \\
  & & \\ \OX_1 \AW[rruu]^{1}_{a} \AW[rr]^{d}_{e}  & & \OV_4  } \end{array}\hspace{.5cm}
  a,b,c,d,e,f,g\neq 0\hspace{.5cm}
\begin{array}{lll}
a+\frac{1}{b}=-1 \\
b+\frac{1}{c}=-1 \\
c+\frac{1}{a}=-1 \\
\end{array}
\end{equation}
\D If $v_4$ is $\OX$, then $\Gamma_3\in\mathcal{F}$ implies $d=-c$ and $\Gamma_1\in\mathcal{F}$ implies $f=-1$.  Then by substitution and by (\ref{e2}), we have $d+\frac{f}{a}=-(c+\frac{1}{a})=1$. Now $\Gamma_{2}$ must be a $D(2,1,\alpha)$ diagram, so (\ref{e3}) implies $d+\frac{f}{a}=-1$.  This is a contradiction. \\
\D If $v_4$ is $\O$, $e=-1$ and $g=-2$, then $\Gamma_2\in\mathcal{F}$ implies $a=\frac{3}{2}$, $d=\frac{-1}{2}$. Now $\Gamma_3\in\mathcal{F}$ implies $\frac{d}{c}<0$.  Hence, $c>0$.  Then $c+\frac{1}{a}>0$, which contradicts (\ref{e2}).\\
\D If $v_4$ is $\O$ and $e=g=-1$, then $\Gamma_2\in\mathcal{F}$ implies $d=\frac{f}{a}=\frac{-1}{2}$, $\Gamma_1\in\mathcal{F}$ implies $f<0$, and $\Gamma_3\in\mathcal{F}$ implies $\frac{d}{c}<0$.  Since $d,f<0$ we have $a,c>0$.  But then $c+\frac{1}{a}>0$, which contradicts (\ref{e2}).\\ \\
\textbf{Case 3:}
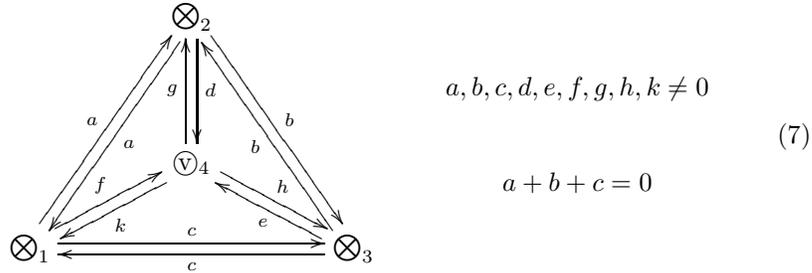
\begin{equation}\label{e8}\begin{array}{c}
\xymatrix{ & & \OX_2 \AW[llddd]^{a}_{a} \AW[rrddd]^{b}_{b} & & \\ & & & & \\
& & \OV_4 \AW[uu]^{g}_{d} \AW[lld]^{k}_{f} \AW[rrd]^{h}_{e} & & \\
\OX_1 \AW[rrrr]^{c}_{c} & & & & \OX_3 }\end{array}\hspace{.5cm} \begin{array}{c} a,b,c,d,e,f,g,h,k\neq 0\\ \\ \\a+b+c=0\end{array}
\end{equation}
Suppose $v_4$ is $\O$.  Then $\frac{d}{b}, \frac{a}{d}, \frac{c}{e}, \frac{e}{b}, \frac{f}{a}, \frac{c}{f}, \in \mathbb{Q}_{<0}$ implies $\frac{a+c}{b}>0$, which contradicts (\ref{e8}). Hence, $v_4$ is $\OX$.  Thus all subdiagrams are $D(2,1,\alpha)$ diagrams.  So we have $g=d$, $h=e$, and $k=f$.  We also have $a+b+c=0$, $c+d+e=0$, $a+e+f=0$, and $b+d+f=0$.  It follows that $d=a$, $e=b$ and $f=c$.  Hence, this is $D(2,1,\alpha)^{(1)}$.
\end{proof}
Now that we have found all subfinite regular Kac-Moody extensions of $D(2,1,\alpha)$ by one vertex, we may restrict our attention to diagrams that do not contain a $D(2,1,\alpha)$ subdiagram.  In particular, the ratio of an isotropic vertex must be a negative rational number by Lemma~\ref{53}, and an isotropic vertex has at most degree two by Corollary~\ref{54}.

\begin{lemma}
The subfinite regular Kac-Moody 4-vertex extensions of $C(3)$ that are not extensions of $D(2,1,\alpha)$ are the
following: $C(4)$, $C(3)^{(1)}$, $B(1,2)^{(1)}$, and $A(1,4)^{(2)}$.
\end{lemma}

\begin{proof}
Since an isotropic vertex has degree at most two, we are reduced to the following case.
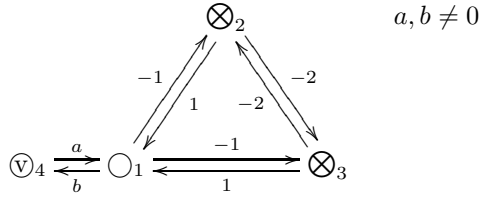
\begin{equation*}
\xymatrix{ & & \OX_2 \AW[ldd]^{1}_{-1} \AW[rdd]^{-2}_{-2}  & \\ & & &  \\
    \OV_4 \AW[r]^{a}_{b} &  \O_1 \AW[rr]^{-1}_{1} & & \OX_3 }\hspace{.5cm}a,b\neq 0
\end{equation*}
\D If $v_4$ is $\OX$, then $\Gamma_{2}\in\mathcal{F}$ implies $b=-1$.  This is $D(2,2)$. \\
\D If $v_4$ is $\B$, then $\Gamma_{2}\in\mathcal{F}$ implies $a=-2$, $b=-1$. This is $B(1,2)^{(1)}$. \\
\D If $v_4$ is $\O$, then $\Gamma_{2}\in\mathcal{F}$ implies $a=-2$, $b=-1$.  This is $A(1,4)^{(2)}$. \\
\D If $v_4$ is $\O$, then $\Gamma_{2}\in\mathcal{F}$ implies $a=-1$, $b \in \{-1,-2,-3\}$.  If $b=-1$, then this is $C(4)$.  If $b=-2$, then this is $C(3)^{(1)}$. If $b=-3$, then by reflecting at $v_2$ and then at $v_1$ we obtain
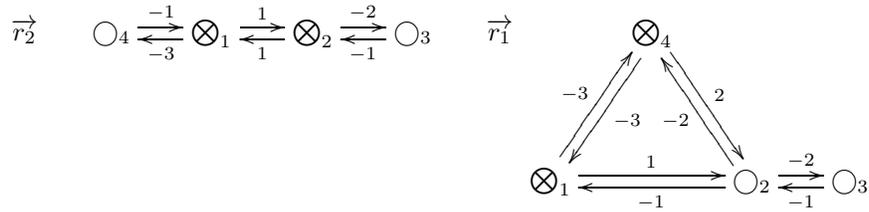
\begin{equation*}
\text{ $\overrightarrow{r_{2}}$ } \hspace{.5cm} \xymatrix{ \O_4 \AW[r]^{-1}_{-3} & \OX_1 \AW[r]^{1}_{1} & \OX_2
\AW[r]^{-2}_{-1} & \O_3 }\hspace{.5cm}
\text{ $\overrightarrow{r_{1}}$ }
\xymatrix{ & \OX_4 \AW[ldd]^{-3}_{-3} \AW[rdd]^{2}_{-2} & & \\ & & & \\
    \OX_1 \AW[rr]^{1}_{-1} & & \O_2 \AW[r]^{-2}_{-1} & \O_3 }
\end{equation*}
But, $r_1(r_2(\Gamma))_1\not\in\mathcal{F}$.  Hence, $b \neq -3$.
\end{proof}

\begin{lemma}
All subfinite regular Kac-Moody extensions of $G(3)$ are extensions of $D(2,1,\alpha)$.
\end{lemma}
\begin{proof}
Since an isotropic vertex has degree at most two, we are reduced to the following case.
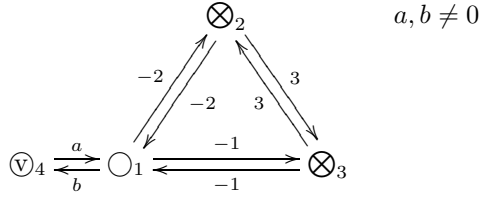
\begin{equation*}
\xymatrix{ & & \OX_2 \AW[ldd]^{-2}_{-2} \AW[rdd]^{3}_{3}  & \\ & & &  \\
    \OV_4 \AW[r]^{a}_{b} &  \O_1 \AW[rr]^{-1}_{-1} & & \OX_3 }\hspace{.5cm}a,b\neq 0
\end{equation*}
But, $\Gamma_3\not\in\mathcal{F}$.
\end{proof}

\begin{lemma}
The subfinite regular Kac-Moody extensions of $B(1,2)$ that are not extensions of $D(2,1,\alpha)$, $C(3)$ or $G(3)$ are
the following: $B(1,3)$, $B(2,2)$, $D(2,2)^{(2)}$, and $A(2,4)^{(4)}$.
\end{lemma}

\begin{proof}
Since $\OX$ has at most degree two and $\B$ has at most degree one, we are reduced to the following case.
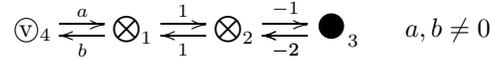
\begin{equation*}
\xymatrix{ \OV_4 \AW[r]^{a}_{b} & \OX_1 \AW[r]^{1}_{1} & \OX_2 \AW[r]^{-1}_{-2} & \B_3 \ar@<.5ex>[l]^{-2} }
\hspace{.5cm}a,b\neq 0
\end{equation*}
\D If $v_4$ is $\OX$, then $\Gamma_{3}\in\mathcal{F}$ implies $b=-1$.  This is $B(2,2)$. \\
\D If $v_4$ is $\B$, then $\Gamma_{3}\in\mathcal{F}$ implies $a=-2$, $b=-1$. This is $D(2,2)^{(2)}$. \\
\D If $v_4$ is $\O$, then $\Gamma_{3}\in\mathcal{F}$ implies $a\in\{-1,-2\}$ and $b\in\{-1,-2,-3\}$.  By assumption, $\Gamma_3$ is not $C(3)$ or $G(3)$, which implies $b=-1$.  If $a=-1$, then this is $B(1,3)$.  If $a=-2$, then this is $A(2,4)^{(4)}$.
\end{proof}

\begin{lemma}
The subfinite regular Kac-Moody extensions of $B(2,1)$ that are not extensions of $D(2,1,\alpha)$, $C(3)$ or $G(3)$ are
the following: $B(3,1)$, $B(2,2)$, $A(2,4)^{(4)}$, and $D(1,3)^{(2)}$.
\end{lemma}

\begin{proof} We consider each case for attaching a vertex to a $B(2,1)$ diagram.  \\
\textbf{Case 1:}
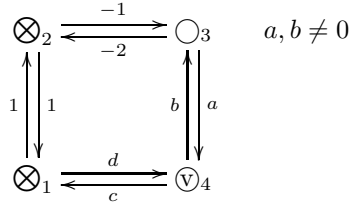
\begin{equation*}
\xymatrix{ \OX_2 \AW[dd]^{1}_{1} \AW[rr]^{-1}_{-2}  & & \O_3 \AW[dd]^{a}_{b} \\  & & \\ \OX_1 \AW[rr]^{d}_{c} & & \OV_4 }\hspace{.5cm}a,b\neq 0
\end{equation*}
Then $\Gamma_1\not\in\mathcal{F}$ even if $c,d=0$.\\
\textbf{Case 2:}
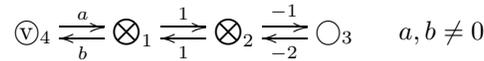
\begin{equation*}
\xymatrix{ \OV_4 \AW[r]^{a}_{b} & \OX_1 \AW[r]^{1}_{1} & \OX_2 \AW[r]^{-1}_{-2} & \O_3 }\hspace{.5cm}a,b\neq 0
\end{equation*}
\D If $v_4$ is $\OX$, then $\Gamma_{3}\in\mathcal{F}$ implies $b=-1$. This is $B(2,2)$. \\
\D If $v_4$ is $\B$, then $\Gamma_{3}\in\mathcal{F}$ implies $a=-2$, $b=-1$. This is $A(2,4)^{(4)}$. \\
\D If $v_4$ is $\O$, then $\Gamma_{3}\in\mathcal{F}$ implies $a\in\{-1,-2\}$ and $b \in \{-1,-2,-3\}$.  By assumption, $\Gamma_3$ is not $C(3)$ or $G(3)$, which implies $b=-1$. If $a=-1$, then this is $B(3,1)$.  If $a=-2$, then this is $D(1,3)^{(2)}$. \\
\end{proof}

\begin{lemma}
The subfinite regular Kac-Moody extensions of $A(0,2)$ that are not extensions of $D(2,1,\alpha)$, $C(3)$ or $G(3)$ are the following: $A(0,3)$, $A(1,2)$, $B(1,3)$, $B(3,1)$, $A(0,2)^{(1)}$, and $q(4)^{(2)}$.
\end{lemma}

\begin{proof} We consider each case for attaching a vertex to an $A(0,2)$ diagram.  \\
\textbf{Case 1:}
\begin{equation*}
\xymatrix{ \OV_4 \AW[r]^{a}_{b} & \OX_1 \AW[r]^{1}_{1} & \OX_2 \AW[r]^{-1}_{-1} & \O_3 }\hspace{.5cm}a,b\neq 0
\end{equation*}
\D If $v_4$ is $\OX$, then $\Gamma_{3}\in\mathcal{F}$ implies $b=-1$. This is $A(1,2)$. \\
\D If $v_4$ is \B, then $\Gamma_{3}\in\mathcal{F}$ implies $a=-2$, $b=-1$.  This is $B(1,3)$. \\
\D If $v_4$ is $\O$, then $\Gamma_{3}\in\mathcal{F}$ implies $a\in\{-1,-2\}$ and $b \in \{-1,-2,-3\}$.  By assumption, $\Gamma_3$ is not $C(3)$ or $G(3)$, which implies $b=-1$. If $a=-1$, then this is $A(0,3)$. If $a=-2$, then this is $B(3,1)$. \\ \\
\textbf{Case 2:}
\begin{equation*}
\xymatrix{  \OX_1   \AW[r]^{1}_{1} & \OX_2 \AW[r]^{-1}_{-1}  & \O_3 \AW[r]^{b}_{a} & \OV_4  }\hspace{.5cm}a,b\neq 0
\end{equation*}
\D If $v_4$ is $\OX$, then $\Gamma_{1}\in\mathcal{F}$ implies $b=-1$.  This is $A(1,2)$. \\
\D If $v_4$ is \B, then $\Gamma_{1}\in\mathcal{F}$ implies $a=-2$, $b=-1$.  This is $B(1,3)$. \\
\D If $v_4$ is $\O$, then $\Gamma_{1}\in\mathcal{F}$ implies $a\in\{-1,-2\}$ and $b\in\{-1,-2,-3\}$.  By assumption, $\Gamma_1$ is not $C(3)$ or $G(3)$, which implies $b=-1$.  If $a=-1$, then this is $A(0,3)$. If $a=-2$, then this is $B(3,1)$. \\ \\
\textbf{Case 3:}
\begin{equation*}
\xymatrix{ \OX_2 \AW[dd]^{1}_{1} \AW[rr]^{-1}_{-1}  & & \O_3 \AW[dd]^{a}_{b} \\  & & \\ \OX_1 \AW[rr]^{d}_{c} & & \OV_4 }\hspace{.5cm}a,b,c,d\neq 0
\end{equation*}
\D If $v_4$ is $\O$, then $\Gamma_2\in\mathcal{F}$ implies $c=-1$ and $b\in\{-1,-2,-3\}$.   By assumption, $\Gamma_2$ is not $C(3)$ or $G(3)$, which implies $b=-1$.  Since we assume that $\Gamma_3$ is not $C(3)$ or $G(3)$, we have that $\Gamma_3\in\mathcal{F}$ implies $d=-1$.  Now $\Gamma_1\in\mathcal{F}$ implies $a\in\{-1,-2,-3\}$ and since we assume that it is not $C(3)$ or $G(3)$, we have $a=-1$.  This is $A(0,2)^{(1)}$. \\
\D If $v_4$ is $\OX$, then $\Gamma_1\in\mathcal{F}$ implies $a=-1$.  Since we assume that $\Gamma_3$ is not $D(2,1,\alpha)$, we have that $\Gamma_3\in\mathcal{F}$ implies $d=-1$. Finally, $\Gamma_2\in\mathcal{F}$ and not equal to $C(3)$ or $G(3)$ implies $\frac{b}{c}=-1$.  This is $q(4)^{(2)}$.
\end{proof}

\begin{lemma}
The subfinite regular Kac-Moody extensions of $A(1,1)$ that are not extensions of $D(2,1,\alpha)$, $C(3)$ or $G(3)$ are the following: $A(1,2)$, $B(2,2)$, $A(1,1)^{(1)}$, and $q(4)^{(2)}$.
\end{lemma}

\begin{proof} We consider each case for attaching a vertex to an $A(1,1)$ diagram.  \\
\textbf{Case 1:}
\begin{equation*}
\xymatrix{  \OX_1   \AW[r]^{1}_{1} & \OX_2 \AW[r]^{-1}_{1} & \OX_3 \AW[r]^{b}_{a} & \OV_4 }\hspace{.5cm}a,b\neq 0
\end{equation*}
\D If $v_4$ is $\OX$, then $\Gamma_1\in\mathcal{F}$ implies $b-1$.  This is $A(1,2)$. \\
\D If $v_4$ is \B, then $\Gamma_1\in\mathcal{F}$ implies $a=-2$, $b=-1$.  This is $B(2,2)$. \\
\D If $v_4$ is $\O$, then $\Gamma_1\in\mathcal{F}$ implies $a\in\{-1,-2\}$ and $b\in\{-1,-2,-3\}$.  By assumption, $\Gamma_1$ is not $C(3)$ or $G(3)$, which implies $b=-1$.  If $a=-1$, then this is $A(1,2)$.  If $a=-2$, then this is $B(2,2)$. \\ \\
\textbf{Case 2:}
\begin{equation*}
\xymatrix{ \OX_1 \AW[dd]^{1}_{1} \AW[rr]^{-1}_{1} & & \OX_2 \AW[dd]^{a}_{b} \\  & & \\ \OX_3 \AW[rr]^{d}_{c} & & \OV_4  }\hspace{.5cm}a,b,c,d\neq 0
\end{equation*}
\D If $v_4$ is $\OX$, then $\Gamma_1\in\mathcal{F}$ implies $a=-1$, $\Gamma_3\in\mathcal{F}$ implies $d=-1$, and $\Gamma_2\in\mathcal{F}$ implies $\frac{b}{c}=-1$.  This is $A(1,1)^{(1)}$. \\
\D If $v_4$ is $\O$, then $\Gamma_1\in\mathcal{F}$ implies $b=c=-1$. Then $\Gamma_2,\Gamma_3\in\mathcal{F}$ and not $C(3)$ or $G(3)$ imply $a=d=-1$.  This is $q(4)^{(2)}$.
\end{proof}

This completes the classification of connected subfinite regular Kac-Moody diagrams with four vertices.  We observe that all are either of finite type or have finite growth.

\subsection{Subfinite regular Kac-Moody: 5 vertices}

For each finite type $4$-vertex diagram, we will consider each case for attaching a vertex to the diagram.   Let $v_5$ denote the additional vertex, and let $\Gamma$ denote the corresponding extended diagram.  Recall that $a_{ij}$ denotes the label of the arrow from the vertex $v_i$ to the vertex $v_j$.  Also, $a_{ij}=0$ if and only if $a_{ji}=0$.

\begin{lemma}
The subfinite regular Kac-Moody extensions of $D(2,2)$ are the following: $D(2,3)$, $D(3,2)$, $B(2,2)^{(1)}$,
$A(4,3)^{(2)}$, $A(3,3)^{(2)}$, $D(2,2)^{(2)}$.
\end{lemma}
\begin{proof}
\begin{equation*}
\xymatrix{ & & \OX_2 \AW[ldd]^{1}_{1} \AW[rdd]^{-2}_{-2} & \\ & & & \\
     \OX_4 \AW[r]^{-1}_{-1} & \OX_1 \AW[rr]^{1}_{1} & & \OX_3 }
\end{equation*}
First observe that $\Gamma_4\in\mathcal{F}$ implies $a_{52}=a_{53}=0$.\\
\D If $v_5$ is \B, then $\Gamma_2\in\mathcal{F}$ implies $a_{51}=0$, $a_{54}=-2$, $a_{45}=-1$. This is $B(2,2)^{(1)}$.\\
\D If $v_5$ is $\O$, then $\Gamma_2\in\mathcal{F}$ implies $a_{51}=0$ and $a_{45},a_{54}\in\{-1,-2\}$.
If $a_{54}=-2$ and $a_{45}=-1$, then this is $A(4,3)^{(2)}$. If $a_{54}=-1$ and $a_{45}=-2$, then this is $D(2,2)^{(2)}$. If $a_{54}=a_{45}=-1$, then this is $D(2,3)$. \\
\D If $v_5$ is $\OX$, then $\Gamma_2\in\mathcal{F}$ implies that either
$a_{51}=a_{15}=0$, $a_{45}=-1$, $a_{54}=1$ and this is $D(3,2)$, or
$a_{51}=-1$, $a_{15}=1$, $a_{45}=a_{54}=-2$ and this is $A(3,3)^{(2)}$.
\end{proof}

\begin{lemma}
The subfinite regular Kac-Moody extensions of $D(3,1)$ are the following: $D(2,3)$, $D(4,1)$, $B(3,1)^{(1)}$,
$A(2,5)^{(2)}$, $A(3,3)^{(2)}$, $A(5,1)^{(2)}$.
\end{lemma}
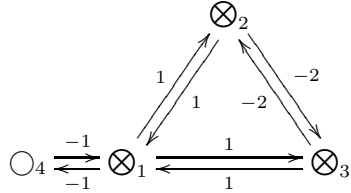
\begin{proof}
\begin{equation*}
\xymatrix{ & & \OX_2 \AW[ldd]^{1}_{1} \AW[rdd]^{-2}_{-2} & \\ & & & \\
     \O_4 \AW[r]^{-1}_{-1} & \OX_1 \AW[rr]^{1}_{1} & & \OX_3 }
\end{equation*}
First observe that $\Gamma_4\in\mathcal{F}$ implies $a_{52}=a_{53}=0$.\\
\D If $v_5$ is \B, then $\Gamma_3\in\mathcal{F}$ implies $a_{51}=0$, $a_{54}=-2$, $a_{45}=-1$. This is $A(2,5)^{(2)}$.\\
\D If $v_5$ is $\O$, then $\Gamma_3\in\mathcal{F}$ implies that $\Gamma$ satisfies the following conditions.  If $a_{51}\neq 0$, then $a_{51}=a_{15}=-1$, $a_{54}=a_{45}=0$, and this is $A(3,3)^{(2)}$.  If $a_{51}=0$, then $a_{54},a_{45}\in\{-1,-2\}$. If $a_{51}=0$, $a_{54}=-2$ and $a_{45}=-1$, then this is $B(3,1)^{(1)}$.  If $a_{51}=0$, $a_{54}=-1$ and $a_{45}=-2$, then this is $A(5,1)^{(2)}$.  If $a_{51}=0$, $a_{54}=-1$ and $a_{45}=-1$, then this is $D(4,1)$.\\
\D If $v_5$ is $\OX$, then $\Gamma_3\in\mathcal{F}$ implies $a_{51}=0$, $a_{45}=-1$, $a_{54}=1$.  This is $D(2,3)$.
\end{proof}

\begin{lemma}
The subfinite regular Kac-Moody extensions of $C(4)$ are the following: $C(5)$, $D(2,3)$, $B(1,3)^{(1)}$, $D(1,3)^{(1)}$,
$A(5,1)^{(2)}$, $A(6,1)^{(2)}$.
\end{lemma}
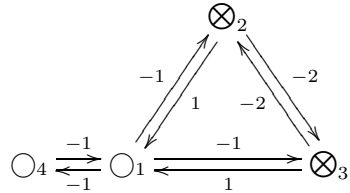
\begin{proof}
\begin{equation*}
\xymatrix{ & & \OX_2 \AW[ldd]^{1}_{-1} \AW[rdd]^{-2}_{-2} & \\ & & & \\
     \O_4 \AW[r]^{-1}_{-1} & \O_1 \AW[rr]^{-1}_{1} & & \OX_3 }
\end{equation*}
First observe that $\Gamma_4\in\mathcal{F}$ implies $a_{52}=a_{53}=0$. \\
\D If $v_5$ is \B, then $\Gamma_3\in\mathcal{F}$ implies
$a_{51}=0$, $a_{54}=-2$, $a_{45}=-1$.  This is $B(1,3)^{(1)}$. \\
\D If $v_5$ is $\O$, then $\Gamma_3\in\mathcal{F}$ implies that $\Gamma$ satisfies the following conditions.  If $a_{51}\neq 0$, then
$a_{51}=a_{15}=-1$, $a_{54}=a_{45}=0$, and this is $A(5,1)^{(2)}$.  If $a_{51}=0$, then $a_{54},a_{45}\in\{-1,-2\}$.
If $a_{51}=0$, $a_{54}=-2$ and $a_{45}=-1$, then this is $A(6,1)^{(2)}$.  If $a_{51}=a_{15}=0$, $a_{54}=-1$ and $a_{45}=-2$, then this is $D(1,3)^{(1)}$.  If $a_{51}=a_{15}=0$, $a_{54}=-1$ and $a_{45}=-1$, then this is $C(5)$. \\
\D If $v_5$ is $\OX$, then $\Gamma_3\in\mathcal{F}$ implies $a_{51}=0$, $a_{54}=1$, $a_{45}=-1$.  This is $D(2,3)$.
\end{proof}

\begin{lemma}
The only subfinite regular Kac-Moody extensions of $F(4)$ is $F(4)^{(1)}$.
\end{lemma}
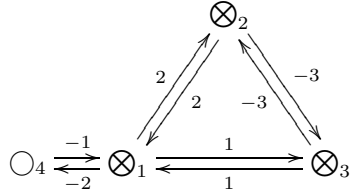
\begin{proof}
\begin{equation*}
\xymatrix{ & & \OX_2 \AW[ldd]^{2}_{2} \AW[rdd]^{-3}_{-3} & \\ & & & \\
     \O_4 \AW[r]^{-1}_{-2} & \OX_1 \AW[rr]^{1}_{1} & &  \OX_3 }
\end{equation*}
First observe that $\Gamma_4\in\mathcal{F}$ implies $a_{52}=a_{53}=0$.  Then $\Gamma_2\in\mathcal{F}$ implies $a_{51}=0$, $a_{54}=a_{45}=-1$.  This is $F(4)^{(1)}$.
\end{proof}

Let $\mathcal{K}=\{\Gamma'\in\mathcal{F}\mid R(\Gamma') \text{ is a chain for every sequence $R$ of odd reflections }\}$.

\begin{lemma} The subfinite regular Kac-Moody extensions of $A(0,3)$, $A(1,2)$, $B(1,3)$,
$B(3,1)$ and $B(2,2)$, that are not extensions of $C(4)$, $D(3,1)$, $D(2,2)$ or $F_{4}$, are the following:
$A(0,4)$, $A(1,3)$, $A(2,2)$, $B(1,4)$, $B(2,3)$, $B(3,2)$, $B(4,1)$, $A(0,3)^{(1)}$,
$A(1,2)^{(1)}$, $A(4,4)^{(2)}$, $A(6,2)^{(2)}$, $D(2,3)^{(2)}$, $D(3,2)^{(2)}$, $D(4,1)^{(2)}$, $q(5)^{(2)}$.
\end{lemma}

\begin{proof} Note that every connected $4$-vertex subdiagram of $\Gamma$ containing an isotropic vertex is an element of $\mathcal{K}$.  We find the extensions for each of the diagrams: $A(0,3)$, $A(1,2)$, $B(1,3)$,
$B(3,1)$ and $B(2,2)$.
    \begin{equation*}B(3,1) \hspace{2cm}
\xymatrix{ \OX_1 \AW[r]^{1}_{1} & \OX_2 \AW[r]^{-1}_{-1} & \O_3 \AW[r]^{-1}_{-2} & \O_4 }
    \end{equation*}
First observe that since $\Gamma_1,\Gamma_4\in\mathcal{K}$ we have $a_{52}=a_{53}=a_{54}=0$.\\
\D If $v_5$ is $\OX$, then $\Gamma_4\in\mathcal{K}$ implies $a_{15}=-1$, $a_{15}=1$. This is $B(3,2)$.\\
\D If $v_5$ is \B, then $\Gamma_4\in\mathcal{K}$ implies $a_{15}=-1$, $a_{51}=-2$.  This is $A(6,2)^{(2)}$.\\
\D If $v_5$ is $\O$, then $\Gamma_4\in\mathcal{K}$ implies $a_{15}=-1$, $a_{51}\in\{-1,-2\}$. If $a_{51}=-1$, then this is $B(4,1)$.  If $a_{51}=-2$, then this is $D(4,1)^{(2)}$.\\
    \begin{equation*}B(1,3) \hspace{2cm}
\xymatrix{ \OX_1 \AW[r]^{1}_{1} & \OX_2 \AW[r]^{-1}_{-1} & \O_3 \AW[r]^{-1}_{-2} & \B_4 }
    \end{equation*}
First observe that since $\Gamma_1,\Gamma_4\in\mathcal{K}$ we have $a_{52}=a_{53}=a_{54}=0$.\\
\D If $v_5$ is $\OX$, then $\Gamma_4\in\mathcal{K}$ implies $a_{15}=-1$, $a_{15}=1$.  This is $B(2,3)$.\\
\D If $v_5$ is \B, then $\Gamma_4\in\mathcal{K}$ implies $a_{15}=-1$, $a_{51}=-2$.  This is $D(2,3)^{(2)}$.\\
\D If $v_5$ is $\O$, then $\Gamma_4\in\mathcal{K}$ implies  $a_{15}=-1$, $a_{51}\in\{-1,-2\}$. If $a_{51}=-1$, then this is $B(1,4)$. If $a_{51}=-2$, then this is $A(6,2)^{(2)}$.\\
    \begin{equation*}B(2,2) \hspace{2cm}
\xymatrix{ \OX_1 \AW[r]^{1}_{1} & \OX_2 \AW[r]^{-1}_{1} & \OX_3 \AW[r]^{-1}_{-2} & \B_4 }
    \end{equation*}
First observe that since $\Gamma_1,\Gamma_4\in\mathcal{K}$ we have $a_{52}=a_{53}=a_{54}=0$.\\
\D If $v_5$ is $\OX$, then $\Gamma_4\in\mathcal{K}$ implies $a_{15}=-1$, $a_{15}=1$.  This is $B(2,3)$.\\
\D If $v_5$ is \B, then $\Gamma_4\in\mathcal{K}$ implies $a_{15}=-1$, $a_{51}=-2$.  This is $A(4,4)^{(2)}$.\\
\D If $v_5$ is $\O$, then $\Gamma_4\in\mathcal{K}$ implies $a_{15}=-1$, $a_{51}\in\{-1,-2\}$. If $a_{51}=-1$, then this is $B(3,2)$. If $a_{51}=-2$, then this is $D(3,2)^{(2)}$. \\

Finally, we restrict our attention to diagrams satisfying: a connected 4-vertex subdiagram containing an isotropic vertex is a diagram for $A(0,3)$ or $A(1,2)$.
\begin{equation*}A(1,2) \hspace{2cm}
\xymatrix{ \OX_1 \AW[r]^{1}_{1} & \OX_2 \AW[r]^{-1}_{1} & \OX_3 \AW[r]^{-1}_{1} & \OX_4 }
    \end{equation*}
First observe that since $\Gamma_1,\Gamma_4\in\mathcal{K}$ we have $a_{52}=a_{53}=0$. There are two distinct cases: $a_{54}=0$ and $a_{51},a_{54} \neq 0$.\\
\D If $v_5$ is $\OX$ and $a_{54}=0$, then $a_{51}=1$, $a_{15}=-1$.  This is $A(2,2)$.\\
\D If $v_5$ is $\O$ and $a_{54}=0$, then $a_{51},a_{15}=-1$.  This is $A(1,3)$. \\
\D If $v_5$ is $\OX$ and $a_{54},a_{51}\neq 0$, then $a_{51}=1$, $a_{15},a_{54},a_{45}=-1$.  This is $q(5)^{(2)}$.\\
\D If $v_5$ is $\O$ and $a_{54},a_{51}\neq 0$, then $a_{51},a_{15},a_{54},a_{45}=-1$. This is $A(1,2)^{(1)}$.
\begin{equation*}A(0,3) \hspace{2cm}
\xymatrix{ \O_1 \AW[r]^{-1}_{1} & \OX_2 \AW[r]^{-1}_{1} & \OX_3 \AW[r]^{-1}_{-1} & \O_4 }
\end{equation*}
Now we assume that a connected 4-vertex subdiagram is a diagram for $A(0,3)$.  First observe that since $\Gamma_1,\Gamma_4\in\mathcal{K}$ we have $a_{52},a_{53}=0$. There are two distinct cases: $a_{54}=0$ and $a_{51},a_{54} \neq 0$. \\
\D If $v_5$ is $\OX$ and $a_{54}=0$, then $a_{51}=1$, $a_{15}=-1$. This is $A(1,3)$.\\ \D If $v_5$ is $\O$ and $a_{54}=0$, then $a_{51},a_{15}=-1$.  This is $A(0,4)$. \\
\D If $v_5$ is $\OX$ and $a_{54},a_{51}\neq 0$, then $a_{51}=1$, $a_{15},a_{54},a_{45}=-1$.  This is $q(5)^{(2)}$.\\
\D If $v_5$ is $\O$ and $a_{54},a_{51}\neq 0$, then $a_{51},a_{15},a_{54},a_{45}=-1$.  This is $A(0,3)^{(1)}$.
\end{proof}

This completes the classification of connected subfinite regular Kac-Moody diagrams with five vertices.  We observe that all are either of finite type or have finite growth.

\subsection{Subfinite regular Kac-Moody: $n\geq 6$ vertices}

Now we handle the general case.  We find all connected subfinite regular Kac-Moody diagrams with six or more vertices.  We will show in Theorem~\ref{thm2} that the subfinite regular Kac-Moody diagrams which are not of finite type are not extendable, completing the classification of regular Kac-Moody diagrams.

\begin{remark} \label{rm261}
A finite type diagram $\Gamma'$ with $n\geq 5$ vertices has the following form:
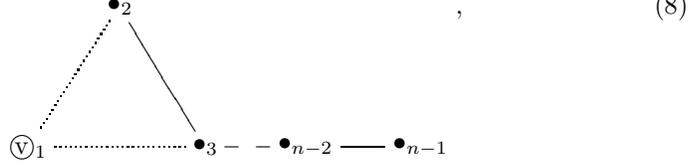
\begin{equation}\label{finite}
\xymatrix{ & \bullet_2 \ar@{-}[rdd] \ar@{..}[ldd] \\   \\
    \OV_1 & &  \bullet_3 \ar@{..}[ll]&  \bullet_{n-2} \ar@{--}[l]  &   \bullet_{n-1} \ar@{-}[l] \\},
\end{equation}
where $\Gamma'_{1}$ is $A(k,l)$ or $A_k$, and the subdiagram $\{v_1,v_2,v_3\}$ is one of the diagrams in the
following table.
\end{remark} \text{ }\\
\begin{center}
\begin{tabular}{|c|c|c|}
 \multicolumn{3}{c}{Table \ref{rm261}} \\
\hline
$\xymatrix{ \bullet_2 \ar@{-}[rrd] \ar@{-}[dd] & & \\
 & & \bullet_3 \\
  \bullet_1 & & \\}$ &
$\xymatrix{ \O_2 \AW[rrd]^{-1}_{} \AW[dd]^{-2}_{-1} & & \\ & & \bullet_3 \\ \O_1 & & \\}$ & $\xymatrix{ \O_2
\AW[rrd]^{-1}_{-1} & & \\ & & \bullet_3 \\ \O_1  \AW[rru]^{-1}_{-1}
& & \\}$ \\
\hline $\xymatrix{ \OX_2 \AW[rrd]^{1}_{} \AW[dd]^{-2}_{-1} & & \\ & & \bullet_3 \\ \O_1 & & \\}$ & $\xymatrix{ \bullet_2
\ar@{-}[rrd] \AW[dd]^{}_{-2} & & \\ & & \bullet_3 \\  \OD_1  & & \\}$  &
$\xymatrix{ \OX_2 \AW[rrd]^{-1}_{-1} \AW[dd]^{2}_{2} & & \\ & & \bullet_3  \\ \OX_1 \AW[rru]^{-1}_{-1} & & \\}$ \\
\hline
\end{tabular}
\end{center} \text{ }
\vspace{.5cm}
\begin{lemma} \label{lm12.1}\relax
Let $\Gamma$ be a subfinite regular Kac-Moody cycle with $n\geq 4$ vertices and let $v_{j}$ be a vertex of $\Gamma$
connected to $v_{i}$ and $v_{k}$. Then either
\begin{equation}\label{eq1}
\begin{array}{l}
v_{j} \text{ is }\O \text{ with }a_{ji}=a_{jk}=-1 \text{; or} \\
v_{j} \text{ is }\OX \text{ with }\frac{a_{ji}}{a_{jk}}=-1. \\
\end{array}
\end{equation}
If $\Gamma$ has an even number of odd roots, then $\mathfrak{g}(A)$ is an $A(k-1,l)^{(1)}$ diagram. If $\Gamma$ has an
odd number of odd roots, then $\mathfrak{g}(A)$ is a $q(n)^{(2)}$ diagram.
\end{lemma}

\begin{proof}
By the above classification, all subfinite regular Kac-Moody 4-vertex and 5-vertex cycles are of type $A(k,l)^{(1)}$ or $q(n)^{(2)}$, and satisfy (\ref{eq1}).  Let $\Gamma$ be a subfinite regular Kac-Moody cycle with $n\geq 6$ vertices and let $v_{j}$ be a vertex which is connected to $v_{i}$ and $v_{k}$. Since $n\geq 6$, $v_j$ is contained in a proper $5$-vertex subdiagram where $v_j$ is the middle vertex of the chain.  By the $5$-vertex classification, we see that for every finite type chain diagram with five vertices, the middle vertex satisfies (\ref{eq1}).  Hence, every vertex of $\Gamma$ satisfies (\ref{eq1}).  If $\Gamma$ has an even number of odd roots, then the corresponding matrix $A$ is symmetrizable and $\mathfrak{g}(A)$ is an $A(k,l)^{(1)}$ diagram. If $\Gamma$ has an odd number of odd roots, then $A$ is non-symmetrizable and $\mathfrak{g}(A)$ is a $q(n)^{(2)}$ diagram.
\end{proof}

\begin{lemma} \label{lm12.4}\relax
If $\Gamma$ contains a proper subdiagram $\Gamma'$ which is a cycle with $n \geq 4$ vertices then $\Gamma$ is not
subfinite regular Kac-Moody.
\end{lemma}

\begin{proof} Suppose $\Gamma$ is a subfinite regular Kac-Moody diagram which contains a cycle with $n \geq 4$ vertices.  By the Lemma~\ref{lm12.3}, we are reduced to the case where $\Gamma=\Gamma' \cup \{v_{n+1}\}$ is subfinite regular Kac-Moody. By the above classification, a subfinite regular Kac-Moody 5-vertex diagram does not contain a 4-vertex cycle subdiagram. So now suppose $\Gamma'$ is a cycle with $n$ vertices where $n \geq 5$ and that $\Gamma=\Gamma' \cup \{v_{n+1}\}$ is subfinite regular Kac-Moody. If $\Gamma'$ contains an isotropic vertex, then by Lemma~\ref{lm12.1} the diagram $\Gamma'$ is either $A(m,n)^{(1)}$ or $q(n)^{(2)}$, which is not of finite type.

Suppose $\Gamma'$ does not contain an isotropic vertex.  Then the additional vertex $v_{n+1}$ of $\Gamma$ is isotropic. Since $\Gamma$ is connected we have the following subdiagram:
    
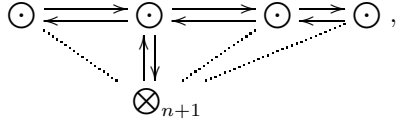
\begin{equation*}
    \xymatrix{ \OD \AW[r]^{}_{} & \OD \AW[d]^{}_{} \AW[r]^{}_{} &
    \OD \AW[r]^{}_{} & \OD  \\
    & \ \quad \OX_{n+1}  \ar@{..}[lu] \ar@{..}[ru] \ar@{..}[rru] & & \\},
    \end{equation*}
where the double lines are necessarily connected, and the dotted lines are possibly connected. But by the 5-vertex
classification, we see that this is not a finite type subdiagram.
\end{proof}

\begin{lemma}\label{gamma}
Let $\Gamma$ be the following diagram
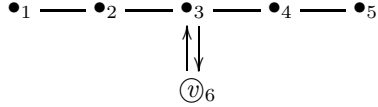
\begin{equation*}
\xymatrix{ \bullet_1 \ar@{-}[r]  & \bullet_2 \ar@{-}[r] & \bullet_3 \ar@{-}[r]& \bullet_4 \ar@{-}[r] & \bullet_5 \\
    & & \OV_6\!\! \AW[u]^{}_{} & & \\}
\end{equation*}
where $\Gamma$ contains an isotropic vertex and $\Gamma_6$ is a diagram for $A(k,l)$ or $A_k$.  Then $\Gamma$ is not a subfinite regular Kac-Moody diagram.
\end{lemma}
\begin{proof}
If $v_6$ is isotropic, then $\Gamma_5\not\in\mathcal{F}$, which is a contradiction.
We may assume that $v_3$ is isotropic, by using odd reflections in the subdiagram $\Gamma_{3,6}$.
Then $\Gamma_1,\Gamma_5\in\mathcal{F}$ implies that the vertices $v_2,v_4,v_6$ are even.  Moreover, $\Gamma_1\in\mathcal{F}$ implies that the ratio of $v_3$ in the subdiagram $\{v_2,v_3,v_6\}$ is $1$, while $\Gamma_5\in\mathcal{F}$ implies that the ratio is $-1$, which is a contradiction.
\end{proof}

\begin{proposition}
If $\Gamma$ is a subfinite regular Kac-Moody diagram with $n\geq 6$ vertices and $\Gamma$ is not a cycle, then $\Gamma$
is of the form
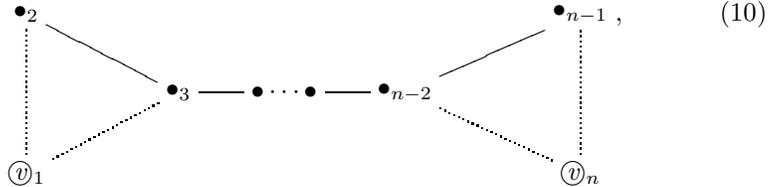
\begin{equation}\label{form}
\xymatrix{
\bullet_2 \ar@{-}[rrd] & & & & & &  \bullet_{n-1} \ar@{-}[lld] \\
& & \bullet_3  \ar@{-}[r] & \bullet \ar@{-}[r]  \cdots \bullet \ar@{-}[l] \ar@{-}[r]  & \bullet_{n-2} & & \\
\OV_1 \ar@{..}[uu] \ar@{..}[rru] & & & & & & \OV_{n} \ar@{..}[uu] \ar@{..}[llu] \\},
\end{equation}
where the subdiagram $\Gamma \setminus\{v_1,v_n\}$ is $A(k,l)$ or $A_k$, and the subdiagrams
$\{v_1,v_2,v_3\}$ and $\{v_n,v_{n-1},v_{n-2}\}$ are diagrams from Table~\ref{rm261}.
\end{proposition}

\begin{proof}  First note that diagrams satisfying (\ref{form}) reflect by odd reflections to diagrams that again satisfy (\ref{form}). So it suffices to prove the proposition for a diagram obtained by odd reflections from the original diagram.

Let $\Gamma$ be a subfinite regular Kac-Moody diagram with $n\geq 6$ vertices, which is not a cycle.  Choose a vertex
$v_n$ such that $\Gamma':=\Gamma \setminus \{v_n\}$ is connected and contains an isotropic vertex. Then $\Gamma'$ satisfies (\ref{finite}).  We may assume that the subdiagram $\{v_1,v_2,v_3\}$ contains an isotropic vertex, by using odd reflections in the subdiagram $\{v_4,\ldots,v_{n-1}\}$.  If $\Gamma'\setminus \{v_1\}$ does not contain an isotropic vertex, then $v_1$ is isotropic. In this case, by Table~\ref{rm261}, $\Gamma'$ is the standard diagram for $A(1,n-2)$ and an odd reflection at $v_1$ results in a diagram with $v_2$ isotropic.  So we may assume that $v_2$ or $v_3$ is isotropic.

Suppose that $v_n$ is connected only to the vertex $v_1$. If $a_{12},a_{21}=0$, then $a_{13},a_{31}\neq 0$ and $\{v_n,v_1,v_3,v_2,v_4\}\in\mathcal{F}$ implies that the subdiagram $\{v_n,v_1,v_3\}$ is $A(k,l)$ or $A_k$.  Then by Lemma~\ref{gamma}, the diagram $\{v_n,v_1,v_2,v_3,v_4,v_5\}$ is not a subfinite regular Kac-Moody diagram.  If $a_{12},a_{21}\neq 0$, then $\Gamma_{n-1}\in\mathcal{F}$ implies $a_{13},a_{31}=0$.  The condition $\{v_n,v_1,v_2,v_3,v_4\}\in\mathcal{F}$ yields two possibilities: either $\{v_n,v_1,v_2\}$ is a diagram from Table~\ref{rm261} and $\{v_1,v_2,v_3,v_4\}$ is $A(k,l)$ or $A_k$, or $\{v_4,v_3,v_2\}$ is a diagram from Table~\ref{rm261} and $\{v_n,v_1,v_2,v_3\}$ is $A(k,l)$ or $A_k$.  In the first case, $\Gamma_{n}$ is $A(k,l)$ or $A_k$, and we are done.  In the second case, $\Gamma$ is $A(k,l)$ or $A_k$, and we are done.

Now we may assume that $\Gamma\setminus \{v_1\}$ is connected, and hence satisfies (\ref{finite}).  Thus, the vertex $v_n$ is not connected to vertices $v_j$ with $3 < j < n-2$.  First suppose that $v_n$ is connected to $v_{n-2}$ or $v_{n-1}$.  Then $\Gamma_1\in\mathcal{F}$ implies that $v_n$ is not connected to $v_2$ or $v_3$ and $\Gamma_1$ satisfies (\ref{finite}).    By Lemma~\ref{lm12.4}, $v_n$ is not connected to $v_1$. Thus, $\Gamma$ satisfies (\ref{form}).
Now suppose that $v_n$ is not connected to $v_{n-2}$ or $v_{n-1}$.  If $v_n$ is connected to $v_3$ then  $\{v_n,v_1,v_2,v_3,v_4\}$ satisfies (\ref{finite}), implying $v_n$ is not connected to $v_1$ or $v_2$, $a_{13},a_{31}=0$, and the subdiagram $\{v_1,v_2,v_3\}$ is $A(k,l)$ or $A_k$.  But then by Lemma~\ref{gamma},
$\{v_n,v_1,v_2,v_3,v_4,v_5\}$ is not a subfinite regular Kac-Moody diagram.  Hence, $v_n$ is not connected to $v_3$. Finally, since subdiagram $\Gamma \setminus \{v_{n-1}\}$ satisfies (\ref{finite}) with $v_n$ connected only to $v_1$ or $v_2$ we conclude that $\Gamma$ satisfies (\ref{form}).
\end{proof}

\newpage
\subsection{Classification theorem}
\begin{theorem}\label{271}
If $A$ is a symmetrizable matrix and $\mathfrak{g}(A)$ has a simple isotropic root, then $\mathfrak{g}(A)$ is regular
Kac-Moody if and only if it has finite growth.
If $A$ is a non-symmetrizable matrix and $\mathfrak{g}(A)$ has a simple isotropic root, then $\mathfrak{g}(A)$ is regular Kac-Moody if and only if it is one of the following three classes:
\begin{center} \begin{tabular}{|c|c|}
  \hline &\\
  \bf{Algebra} & \bf{Dynkin diagrams} \\
\hline &\\
$q(n)^{(2)}$ & $\xymatrix{ & & \bullet \ar@<.5ex>[lld]^{a} \ar@<.5ex>[rrd]^{b} & & \\
\bullet \ar@<.5ex>[rru] \ar[r] & \bullet \ar[l] \ar[r] & \cdots \ar[l] \ar[r] & \bullet \ar[l] \ar[r]
& \bullet \ar[l] \ar@<.5ex>[llu] \\}$ \text{  } $\begin{array}{c} \text{There are n $\bullet$.} \\
\text{Each $\bullet$ is either $\bigcirc$ or $\bigotimes$.} \\ \text{An odd number of them are
$\bigotimes$.} \\
\text{If $\bullet$ is $\bigcirc$,
 then $a=b=-1$.} \\
\text{If $\bullet$ is $\bigotimes$, then $\frac{a}{b}=-1$.} \\
 \end{array}$
\\
\hline $\begin{array}{c} \\ \\ \\ S(1,2,\alpha) \end{array}$ &
$\xymatrix{& \bigotimes \ar@<.5ex>[rdd]^{-1+\alpha} \ar@<.5ex>[ldd]^{1} & \\
& & \\ \bigotimes \ar@<.5ex>[ruu]^{1} \ar@<.5ex>[rr]^{-1-\alpha} & &
\bigcirc \ar@<.5ex>[ll]^{-1} \ar@<.5ex>[luu]^{-1}}$ \text{  } $\begin{array}{c} \\ \\ \\ \alpha \neq 0, \text{  }\frac{1}{\alpha} \in \mathbb{C} \setminus \mathbb{Z} \end{array}$\\
 \hline &\\
& $\xymatrix{& \bigotimes \ar@<.5ex>[rdd]^{b} \ar@<.5ex>[ldd]^{1} & \\
& & \\ \bigotimes \ar@<.5ex>[ruu]^{c} \ar@<.5ex>[rr]^{1} & & \bigotimes \ar@<.5ex>[ll]^{a} \ar@<.5ex>[luu]^{1}}$ \text{
}
$\xymatrix{& \bigotimes \ar@<.5ex>[rdd]^{b} \ar@<.5ex>[ldd]^{1} & \\
& & \\ \bigcirc \ar@<.5ex>[ruu]^{-1} \ar@<.5ex>[rr]^{1+b+\frac{1}{c}} & &
\bigcirc \ar@<.5ex>[ll]^{1+a+\frac{1}{b}} \ar@<.5ex>[luu]^{-1}}$  \\
$Q^{\pm}(m,n,t)$ & $\xymatrix{& \bigotimes \ar@<.5ex>[rdd]^{c} \ar@<.5ex>[ldd]^{1} & \\
& & \\ \bigcirc \ar@<.5ex>[ruu]^{-1} \ar@<.5ex>[rr]^{1+c+\frac{1}{a}} & & \bigcirc \ar@<.5ex>[ll]^{1+b+\frac{1}{c}}
\ar@<.5ex>[luu]^{-1}}$ \text{  }
$\xymatrix{& \bigotimes \ar@<.5ex>[rdd]^{a} \ar@<.5ex>[ldd]^{1} & \\
& & \\ \bigcirc \ar@<.5ex>[ruu]^{-1} \ar@<.5ex>[rr]^{1+a+\frac{1}{b}} & &
\bigcirc \ar@<.5ex>[ll]^{1+c+\frac{1}{a}} \ar@<.5ex>[luu]^{-1}}$ \\
& $\begin{array}{l}
1+a+\frac{1}{b} = m \\
1+b+\frac{1}{c} = n \\
1+c+\frac{1}{a} = t \\
\end{array}$
\hspace{.5cm} $\begin{array}{l}
 m,n,t \in \mathbb{Z}_{\leq -1} \text{ and}\\
 \text{not all equal to }-1, \\
 a,b,c \in \mathbb{R}\backslash\mathbb{Q}. \\
 \end{array}$ \\
 &\\
\hline
\end{tabular}
\end{center}
\end{theorem}

\begin{remark}
This theorem follows from the classification of connected subfinite regular Kac-Moody diagrams in Section~\ref{s1} and Theorem~\ref{thm2}.
The fact that $Q^{\pm}(m,n,t)$ is not symmetrizable and does not have finite growth will be proven in Section~\ref{s2}.  Note that $Q^{\pm}(m,n,t) \cong Q^{\pm}(n,t,m)$ as algebras (see Remark~\ref{remisom}).
\end{remark}

\section{The Lie superalgebra $Q^{\pm}(m,n,t)$}\label{s2}
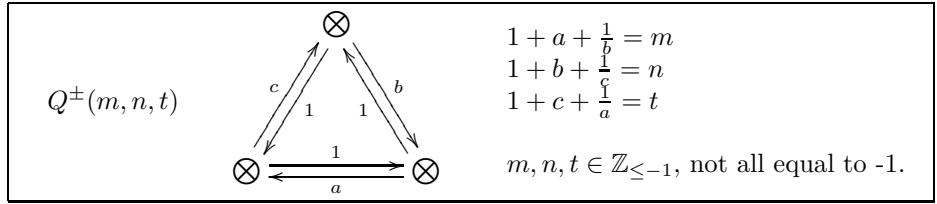
\begin{equation*}
\begin{tabular}{|ccc|}
\hline \hspace{.2cm} $Q^{\pm}(m,n,t)$ & $\begin{array}{l}\xymatrix{& \OX \AW[ldd]^{1}_{c} \AW[rdd]^{b}_{1} & \\ & & \\
\OX \AW[rr]^{1}_{a} & & \OX }
\end{array}$ &
$\begin{array}{l}
1+a+\frac{1}{b} = m \\
1+b+\frac{1}{c} = n \\
1+c+\frac{1}{a} = t \\ \\
m,n,t \in \mathbb{Z}_{\leq -1} \text{, not all equal to -1.} \end{array}$ \\
\hline
\end{tabular}\end{equation*}

In this section, we describe the parameters of the defining matrices for $Q^{\pm}(m,n,t)$, and then we show that
$Q^{\pm}(m,n,t)$ is not symmetrizable and does not have finite growth.

\begin{lemma}\label{lm333}
For each diagram $Q^{\pm}(m,n,t)$, it follows that $a,b,c \in \mathbb{R}\setminus\mathbb{Q}$, and there are two solutions of the above equations, namely $Q^{-}(m,n,t)$ with $a,b,c<-1$ and
$Q^{+}(m,n,t)$ with $-1<a,b,c<0$.
\end{lemma}
\begin{proof}
First, we show that a solution exists.  To simplify the calculations we let
\begin{equation*} M=1-m, \hspace{1cm} N=1-n, \hspace{1cm} T=1-t. \end{equation*}
Then $M,N,T \in \mathbb{Z}_{\geq 2}$.  Solving Equation 1 for $b$ and Equation 3 for $c$, and then substituting into
Equation 2 yields
\begin{equation*}
-N = \frac{-1}{a+M} + \frac{-a}{aT+1}.
\end{equation*}
By clearing denominators and regrouping, we have $f(a)=0$ where
\begin{equation*}
f(a)=(NT-1)a^{2}+(MNT-M+N-T)a+(MN-1).
\end{equation*}

Since $M,N,T\geq 2$, one has $NT-1>0$.  The discriminant of $f(a)$ is $D =(MNT-M-N-T)^{2} - 4$.  Now since $M-1,N-1,T-1 \in\mathbb{Z}_{\geq 1}$ and not all equal to 1, we have
\begin{align*}
3 &\leq (M-1)(N-1)(T-1) +1\\
&= MNT-(N-1)M-(T-1)N-(M-1)T \\
&< MNT-M-N-T,
 \end{align*}
which implies $D>0$.  Hence, $f(a)$ has two real roots. Moreover, these roots are not rational, since by taking $k=MNT-M-N-T$ and $y^{2}=D$, we obtain the equation $y^{2}=k^{2}-4$, which has only two solutions with integral $k$ and rational $y$, namely $k=\pm 2$, $y=0$.  Since $D\neq 0$, we conclude that $y$ is not rational.  Hence, $a\in\mathbb{R}\setminus\mathbb{Q}$, and it follows from the defining equations that $b,c\in\mathbb{R}\setminus\mathbb{Q}$.

Let $a_1 > a_2$ be the roots of the quadratic equation $f(a) = 0$. Since $NT-1,MN-1 >0$, the function $f(a)$ is concave up with $f(0)>0$.  We can express $f(-1)$ as
$$f(-1)= - (M-1)(N-1)(T-1)-(M-2)(T-2)+1,$$ where it is easy to see that $f(-1)<0$ for $M,N,T\in\mathbb{Z}_{\geq 2}$, not all equal to $2$.  Hence, $$a_2 < -1 < a_1 < 0.$$
Denote by $b_1 , b_2, c_1 , c_2$ the corresponding values of $b,c$.  From $a_i + b_i^{-1} \leq -2$ and similar formulas we obtain
$$-1 < a_1 < 0 \ \Longrightarrow\ -1 < b_1 < 0 \ \Longrightarrow\ -1 < c_1 < 0$$
and
$$a_2 < -1 \ \Longrightarrow\ c_2 < -1 \ \Longrightarrow\ b_2 < -1.$$
\end{proof}

\begin{corollary}\label{cor3}
The determinant of the Cartan matrix equals $1 + abc$ and is nonzero. Hence, the dimension of the Cartan subalgebra is
$3$.
\end{corollary}

\begin{remark}\label{remisom}
It is clear that $Q^{\pm}(m,n,t) \cong Q^{\pm}(n,t,m)$ as algebras by cyclic permutation of the variables $a,b,c$.  We also have
$Q^{\pm}(m,n,t) \cong Q^{\mp}(m,t,n)$ by transforming the equations: $a \rightarrow \frac{1}{b}$, $b \rightarrow
\frac{1}{a}$, $c \rightarrow \frac{1}{c}$.
\end{remark}

\begin{lemma}\label{Qnotsymm} $Q^{\pm}(m,n,t)$ is not symmetrizable. \end{lemma}
\begin{proof}
Suppose that we have a symmetrizable solution $a,b,c$.  We show that this implies $a \in \mathbb{Q}$, which is a
contradiction.  If the matrix is symmetrizable then $abc = 1$.  So, we substitute $c = \frac{1}{ab}$ into the defining
equations.  This yields
\begin{align*}
1+a+\frac{1}{b} = m \\
1+(a+1)b = n \\
1+\frac{1}{ab}+\frac{1}{a} = t\\
\end{align*}
From the first equation we have $b=\frac{1}{m-a-1}$.  Substituting this into the second equation and solving for $a$,
we have $a=\frac{mn-m-n}{n} \in \mathbb{Q}$.  Hence, there is no symmetrizable solution.\\
\end{proof}

\begin{lemma}\label{Qnotfg} $Q^{\pm}(m,n,t)$ does not have finite growth. \end{lemma}
\begin{proof}
The set of principal roots of $Q^{\pm}(m,n,t)$ is $\Pi_{0}=\{\alpha_1+\alpha_2,\alpha_2+\alpha_3,\alpha_3+\alpha_1\}$.
The Cartan matrix $B$ of the subalgebra of $Q^{\pm}(m,n,t)$ generated by $\Pi_{0}$ is
\begin{equation*}
B = \left( \begin{array}{ccc}
2 & m & m  \\
n & 2 & n \\
t & t & 2 \\
\end{array} \right).
\end{equation*}
All off diagonal entries of $B$ are negative integers.  Since they are not all equal to $-1$, this is not a Cartan matrix
of a finite-growth Kac-Moody algebra.  By Lemma~\ref{lemma B}, $Q^{\pm}(m,n,t)$ does not have finite growth.
\end{proof}

A Kac-Moody superalgebra is called {\em hyperbolic} if after the removal of any simple root the superalgebra is either of
finite or affine type.

\begin{lemma}
The regular Kac-Moody superalgebra $Q^{\pm}(m,n,t)$ is hyperbolic for the following $m \geq n \geq t$.
\begin{center}
\begin{tabular}{|c|c|c|}
\hline
m & n & t \\
\hline
-1 & -1 & -2 \\
-1 & -1 & -3 \\
-1 & -1 & -4 \\
-1 & -2 & -2 \\
-2 & -2 & -2 \\
\hline
\end{tabular}
\end{center}
\end{lemma}\

\section{Integrable modules}\label{s3}

Let $\mathfrak{g}(A)$ be a regular Kac-Moody superalgebra, and let
$\mathfrak{g}(A)=\mathfrak{h}\oplus(\oplus_{\alpha\in\Delta}\mathfrak{g}(A)_{\alpha})$ be the root space decomposition. Recall that we may assume that $a_{ii}\in \{0,2\}$ without loss of generality.  An element $\rho \in \mathfrak{h}^{*}$ such that $\rho(h_i)=\alpha_i(h_i)=\frac{1}{2}a_{ii}$ for all $i\in I$ is called a {\em Weyl vector}.

A root $\alpha$ is called {\em real} if $\alpha$ or $\frac{1}{2} \alpha$ is simple in some base obtained by a sequence of even and odd reflections, and it is called {\em imaginary} otherwise. For each
real root $\alpha$, the vector space $[\mathfrak{g}_{\alpha},\mathfrak{g}_{-\alpha}]$ is a one dimensional subspace of the Cartan subalgebra $\mathfrak{h}$. Moreover, if $\alpha$ is nonisotropic, then for each nonzero $h\in[\mathfrak{g}_{\alpha},\mathfrak{g}_{-\alpha}]$ we have that $\alpha(h)\neq 0$. In this case, we may choose $h_{\alpha}$ to be the unique vector in $[\mathfrak{g}_{\alpha},\mathfrak{g}_{-\alpha}]$ which satisfies $\alpha(h_{\alpha})=2$ (see Remark~\ref{remrescale}). Then take $x_{\alpha}\in\mathfrak{g}_{\alpha}$, $y_{\alpha}\in\mathfrak{g}_{-\alpha}$ such that $[x_{\alpha},y_{\alpha}]=h_{\alpha}$.  If $\alpha$ is even, then $\{x_{\alpha},y_{\alpha},h_{\alpha}\}$ forms an $\mathfrak{sl}_{2}$-triple.  If $\alpha$ is odd nonisotropic, then $\{x_{\alpha},y_{\alpha},h_{\alpha}, [x_{\alpha},x_{\alpha}],[y_{\alpha},y_{\alpha}] \}$ is a basis for a subalgebra which is isomorphic to $\mathfrak{osp}(1,2)$, and  $\{-\frac{1}{4}[x_{\alpha},x_{\alpha}],\frac{1}{4}[y_{\alpha},y_{\alpha}],\frac{1}{2}h_{\alpha} \}$ forms an $\mathfrak{sl}_{2}$-triple.

Let $U(\mathfrak{g})$ denote the universal enveloping algebra of a Lie superalgebra $\mathfrak{g}$. Let
$\mathfrak{n}_{+}$ (resp. $\mathfrak{n}_{-}$) denote the subalgebra of $\mathfrak{g}(A)$ generated by the elements
$X_i$ (resp. $Y_i$), $i\in I$.  Then one has the triangular decomposition
$\mathfrak{g}(A)=\mathfrak{n}_{-}+\mathfrak{h}+\mathfrak{n}_{+}$.

A $\mathfrak{g}(A)$-module $V$ is called a {\em weight module} if $V=\oplus_{\mu \in \mathfrak{h}^{*}} V_{\mu}$, where
$V_{\mu}=\{v\in V \mid hv=\mu(h)v, \text{ for all } h\in\mathfrak{h}\}$. If $V_{\mu}$ is non-zero, then $\mu$ is called a {\em weight}.  A $\mathfrak{g}(A)$-module $V$ is called a {\em highest weight module} with {\em highest weight} $\lambda \in \mathfrak{h}^{*}$ if there exists a vector $v_{\lambda}\in V$ such that
\begin{equation*}
\mathfrak{n}_{+}v_{\lambda}=0, \hspace{1cm} hv=\lambda(h)v_{\lambda} \text{ for } h\in\mathfrak{h}, \hspace{1cm}
U(\mathfrak{g}(A))v_{\lambda}=V.
\end{equation*}
A highest weight module is a weight module.  We let $L(\lambda)$ denote the irreducible highest weight module with
highest weight $\lambda$.
To simplify notation we define $\lambda_i := \lambda(h_i)$, for $i\in I$.

A subalgebra $\mathfrak{s}$ of $\mathfrak{g}(A)$ is {\em locally finite} on a module $V$ if $\operatorname{dim
}U(\mathfrak{s})v<\infty$ for any $v\in V$.  An element $x\in\mathfrak{g}$ is {\em locally nilpotent} on $V$ if for any
$v\in V$ there exists a positive integer $N$ such that $x^{N}v=0$.\\

\subsection{Integrable modules over affine Lie superalgebras}\label{SectionAffine}

Let $\mathfrak{g}$ be a basic Lie superalgebra, that is, a finite-dimensional simple Lie superalgebra over $\mathbb{C}$ with a nondegenerate even symmetric invariant bilinear form $(\cdot,\cdot)$, such that $\mathfrak{g}_{\bar{0}}$ is reductive \cite{K77}.  The associated (non-twisted) affine Lie superalgebra is
\begin{equation*}
\widehat{\mathfrak{g}} = \left(\mathbb{C}[t,t^{-1}] \otimes_{\mathbb{C}} \mathfrak{g}\right)
\oplus \mathbb{C}K \oplus \mathbb{C}d
\end{equation*}
with commutation relations
\begin{equation*}
[a(n),b(l)]=[a,b](n+l) + n \delta_{n,-l} (a|b)K, \hspace{.5cm} [d,a(n)]=-n a(n),
\hspace{.5cm} [K,\widehat{\mathfrak{g}}]=0
\end{equation*}
where $a,b \in \mathfrak{g}$; $n,l \in \mathbb{Z}$ and $a(n)=t^{n}\otimes a$.  By identifying $\mathfrak{g}$ with
$1 \otimes \mathfrak{g}$, we have that the Cartan subalgebra of $\widehat{\mathfrak{g}}$ is
$$\mathfrak{h} = \overset{\circ}{\mathfrak{h}} \oplus \mathbb{C}K \oplus \mathbb{C}d$$
where $\overset{\circ}{\mathfrak{h}}$ is the Cartan subalgebra of $\mathfrak{g}$.
Let $\delta$ be the linear function defined on $\mathfrak{h}$ by
$\delta\mid_{\overset{\circ}{\mathfrak{h}} \oplus \mathbb{C}K} = 0$ and $\delta(d) = 1$.
Let $\theta$ be the unique highest weight of the $\mathfrak{g}$.  Then $\alpha_{0}=\delta-\theta$ is the
additional simple root which extends $\mathfrak{g}$ to $\widehat{\mathfrak{g}}$.

Now suppose $\sigma$ is an automorphism of $\mathfrak{g}$ with order $m\neq 1$.  Then $\mathfrak{g}$ decomposes into
eigenspaces:
\begin{equation*}
\mathfrak{g} = \bigoplus_{j \in \mathbb{Z}/{m\mathbb{Z}}} \mathfrak{g}(\sigma)_{j}, \hspace{1cm}
 \mathfrak{g}(\sigma)_{j}= \{x \in \mathfrak{g} \mid \sigma(x) = e^{\frac{j 2\Pi i}{m}} x\}.
\end{equation*}
The associated twisted affine Lie superalgebra is
\begin{equation*}
\mathfrak{g}^{(m)} = \left(\bigoplus_{j \in \mathbb{Z}} t^{j} \otimes \mathfrak{g}(\sigma)_{j \operatorname{mod } m}\right)
\oplus \mathbb{C}K \oplus \mathbb{C}d
\end{equation*}
with the same commutation relations as given above. By identifying $\mathfrak{g}(\sigma)_{0}$ with
$1 \otimes \mathfrak{g}(\sigma)_{0}$, we have that the Cartan subalgebra of $\mathfrak{g}^{(m)}$ is
$\mathfrak{h} = \overset{\circ}{\mathfrak{h}} \oplus \mathbb{C}K \oplus \mathbb{C}d$ where $\overset{\circ}{\mathfrak{h}}$ is the Cartan subalgebra of $\mathfrak{g}(\sigma)_{0}$.
Let $\delta$ be the linear function defined on $\mathfrak{h}$ by
$\delta\mid_{{\mathfrak{h}}_{0} \oplus \mathbb{C}K} = 0$ and $\delta(d) = 1$.
Let $\theta$ be the unique highest weight of the $\mathfrak{g}(\sigma)_{0}$-module
$\mathfrak{g}(\sigma)_{1}$.  Then $\alpha_{0}=\delta-\theta$ is the additional simple root
which extends $\mathfrak{g}(\sigma)_{0}$ to $\mathfrak{g}^{(m)}$.

The bilinear form $(\cdot,\cdot)$ on $\mathfrak{g}$ gives rise to a nondegenerate symmetric
invariant bilinear form on $\widehat{\mathfrak{g}}$ and on $\mathfrak{g}^{(m)}$ by:
\begin{eqnarray*}
&(a(n),b(l))=\delta_{n,-l}(a,b) &(\mathbb{C}[t,t^{-1}]\otimes\mathfrak{g},\mathbb{C}K\oplus\mathbb{C}d)=0 \\
&(K,K)=(d,d)=0,\ (K,d)=1 &(\bigoplus_{j \in \mathbb{Z}} t^{j} \otimes \mathfrak{g}(\sigma)_{j \operatorname{mod } m},\mathbb{C}K\oplus\mathbb{C}d)=0.
\end{eqnarray*}
The restriction of this form to $\mathfrak{h}$ is nondegenerate and is also denoted by $(\cdot,\cdot)$.  We use this form to identify $\mathfrak{h}$ with $\mathfrak{h}^{*}$, which induces a nondegenerate bilinear form on $\mathfrak{h}^{*}$.  Let $V(\lambda)$ be a highest module over $ \widehat{\mathfrak{g}}$ . Then $k = \lambda(K)$ is the {\em level} of $L(\lambda)$.

A highest weight module $V$ over an affine Lie superalgebra $ \widehat{\mathfrak{g}} $ (resp. $\mathfrak{g}^{(m)}$) is called {\em integrable} if it is integrable over the affine Lie algebra $\widehat{\mathfrak{g}}_{\bar{0}}$ (resp. $\mathfrak{g}^{(m)}_{\bar{0}}$), that is, if $\mathfrak{g}_{\alpha}(n)$ is locally finite on $V$ for every root $\alpha$ of $\mathfrak{g}_{\bar{0}}$ and $n\in\mathbb{Z}$.

Since $\mathfrak{g}_{\bar{0}} = \oplus_{j \in \mathbb{Z}/{m\mathbb{Z}}} \left(\mathfrak{g}(\sigma)_{j}\right)_{\bar{0}}$ is a graded reductive Lie algebra, it follows that $(\mathfrak{g}(\sigma)_{0})_{\bar{0}}$ is reductive. Hence we may write $\mathfrak{g}_{\bar{0}}$ (resp. $(\mathfrak{g}(\sigma)_{0})_{\bar{0}}$)
as a sum $\oplus_{i=0}^{N}\mathfrak{g}_{\bar{0}i}$, where $\mathfrak{g}_{\bar{0}0}$ is abelian and $\mathfrak{g}_{\bar{0}i}$ for $i=1,\ldots ,N$ are simple Lie algebras.  Then for each $i$, the superalgebra $\widehat{\mathfrak{g}}$ (resp. $\mathfrak{g}^{(m)}$) contains an affine Lie algebra $\widehat{\mathfrak{g}}_{\bar{0}i}$ associated to $\mathfrak{g}_{\bar{0}i}$.
Explicitly, we have $$\widehat{\mathfrak{g}}_{\bar{0}i}=
\left(\mathbb{C}[s,s^{-1}] \otimes_{\mathbb{C}} \mathfrak{g}_{\bar{0}i}\right)
\oplus \mathbb{C}K' \oplus \mathbb{C}d,$$ where $s=t$, $K'=K$ (resp. $s=t^{m}$, $K'=mK$).  The Cartan subalgebra of $\widehat{\mathfrak{g}}_{\bar{0}i}$ is $\mathfrak{h}_i=\overset{\circ}{\mathfrak{h}_i}\oplus\mathbb{C}K\oplus\mathbb{C}d$, where $\overset{\circ}{\mathfrak{h}_i}=\overset{\circ}{\mathfrak{h}}\cap \mathfrak{g}_{\bar{0}i}$.
If $V$ is integrable over $ \widehat{\mathfrak{g}} $ (resp. $\mathfrak{g}^{(m)}$), then it follows from the definition that $V$ is integrable over $\widehat{\mathfrak{g}}_{\bar{0}i}$, $i=1,\ldots,N$.

It was shown in \cite{KW01}, that most non-twisted affine Lie superalgebras have only trivial irreducible integrable
highest weight modules, which led to the consideration of weaker notions of integrability.

\begin{proposition}[Kac, Wakimoto]\label{p1}
The only non-twisted affine Lie superalgebras with non-trivial irreducible integrable highest weight
modules are $B(0,n)^{(1)}$, $C(n)^{(1)}$ and $A(0,m)^{(1)}$.
\end{proposition}

In \cite{EF03}, S. Eswara Rao and V. Futorny show that over non-twisted affine Lie superalgebras all irreducible
integrable highest weight modules with non-zero level are highest weight modules.  The proof of the following statement is similar to the non-twisted case.

\begin{proposition}\label{p2}
Let $\mathfrak{g}^{(m)}$ be a twisted affine Lie superalgebra which is not one of the algebras:
$A(0,2n-1)^{(2)}$, $A(0,2n)^{(4)}$ and $C(n)^{(2)}$.  Then an irreducible integrable highest weight module over
$\mathfrak{g}^{(m)}$ is trivial.
\end{proposition}

\begin{proof}
First we consider the non-symmetrizable twisted affine Lie superalgebra, $q(n)^{(2)}$.  The Lie superalgebra $q(n)^{(2)}$ is not covered by the construction given above, so we must handle this case separately.  If $n$ is odd, then we have a Cartan matrix defined by $a_{i,i+1}=-1$, $a_{i+1,i}=1$ for $i=1,\dots,n$ $(mod\ n)$, and all other entries zero. All
simple roots in this case are odd isotropic.  By an odd reflection at $\alpha_i$,  we obtain an even root $\alpha_i +
\alpha_{i+1}$ with $h_{\alpha_i + \alpha_{i+1}}=h_{i+1}-h_{i}$.  By Lemma~\ref{lm67}, the conditions for integrability
are $\lambda_{i+1}-\lambda_{i} \in \mathbb{Z}_{\geq 0}$, $(mod\ n)$. This implies $\lambda_i=0$ for all $i=1,\dots,n$.

If $n$ is even, the we have a Cartan matrix defined by $a_{i,i+1}=-1$, $a_{i+1,i}=1$ for $i=1,\dots,n-1$, $a_{1,1}=2$,
$a_{1,n}=a_{n,1}=-1$, and all other entries zero.  In this case, the simple root $\alpha_1$ is even, and all other
simple roots are odd isotropic.  Using odd reflections, we obtain the set of principal roots
$\{\alpha_1,\alpha_2+\alpha_3,\alpha_3+\alpha_4,\dots,\alpha_{n-1}+\alpha_{n},\alpha_{n}+\alpha_{1}+\alpha_{2}\}$.
For $i=2,\dots,n-1$ we have $h_{\alpha_{i}+\alpha_{i+1}}=h_{i+1}-h_{i}$, and
$h_{\alpha_n+\alpha_1+\alpha_2}=h_2-h_n-h_1$.  This yields integrability conditions $\lambda_1 \in \mathbb{Z}_{\geq 0}$,
$\lambda_{i+1}-\lambda_{i} \in \mathbb{Z}_{\geq 0}$ for $i=2,\dots,n-1$, and
$\lambda_2-\lambda_n-\lambda_1 \in \mathbb{Z}_{\geq 0}$.  Together these imply $\lambda_i=0$ for $i=1,\dots,n$.

For a symmetrizable twisted affine Lie superalgebra $\mathfrak{g}^{(m)}$, we have the standard nondegenerate
symmetric invariant bilinear form.  The structure of $\mathfrak{g}^{(m)}_{\overline{0}}$, the even part of
$\mathfrak{g}^{(m)}$, is given by van de Leur in \cite{L86}.  Let $\mathfrak{g}^{(m)}$ be a symmetrizable twisted
affine Lie superalgebra, which is not one of the algebras: $A(0,2n-1)^{(2)}$, $A(0,2n)^{(4)}$ and
$C(n)^{(2)}$, and choose a base $\Pi=\{\alpha_0,\dots,\alpha_l\}$ with a unique simple isotropic root, $\alpha_d$.
Using the bilinear form to identify $\mathfrak{h}$ with $\mathfrak{h}^{*}$, we have
$\alpha_i = {\alpha_i}^{\vee}$ for $i \leq d$ and $\alpha_i = -{\alpha_i}^{\vee}$ for $i > d$.
Note that $d\neq0$ and $d\neq l$ by choice of $\mathfrak{g}^{(m)}$.

Then $\mathfrak{g}^{(m)}_{\overline{0}}$ has as subalgebras,
both an affine subalgebra $\mathfrak{g'}$ on which this form is positive definite,
and an affine subalgebra $\mathfrak{g''}$ on which this form is negative definite.
Denote by $\overset{\circ}{\mathfrak{g}}'$ (resp. $\overset{\circ}{\mathfrak{g}}''$) the finite part of
$\mathfrak{g'}$ (resp. $\mathfrak{g''}$).  Let $\theta'$ (resp. $\theta''$) denote the highest weight
of the module $\overset{\circ}{\mathfrak{g}}'_{1}$
(resp. $\overset{\circ}{\mathfrak{g}}''_{1}$) over
$\overset{\circ}{\mathfrak{g}}'_{0}$ (resp. $\overset{\circ}{\mathfrak{g}}''_{0}$).
Then the simple root $\alpha'_{0}=\delta-\theta'$ (resp. $\alpha''_{0}=\delta-\theta''$) extends
$\overset{\circ}{\mathfrak{g}}'$ (resp. $\overset{\circ}{\mathfrak{g}}''$) to
$\mathfrak{g'}$ (resp. $\mathfrak{g''}$).  Also,  $\theta'=\sum_{i=0}^{d-1} c_i \alpha_i$ and
$\theta''=\sum_{i=d+1}^{l} c_i \alpha_i$ with $c_i \in \mathbb{Z}_{>0}$.

Let $L(\lambda)$ be an irreducible integrable highest weight module over $\mathfrak{g}^{(m)}$.  Let $k=\lambda(K)$,
$k_i=\lambda({\alpha_i}^{\vee})$, $k'=\lambda({\alpha'_{0}}^{\vee})$ and $k''=\lambda({\alpha''_{0}}^{\vee})$.
Then by Lemma~\ref{lm67}, $k_i \in \mathbb{Z}_{\geq 0}$ for $i \in I\setminus \{d\}$ and $k',k'' \in \mathbb{Z}_{\geq 0}$.
Now using the bilinear form to identify $\mathfrak{h}$ with $\mathfrak{h}^{*}$, we have
\begin{equation*}
\begin{array}{llll}
{\alpha'_{0}}^{\vee} & = \alpha'_{0} & = \delta-\theta' & = K-{\theta'}^{\vee} \\
{\alpha''_{0}}^{\vee} & = -\alpha''_{0} & = -(\delta-\theta'') & = -(K-{\theta''}^{\vee}).
\end{array}
\end{equation*}
Hence,
\begin{equation*}
\begin{array}{rlll}
k & =\lambda(K) & =\lambda({\alpha'_{0}}^{\vee})+\lambda({\theta'}^{\vee}) & = k' + \lambda({\theta'}^{\vee})\\
-k & = -\lambda(K) & = \lambda({\alpha''_{0}}^{\vee})-\lambda({\theta''}^{\vee}) & = k'' - \lambda({\theta''}^{\vee}).
\end{array}
\end{equation*}
Since ${\theta'}^{\vee}$ (resp. ${\theta''}^{\vee}$) is a positive (resp. negative)
combination of ${\alpha_i}^{\vee}$ with $i \in I \setminus \{d\}$ and
$\lambda({\alpha_i}^{\vee})\geq 0$ for $i \in I \setminus \{d\}$, we have that
$\lambda({\theta'}^{\vee}) \geq 0$ and $ \lambda({\theta''}^{\vee}) \leq 0$.
Hence, $k=0$ and the level of $L(\lambda)$ is zero.  The above equations then imply that
$k',k'',\lambda({\theta'}^{\vee}),\lambda({\theta''}^{\vee})=0$.  Since
$\lambda({\theta'}^{\vee})=\sum_{i=0}^{d-1} c_i k_i$ and
$\lambda({\theta''}^{\vee})=-\sum_{i=d+1}^{l} c_i k_i$ with $c_i \in \mathbb{Z}_{>0}$,
we have that $k_i=0$ for $i \in I \setminus \{d\}$.  Finally, $k'=0$ then implies $k_d=0$.
Hence, the weight $\lambda$ is zero, and so the module $L(\lambda)$ is trivial.
\end{proof}

\begin{remark} We can summarize this as follows.  If $\overset{\circ}{\mathfrak{g}}_{\bar{0}}$ is a direct sum of two or more simple Lie algebras, then the corresponding affine (or twisted affine) Lie superalgebra does not have any nontrivial irreducible integrable highest weight modules.
\end{remark}

\subsection{Integrable modules over regular Kac-Moody superalgebras}

Let $\mathfrak{g}(A)$ be a regular Kac-Moody superalgebra, and let $V$ be a weight module over $\mathfrak{g}(A)$.  We
call $V$ {\em integrable} if for every real root $\alpha$  the element $X_{\alpha} \in \mathfrak{g}(A)_{\alpha} $ is
locally nilpotent on $V$.  If $\mathfrak{g}(A)$ is an affine Lie superalgebra then this definition coincides with the definition given in Section~\ref{SectionAffine}.  This follows from the fact that every even root of an affine Lie
superalgebra with non-zero length is real, as was shown in the dissertation of V. Serganova.

The following lemma follows from Lemma~\ref{lm23}.

\begin{lemma}\label{lm411}
The adjoint module of a regular Kac-Moody superalgebra is an integrable module.  In particular, $\operatorname{ad }_{X_{\alpha}}$ is locally nilpotent for every real root $\alpha$, where $X_{\alpha} \in \mathfrak{g}(A)_{\alpha}$.
\end{lemma}

The following lemma follows from Proposition~\ref{p1} and Proposition~\ref{p2}.

\begin{lemma}\label{lm123}
Suppose $\mathfrak{g}(A)$ is a regular Kac-Moody superalgebra with a subfinite regular Kac-Moody diagram which is not of
finite type and it is not one of the algebras: $A(0,m)^{(1)}$, $C(n)^{(1)}$, $S(1,2,\alpha)$, and $Q^{\pm}(m,n,t)$. Then
all irreducible integrable highest weight modules are trivial.
\end{lemma}

The conditions for an irreducible highest weight module over $A(0,m)^{(1)}$ or $C(n)^{(1)}$ to be integrable were given
in \cite{KW01}.  We are interested in the irreducible integrable highest weight modules for the regular Kac-Moody
superalgebras: $S(1,2,\alpha)$ and $Q^{\pm}(m,n,t)$.  We will see that they have non-trivial irreducible integrable
highest weight modules, and we will describe the weights.  First we need the following lemma.

\begin{lemma}\label{lm67}
Let $\mathfrak{g}(A)$ be a Kac-Moody superalgebra, and let $L(\lambda)$ be an irreducible highest weight module.  If
$\alpha_i$ is a simple non-isotropic root of $\mathfrak{g}(A)$, then $Y_i \in \mathfrak{g}(A)_{-\alpha_i}$ is locally
nilpotent on $L(\lambda)$ if and only if $\lambda_i \in 2^{p(i)}\mathbb{Z}_{\geq 0}$.  If $\alpha_i$ is a simple
isotropic root, then $Y_i \in \mathfrak{g}(A)_{-\alpha_i}$ is locally nilpotent on $L(\lambda)$.
\end{lemma}

\begin{proof}
If $\alpha_i$ is a simple even root, let $e=X_i$, $f=Y_i$ and $h=h_i$.  Then $\{e, f, h\}$ is an $\mathfrak{sl}_{2}$-triple. If
$\alpha_i$ is a simple odd non-isotropic root, let $e=-\frac{1}{4}[X_i,X_i]$, $f=\frac{1}{4}[Y_i,Y_i]$, and
$h=\frac{1}{2}h_i$.  Then $\{e, f, h\}$ is an $\mathfrak{sl}_{2}$-triple.  Since $\mathfrak{g}(A)$ is a Kac-Moody superalgebra, it
follows that $f$ is locally nilpotent on $V$ if and only if $f$ is nilpotent on the highest weight vector $v_{\lambda}$.
Now $f$ is nilpotent on $v_{\lambda}$ if and only if $\lambda(h) \in \mathbb{Z}_{\geq 0}$.  If $\alpha_i$ is odd, then
$[Y_i,Y_i] v = 2(Y_i)^{2} v$ for $v \in V$. Thus, $f$ is nilpotent on $v_{\lambda}$ if and only if $Y_i$ is nilpotent on
$v_{\lambda}$. Hence, $Y_i \in \mathfrak{g}(A)_{-\alpha_i}$ is locally nilpotent on $L(\lambda)$ if and only if
$\lambda_i \in 2^{p(i)}\mathbb{Z}_{\geq 0}$.  If $\alpha_i$ is a simple isotropic root, then $(Y_i)^{2}v=0$ for all $v
\in L(\lambda)$.
\end{proof}

\begin{lemma} \label{lm222}
Let $L(\lambda)$ be an irreducible integrable highest weight module of a regular Kac-Moody superalgebra
$\mathfrak{g}(A)$, and let $\alpha_s$ be a simple root.  Let $\mathfrak{n}'_{+} = r_s(\mathfrak{n}_+)$ and $A'=r_s(A)$.
Then after a simple even or odd reflection $r_s$ the module $L(\lambda)$ is an irreducible integrable highest weight
module of $\mathfrak{g}(A')$ with highest weight given below.
\begin{enumerate}
  \item If $\lambda(h_s)=0$, then $v'_{\lambda}=v_{\lambda}$ is a highest weight vector with respect to
$\mathfrak{n}'_{+}$ and the highest weight is $\lambda'=\lambda$.
  \item If $\lambda(h_s)\neq 0$ and $\alpha_s$ is isotropic, then $v'_{\lambda}=Y_s v_{\lambda}$ is a highest weight
vector with respect to $\mathfrak{n}'_{+}$ and the highest weight is $\lambda'=\lambda-\alpha_s$.
  \item If $\lambda(h_s)=k\neq 0$ and $\alpha_s$ is even, then $v'_{\lambda}= (Y_s)^{k} v_{\lambda}$ is a
highest weight vector with respect to $\mathfrak{n}'_{+}$ and the highest weight is $\lambda'=\lambda-k\alpha_s$.
 \item If $\lambda(h_s)=2n\neq 0$ and $\alpha_s$ is odd non-isotropic, then $v'_{\lambda}= (Y_s)^{2n} v_{\lambda}$ is a
highest weight vector with respect to $\mathfrak{n}'_{+}$ and the highest weight is $\lambda'=\lambda-2n\alpha_s$.
\end{enumerate}
\end{lemma}

\begin{proof}
(1) and (2) follow immediately from the facts that $[Y_s,Y_s] v_{\lambda} = 2(Y_s)^{2} v_{\lambda}$, and that $Y_s
v_{\lambda}=0$ if and only if $\lambda(h_s)=0$.  For (3) and (4) suppose now that $\alpha_s$ is a non-isotropic root.
Then by Lemma~\ref{lm67}, we have $k=\lambda(h_s) \in 2^{p(s)}\mathbb{Z}_{\geq 0}$ since the module $L(\lambda)$ is
assumed to be integrable.  If $\alpha_s$ is a simple even root, then $\{X_s, Y_s, h_s\}$ is an $\mathfrak{sl}_{2}$-triple.  Set $v_j =
\frac{1}{j!}(Y_s)^{(j)} v_{\lambda}$. Then $h_s v_j = (k-2j) v_j$ and $X_s v_j = (k+1-j)v_{j-1}$. Thus $X_s v_{k+1}=0$,
and $X_i v_{k+1} = 0$ for all $i \in I \setminus \{s\}$. Since the module $L(\lambda)$ is irreducible this implies that
$v_{k+1}=0$. Thus $Y_s v_{k} = 0$ with $v_{k}\neq 0$.  Hence, $L(\lambda)$ is a highest weight module with respect to
$\mathfrak{n}'_{+}$, with highest weight vector $v_{k}=\frac{1}{k!}(Y_s)^{k} v_{\lambda}$ and the highest weight is
$\lambda-k\alpha_s$.

Finally, suppose that $\alpha_s$ is an odd non-isotropic root.  Set $v_{j} = (Y_s)^{(j)} v_{\lambda}$. Then $h_s v_{j} =
(2n-2j) v_{j}$, $X_s v_{2i} = (2i) v_{2i-1}$ and $X_s v_{2i-1} = (2n+2-2i) v_{2i-2}$. Thus, $X_s (Y_s)^{2n+1} v_{\lambda}
= 0$, and $X_i (Y_s)^{2n+1} v_{\lambda} = 0$ for $i\in I\setminus \{s\}$. Since the module $L(\lambda)$ is irreducible,
we conclude that $(Y_s)^{2n+1} v_{\lambda} = 0$. Hence, $Y_s v_{2n} = 0$ with $v_{2n}\neq 0$. Therefore, $L(\lambda)$ is
a highest weight module with respect to $\mathfrak{n}'_{+}$, with highest weight vector $(Y_s)^{2n} v_{\lambda}$ and
highest weight $\lambda-(2n)\alpha_s$. The fact that the module $L(\lambda)$ is integrable as a $\mathfrak{g}(A')$ module
follows from the fact that the real roots of $\mathfrak{g}(A)$ and  $\mathfrak{g}(A')$ coincide.
\end{proof}

\begin{lemma}
Let $\mathfrak{g}(A)$ be a Kac-Moody superalgebra, and let $L(\lambda)$ be a highest weight module.  Let $\alpha_s$ and
$\alpha_i$ be simple non-isotropic roots.  Let $r_s$ be the even reflection with respect to $\alpha_s$ (or $2\alpha_s$ if
$\alpha_s$ is odd).  Suppose $Y_{i} \in \mathfrak{g}(A)_{-\alpha_i}$ is locally nilpotent on $L(\lambda)$, then $Y'_i \in
\mathfrak{g}(A)_{-r_s(\alpha_i)}$ is locally nilpotent on $L(\lambda)$.
\end{lemma}

\begin{proof}
The reflection $r_s$ does not change the Cartan matrix, so $\mathfrak{g}(A')$ is again a Kac-Moody superalgebra. Thus, it
suffices to show that $Y'_i$ is nilpotent on the highest weight vector $v_{\lambda'}$ of $L(\lambda')$. It is sufficient
to consider the case when both $\alpha_s$ and $\alpha_i$ are even roots. Then this is equivalent to the condition
$\lambda'(h'_i) \in \mathbb{Z}_{\geq 0}$. Now $\lambda' = \lambda - \lambda(h_s)\alpha_s$. We have $\alpha'_s= -\alpha_s$
and $\alpha'_i = \alpha_i - a_{si} \alpha_s$. Also, $h'_s=-h_s$ and $h'_i=h_i - a_{is}h_s$. Then
\begin{equation*}
\lambda'(h'_i)= (\lambda - \lambda(h_s)\alpha_s)(h_i - a_{is}h_s) = \lambda(h_i).
\end{equation*}
Since $Y_i$ is locally nilpotent on $L(\lambda)$ we have $\lambda(h_i) \in \mathbb{Z}_{\geq 0}$. Hence $Y'_i$ is locally
nilpotent.
\end{proof}

\begin{corollary} \label{cor2}
Let $\mathfrak{g}(A)$ be a regular Kac-Moody superalgebra.  The irreducible highest weight module $L(\lambda)$ is an
integrable module if and only if the element $Y_{\alpha} \in \mathfrak{g}(A)_{-\alpha}$ is locally nilpotent on
$L(\lambda)$ for each principal root $\alpha$.
\end{corollary}

We call a weight $\lambda$ {\em typical} if for any real isotropic root $\alpha$ we have $(\lambda+\rho)(h_\alpha)\neq 0$.

\subsection{Integrable modules of the Lie superalgebra $S(1,2,\alpha)$}

Now we describe the weights for integrable highest weight modules of the Lie superalgebra $S(1,2,\alpha)$. The Cartan
matrix for the superalgebra $S(1,2,\alpha)$ is
\begin{equation*}
B= \left(\begin{array}{ccc}
2 & -1 & -1 \\
-1+\alpha & 0 & 1 \\
-1-\alpha & 1 & 0 \\
\end{array} \right)
\end{equation*}
with $\alpha\neq 0$ and $\frac{1}{\alpha} \not\in\mathbb{Z}$.

\begin{lemma}
Let $L(\lambda)$ be an irreducible highest weight module for $S(1,2,\alpha)$. If $\lambda_2 \neq 0$ or $\lambda_3 \neq
0$, then $L(\lambda)$ is integrable if and only if
\begin{eqnarray*}
\lambda_1 &\in &\mathbb{Z}_{\geq 0} \\
\lambda_2+\lambda_3 -1 &\in &\mathbb{Z}_{\geq 0}.
\end{eqnarray*}
If $\lambda_2 = \lambda_3 = 0$, then $L(\lambda)$ is integrable if and only if $\lambda_1 \in \mathbb{Z}_{\geq 0}$.
\end{lemma}

\begin{proof}
By Corollary~\ref{cor2}, it suffices to find the conditions for the principal roots to be locally nilpotent on
$L(\lambda)$. The principal roots of $S(1,2,\alpha)$ are $\alpha_1$ and $\alpha_2+\alpha_3$.  We have that $h_{\alpha_2 + \alpha_3}=h_2+h_3$ by Lemma~\ref{lm1} and the formula above it, (after rescaling $h_{\alpha_2 + \alpha_3}$ so that $\alpha_{\alpha_2 + \alpha_3}(h_{\alpha_2 + \alpha_3})=2$; see Remark~\ref{remrescale}).  By Lemma~\ref{lm67}, $Y_1 \in \mathfrak{g}(A)_{-\alpha_1}$ is locally nilpotent on
$L(\lambda)$ if and only if $\lambda_1 \in \mathbb{Z}_{\geq 0}$. First suppose $\lambda_2 \neq 0$ and consider the odd
reflection $r_2$. We have that $r_2(\alpha_3) = \alpha_2 + \alpha_3$ and by Lemma~\ref{lm222} we have that
$\lambda'=\lambda-\alpha_2$. Then by Lemma~\ref{lm67}, $Y'_3 \in \mathfrak{g}(A)_{-(\alpha_2 + \alpha_3)}$ is locally
nilpotent on $L(\lambda)$ if and only if $\lambda'(h'_3) \in \mathbb{Z}_{\geq 0}$ where
\begin{equation*}
\lambda'(h'_3) = (\lambda-\alpha_2)(h_2+h_3)= \lambda_2 + \lambda_3 - 1.
\end{equation*}
The argument for $\lambda_3 \neq 0$ is similar and yields the same condition.  Finally, if $\lambda_2 = \lambda_3 = 0$
then the additional integrability condition is $\lambda(h_2+h_3) \in \mathbb{Z}_{\geq 0}$, which is vacuously satisfied.
\end{proof}

\subsection{Integrable modules of the Lie superalgebra $Q^{\pm}(m,n,t)$}

Now we describe the weights for integrable highest weight modules of the Lie superalgebra $Q^{\pm}(m,n,t)$. The Cartan
matrix for $Q^{\pm}(m,n,t)$ is
\begin{equation*}
\left(\begin{array}{ccc}
0 & 1 & a \\
b & 0 & 1 \\
1 & c & 0 \\
\end{array}\right)
\hspace{2cm}
\begin{array}{l}
1+a+\frac{1}{b} = m \\
1+b+\frac{1}{c} = n \\
1+c+\frac{1}{a} = t \\
\end{array}
\end{equation*}
with $m,n,t \in \mathbb{Z}_{\leq -1}$, not all equal to -1, and the simple roots are odd isotropic.  The principal even
roots are $\{\alpha_1+\alpha_2, \alpha_2+\alpha_3, \alpha_1+\alpha_3\}$.  One can check that for $i\neq j$,
\begin{equation*}
h_{\alpha_i+\alpha_j}=\frac{h_j}{a_{ji}}+\frac{h_i}{a_{ij}}
\end{equation*}
using the formula appearing before Lemma~\ref{lm1} and rescaling $h_{\alpha_i + \alpha_j}$ so that $(\alpha_i+\alpha_j)(h_{\alpha_i+\alpha_j})=2$.

\begin{lemma}
A highest weight module $V(\lambda)$ for the algebra $Q^{\pm}(m,n,t)$ is typical if and only if $\lambda_1, \lambda_2,
\lambda_3 \neq 0$.
\end{lemma}

\begin{proof}
Since odd reflections of the diagram $\Gamma$ do not yield new simple odd roots, the only conditions for the module to be typical are $\lambda(h_1), \lambda(h_2), \lambda(h_3) \neq 0$.
\end{proof}

\begin{lemma}\label{lm111}
An irreducible highest weight module for $Q^{\pm}(m,n,t)$, with typical weight, is integrable if and only if
\begin{align*}
&\lambda_1+ \frac{1}{b}\lambda_2 -1 \\
&\lambda_2+ \frac{1}{c}\lambda_3 -1  \hspace{1cm} \in \mathbb{Z}_{\geq 0}.\\
&\lambda_3+ \frac{1}{a}\lambda_1 -1 \\
\end{align*}
\end{lemma}

\begin{proof}
By Corollary~\ref{cor2}, it suffices to find the conditions for $Y_{\alpha} \in \mathfrak{g}(A)_{-\alpha}$ to be locally
nilpotent on $L(\lambda)$ when $\alpha$ is a principal root. Let $\alpha_i$ be a simple isotropic root and let $r_i$ be
the odd reflection with respect to $\alpha_i$.  Since the weight $\lambda$ is typical, $\lambda_i \neq 0$.  Then by
Lemma~\ref{lm222}, $\lambda'=\lambda-\alpha_i$.  Since the simple even roots of $\mathfrak{g}(A')$ are $\alpha_i +
\alpha_j$ for $i\neq j$, we have by Lemma~\ref{lm67} that the conditions of integrability are $\lambda'(h'_j) \in
\mathbb{Z}_{\geq 0}$, where
\begin{equation*}
\lambda'(h'_j) = \lambda'(h_{\alpha_i+\alpha_j}) = (\lambda-\alpha_i)\left(\frac{h_j}{a_{ji}}+\frac{h_i}{a_{ij}}\right)
=\frac{\lambda_j}{a_{ji}}+\frac{\lambda_i}{a_{ij}}-1.
\end{equation*}
\end{proof}

\begin{proposition}
The non-trivial irreducible integrable highest weight modules of $Q^{\pm}(m,n,t)$ are $L(\lambda)$ such that
\begin{equation*}
\left(\begin{array}{c} \lambda_1 \\ \lambda_2 \\ \lambda_3 \end{array}\right) = \left(\frac{1}{1+abc}\right)
\left(\begin{array}{ccc}
1 & \frac{-1}{b} & \frac{1}{bc} \\
\frac{1}{ac} & 1 & \frac{-1}{c} \\
\frac{-1}{a} & \frac{1}{ba} & 1 \\
\end{array}\right)
\left(\begin{array}{c} x \\ y \\ z \end{array}\right),
\end{equation*}
with $x,y,z \in \mathbb{Z}_{>0}$.  These weights are typical.
\end{proposition}

\begin{proof}
The dimension of the Cartan subalgebra is $3$ by Corollary~\ref{cor3}, and hence $\lambda$ is determined by its values on
$h_1$, $h_2$ and $h_3$.  First we consider the case when $\lambda$ is typical. We can rewrite the conditions of
Lemma~\ref{lm111} using matrices:
\begin{equation*}
\left(\begin{array}{ccc}
1 & \frac{1}{b} & 0 \\
0 & 1 & \frac{1}{c} \\
\frac{1}{a} & 0 & 1 \\
\end{array}\right)
\left(\begin{array}{c}
\lambda_1 \\ \lambda_2 \\ \lambda_3 \\
\end{array}\right)
= \left(\begin{array}{c}
x \\ y \\ z \\
\end{array}\right)
\end{equation*}
with $x,y,z \in \mathbb{Z}_{>0}$.  The determinant of the left most matrix is $\frac{1+abc}{abc}$, which is nonzero by
Corollary~\ref{cor3}.  Hence, the matrix is invertible, and is calculated above.  Finally, suppose that $\lambda$ is not
typical.  Then without loss of generality suppose $\lambda_1=0$.  Consider the odd reflection $r_1$ with respect to
$\alpha_1$.  By Lemma~\ref{lm222}, $\lambda'=\lambda$.  Then by Lemma~\ref{lm67}, we have the integrability conditions
$x,z \in \mathbb{Z}_{\geq 0}$ with
\begin{align*}
&x=\lambda'(h_{\alpha_1+\alpha_2})= \lambda_1+\frac{1}{b}\lambda_2=\frac{1}{b}\lambda_2 \\
&z=\lambda'(h_{\alpha_1+\alpha_3})= \lambda_3+\frac{1}{a}\lambda_1=\lambda_3.
\end{align*}
If $\lambda_2$ or $\lambda_3$ is nonzero, then by a reflection at the corresponding simple root we obtain the
integrability condition $y\in \mathbb{Z}_{\geq 0}$ with
\begin{align*}
&y=\lambda_2+\frac{1}{c}\lambda_3-1.
\end{align*}
By Lemma~\ref{lm333}, $b,c < 0$.  But this together with $\frac{1}{b}\lambda_2, \lambda_3 \geq 0$ implies
$\lambda_2+\frac{1}{c}\lambda_3-1 < 0$, which is a contradiction.  Hence, if $\lambda$ is not typical then $\lambda=0$.
\end{proof}

\section{Extending regular Kac-Moody diagrams that are not of finite type}\label{s35}

In this section, we prove that a subfinite regular Kac-Moody diagram with an isotropic vertex, which is not of finite type, is not extendable. Hence, a regular Kac-Moody diagram with an isotropic vertex is subfinite.

\begin{lemma}
Suppose $\Gamma$ and $\Gamma' = \Gamma \cup \{v_{n+1}\}$ are connected regular Kac-Moody diagrams.  Let $\mathfrak{g}(A)$
(resp. $\mathfrak{g}(A')$) be the Kac-Moody superalgebra with diagram $\Gamma$ (resp. $\Gamma'$), and let $Y_{n+1} \in
\mathfrak{g}(A')_{-\alpha_{n+1}}$ be the generator corresponding to the vertex $v_{n+1}$. Then the submodule $M$ of
$\mathfrak{g}(A')$ generated by $\mathfrak{g}(A)$ acting on $Y_{n+1}$ is an integrable highest weight module over the
subalgebra $\mathfrak{g}(A)$.
\end{lemma}

\begin{proof}
The fact that $M$ is a highest weight module follows immediately from $[X_i,Y_{n+1}]=0$ for all $i=1,\dots,n$. The module
$M$ has highest weight $-\alpha_{n+1}$.  A real root $\alpha$ of the subalgebra $\mathfrak{g}(A)$ is also a real root of
$\mathfrak{g}(A')$.  By Lemma~\ref{lm411}, the adjoint module of $\mathfrak{g}(A')$ is integrable.  Thus for each real
root $\alpha$ of $\mathfrak{g}(A)$ we have that $Y_{\alpha} \in \mathfrak{g}(A)_{-\alpha}$  acts locally nilpotently on
the submodule $M$ of $\mathfrak{g}(A')$.  Hence the submodule $M$ is an integrable highest weight module over the
subalgebra $\mathfrak{g}(A)$.
\end{proof}

\begin{corollary}
If $\Gamma$ is a diagram for a regular Kac-Moody superalgebra $\mathfrak{g}(A)$ that does not have non-trivial
irreducible integrable highest weight modules, then $\Gamma$ is not extendable.
\end{corollary}

\begin{proof}
If $\mathfrak{g}(A)$ has only trivial irreducible integrable highest weight modules, then the highest weight of the
module $M$ is $0$.  Hence $-\alpha_{n+1}=0$, which implies $a_{i,n+1}=0$ for $i=1,\dots,n$.  Since we assumed that the
matrix $A$ is a generalized Cartan matrix, this implies $a_{n+1,j}=0$ for $j=1,\dots,n$.  This is not possible with
$\Gamma'$ being a connected diagram.
\end{proof}

\begin{corollary}\label{cor4}
Suppose that $\Gamma$ is a subfinite regular Kac-Moody diagram for $\mathfrak{g}(A)$ which is not of finite type and not
one of the algebras: $A(0,m)^{(1)}$, $C(n)^{(1)}$, $S(1,2,\alpha)$, and $Q^{\pm}(m,n,t)$. Then the diagram $\Gamma$ is
not extendable.
\end{corollary}

\begin{proof} This follows immediately from Lemma~\ref{lm123}.
\end{proof}

From the classification of 3-vertex regular Kac-Moody diagrams we obtain:

\begin{lemma}\label{3rm}
If $\Gamma$ is a 3-vertex regular Kac-Moody diagram which is not a diagram for $Q^{\pm}(m,n,t)$, $S(1,2,\alpha)$ or $D(2,1,\alpha)$, then the ratio of an isotropic vertex in $\Gamma$ is rational and negative.
The ratio of an isotropic vertex for $Q^{\pm}(m,n,t)$ is real, irrational and negative.  For $S(1,2,\alpha)$ and $D(2,1,\alpha)$, if the ratios are rational then at most one of them is positive.
\end{lemma}

\begin{theorem}\label{thm2}\label{p3}
A connected regular Kac-Moody diagram containing an isotropic vertex is subfinite.  In particular, if
$\Gamma$ is a subfinite regular Kac-Moody diagram which is not of finite type, then $\Gamma$ is not extendable.
\end{theorem}
\begin{proof}
We prove this by induction on the number of vertices.   Suppose that the claim is true for all diagrams with less than $n$ vertices, and let $\Gamma$ be a connected $n$-vertex diagram which is regular Kac-Moody and contains an isotropic vertex.  If $\Gamma'$ is a connected proper subdiagram of $\Gamma$ containing an isotropic vertex, then $\Gamma'$ is regular Kac-Moody and so by the induction hypothesis $\Gamma'$ is subfinite.

Let $\mathcal{S}$ denote the set of connected subfinite regular Kac-Moody
which are extendable. By Corollary~\ref{cor4},
$\Gamma'\in\mathcal{S}$ is either of finite type or can be a diagram for $A(0,m)^{(1)}$, $C(n)^{(1)}$, $S(1,2,\alpha)$, $Q^{\pm}(m,n,t)$. In the following lemmas, we prove that if $\Gamma'$ is a diagram for: $A(0,m)^{(1)}$, $C(n)^{(1)}$, $S(1,2,\alpha)$, or $Q^{\pm}(m,n,t)$, then $\Gamma'$ is not a proper subdiagram of a connected regular
Kac-Moody diagram $\Gamma$ which satisfies the condition for all reflected diagrams: all proper connected regular Kac-Moody diagrams containing an isotropic vertex are in $\mathcal{S}$.

\begin{lemma} $Q^{\pm}(m,n,t)$ is not extendable and hence $Q^{\pm}(m,n,t)\not\in\mathcal{S}$. \end{lemma}

\begin{proof}We consider each case for attaching a vertex to a $Q^{\pm}(m,n,t)$ diagram.  Let $\Gamma$ denote the extended diagram.  Recall that $a,b,c \in \left(\mathbb{R}\setminus\mathbb{Q}\right)_{<0}$ satisfy: \begin{equation}\label{eqn11}
\begin{array}{l}
1+a+\frac{1}{b}=m \\
1+b+\frac{1}{c}=n \\
1+c+\frac{1}{a}=t \\
\end{array}
\in \mathbb{Z}_{<0} \text{, or all equal to zero.}
\end{equation}

\noindent\textbf{Case 1:}
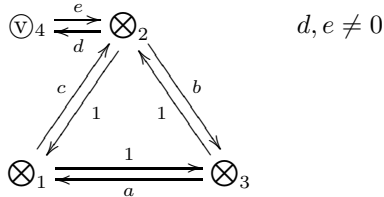
\begin{equation*}
\xymatrix{\OV_4 \AW[r]^{e}_{d}& \OX_2 \AW[ldd]^{1}_{c} \AW[rdd]^{b}_{1} &\\ & & \\
  \OX_1 \AW[rr]^{1}_{a} &  & \OX_3  \\}\hspace{.5cm}d,e\neq 0
\end{equation*}
Now $\Gamma_1,\Gamma_3\in\mathcal{S}$ implies $d, \frac{b}{d} \in \mathbb{Q}$.
But then $b \in \mathbb{Q}$, which is a contradiction.

\noindent\textbf{Case 2:}
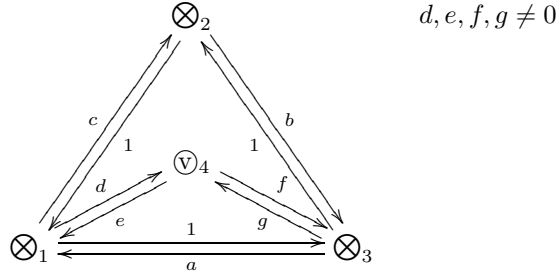
\begin{equation*}
\xymatrix{ & &  \OX_2 \AW[llddd]^{1}_{c} \AW[rrddd]^{b}_{1} & & \\ & & & & \\
& &  \OV_4 \AW[lld]^{e}_{d} \AW[rrd]^{f}_{g} & & \\ \OX_1 \AW[rrrr]^{1}_{a} & &  & & \OX_3 \\}\hspace{.5cm}d,e,f,g\neq 0
\end{equation*}
Now $\Gamma_1\in\mathcal{S}$ implies $g<0$, and $\Gamma_3\in\mathcal{S}$ implies $\frac{c}{d}<0$. Since $a,b,c < 0$, this implies $d, \frac{g}{a} >0$.  By Lemma~\ref{3rm}, this implies $\Gamma_2\not\in\mathcal{S}$, which is a contradiction.

\noindent\textbf{Case 3:}
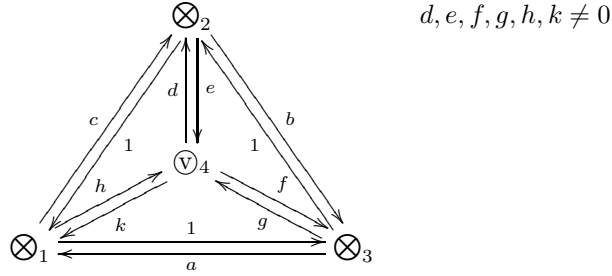
\begin{equation*}
\xymatrix{ & &  \OX_2 \AW[llddd]^{1}_{c} \AW[rrddd]^{b}_{1} & & \\ & & & & \\
& &  \OV_4 \AW[uu]^{d}_{e} \AW[lld]^{k}_{h} \AW[rrd]^{f}_{g} & & \\ \OX_1 \AW[rrrr]^{1}_{a} & &  & & \OX_3  \\}\hspace{.5cm}d,e,f,g,h,k\neq 0
\end{equation*}
The ratios of $v_1,v_2,v_3$ in $\Gamma_4$ are $a,b,c\in(\mathbb{R}\setminus\mathbb{Q})_{<0}$. If $\Gamma_1$ is a subdiagram such that the ratios of $v_2$ and $v_3$ in $\Gamma_1$ are real negative numbers, then the ratio of $v_2$ in $\Gamma_3$ and of $v_3$ in $\Gamma_2$ are real positive numbers. But then by Lemma~\ref{3rm} all of the ratios of $v_1$ are real negative numbers, which is a contradiction. Hence, $\Gamma_1,\Gamma_2,\Gamma_3$ are diagrams for $D(2,1,\alpha)$ or $S(1,2,\alpha)$.

\noindent\D If $v_4$ is $\OX$ and $\Gamma_1,\Gamma_2,\Gamma_3$ are diagrams for $D(2,1,\alpha)$,
then by (\ref{e3}) the we have:
\begin{align}
&h+\frac{g}{a}=-1 &e+\frac{h}{c}=-1 &&g+\frac{e}{b}=-1 \label{row12} \\
&\frac{a}{g}+\frac{k}{f}=-1 &\frac{k}{d}+\frac{1}{e}=-1 &&\frac{b}{e}+\frac{f}{d}=-1 \\
&\frac{f}{k}+\frac{1}{h}=-1 &\frac{c}{h}+\frac{d}{k}=-1  &&\frac{d}{f}+\frac{1}{g}=-1.
\end{align}
By solving (\ref{row12}) for $h$ we find that $h=\frac{bc-c-abc}{1+abc}\in\mathbb{R}$, since $abc\neq -1$ by Lemma~\ref{lm333}. Similarly, $d,e,f,g,k\in\mathbb{R}$ and all vertex ratios are real.  Now at least one ratio at each vertex $v_1,v_2,v_3$ must be positive, and the diagrams $\Gamma_1,\Gamma_2,\Gamma_3$ are $D(2,1,\alpha)$ diagrams and so they have at most one positive vertex ratio.  This implies that all of the ratios at $v_4$ are negative, which is a contradiction.

\noindent\D If $v_4$ is $\O$ and $\Gamma_1,\Gamma_2,\Gamma_3$ are diagrams for $S(1,2,\alpha)$, then
$d,f,k=-1$, $h+\frac{g}{a}=-2$, $e+\frac{h}{c}=-2$ and $g+\frac{e}{b}=-2$.  Solving for $e$ we find that $e=\frac{2(ab-b-abc)}{1+abc}\in\mathbb{R}$ since $abc\neq -1$ by Lemma~\ref{lm333}.  Similarly, $h,g\in\mathbb{R}$. Now by reflecting at $v_{2}$ we have that $\Gamma'$ is
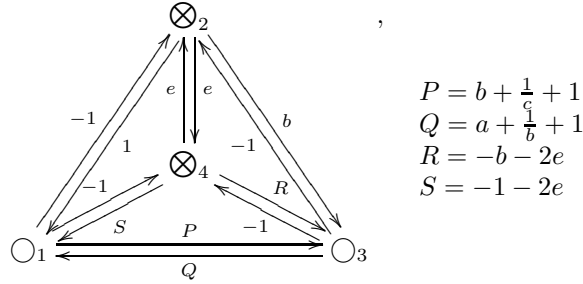
\begin{equation*}\begin{array}{c}
\xymatrix{ & &  \OX_2 \AW[llddd]^{1}_{-1} \AW[rrddd]^{b}_{-1} & & \\ & & & & \\
& &  \OX_4 \AW[uu]^{e}_{e} \AW[lld]^{S}_{-1} \AW[rrd]^{R}_{-1} & & \\ \O_1
\AW[rrrr]^{P}_{Q} & &  & & \O_3 },\end{array}\text{  }\begin{array}{l}
P=b+\frac{1}{c}+1\\ Q=a+\frac{1}{b}+1 \\R =-b-2e \\ S=-1-2e
\end{array}
\end{equation*}
which implies that $\Gamma'_2$ is a diagram $Q^{\pm}(m,n,t)$, and the ratio of $v_4$ in $\Gamma'_2$ is $\frac{-b-2e}{-1-2e} \in (\mathbb{R}\setminus\mathbb{Q})_{<0}$.  But by substituting $e=\frac{2(ab-b-abc)}{1+abc}$ we have
$$\frac{-b-2e}{-1-2e}=\frac{4abc+3b-4ab-ab^{2}c}{3abc+4b-4ab-1}>0,$$
which is a contradiction.

We conclude that $Q^{\pm}(m,n,t)$ is not extendable.
\end{proof}

\begin{lemma}
$S(1,2,\alpha)$ is not extendable and hence $S(1,2,\alpha)\not\in\mathcal{S}$.
\end{lemma}

\begin{proof}
 Note that an odd reflection of a $S(1,2,\alpha)$ diagram is again a $S(1,2,\alpha)$ diagram, but with a different $\alpha$.  Let $a=\frac{1}{\alpha}$.  Then $a\not\in\mathbb{Z}$.

\noindent\textbf{Case 1:}
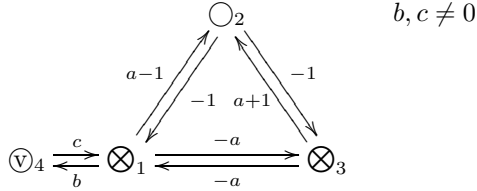
\begin{equation*}
\xymatrix{ & & \O_2 \AW[ldd]^{-1}_{a-1} \AW[rdd]^{-1}_{a+1\!\!} & \\ & & &  \\ \OV_4 \AW[r]^{c}_{b} &  \OX_1
\AW[rr]^{-a}_{-a} & & \OX_3 }\hspace{.5cm}b,c\neq 0
\end{equation*}

\noindent\D If $v_4$ is $\OX$, then $\Gamma_2\in\mathcal{S}$ implies $c=1$ and $b=a$.  By reflecting at $v_4$
we have
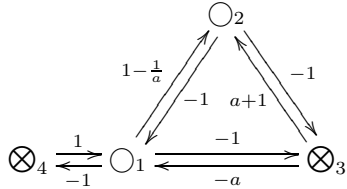
\begin{equation*}
\xymatrix{ & & \O_2  \AW[ldd]^{\!\! -1}_{1-\frac{1}{a}} \AW[rdd]^{-1}_{a+1\!} & \\ & & &  \\ \OX_4 \AW[r]^{1}_{-1} &  \O_1 \AW[rr]^{-1}_{-a} & & \OX_3 }.
\end{equation*}
But $\Gamma'_4\not\in\mathcal{S}$.

\noindent\D If $v_4$ is $\OD$ and $c=-2$, then $\Gamma_2\in\mathcal{S}$ implies $b=a$. Then $\Gamma_3\in\mathcal{S}$ implies $\frac{a-1}{a} \in \{-1,\frac{-3}{2}\}$, and so $a\in\{\frac{1}{2},\frac{2}{5}\}$.  By reflecting at $v_3$ and then at $v_2$, we obtain $\Gamma'':=r_2(r_3(\Gamma))$
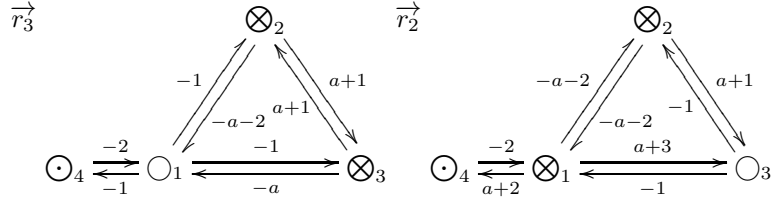
\begin{equation*}
\overrightarrow{r_3} \xymatrix{ & & \OX_2 \ar@{-^>}@<.5ex>[ldd]^(.6){\!\! -a-2} \ar@{_<-}@<-.5ex>[ldd]_{-1} \AW[rdd]^{a+1}_{a+1\!\!}
& \\ & & &  \\ \OD_4 \AW[r]^{-2}_{-1} &  \O_1 \AW[rr]^{-1}_{-a} & & \OX_3 } \overrightarrow{r_2} \xymatrix{ & & \OX_2
 \ar@{-^>}@<.5ex>[ldd]^(.6){\!\! -a-2} \ar@{_<-}@<-.5ex>[ldd]_{-a-2}\AW[rdd]^{a+1}_{-1\!} & \\ & & &  \\ \OD_4 \AW[r]^{-2}_{a+2} &
\OX_1 \AW[rr]^{a+3}_{-1} & & \O_3 }.
\end{equation*}
Then $\Gamma''_2\in\mathcal{S}$ implies $\frac{a+3}{a+2} \in \{-1,\frac{-3}{2}\}$, which is a contradiction.

\noindent\D If $v_4$ is $\O$ and $c=-1$, then $\Gamma_2\in\mathcal{S}$ implies $b=ka$ with $k\in\{1,2,3\}$.
Substituting $b=ka$ and reflecting at $v_1$ yields  $\Gamma'$
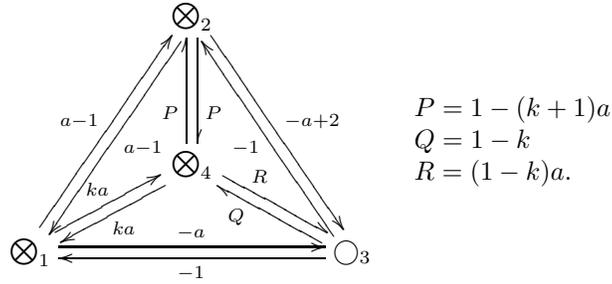
\begin{equation*}\begin{array}{c}
\xymatrix{ & & \OX_2 \AW[llddd]^{a-1}_{a-1} \AW[rrddd]^{-a+2}_{-1} & & \\ & & & & \\
& & \OX_4  \ar@{-^>}@<.5ex>[uu]^(.35){P} \ar@{_<-}@<-.5ex>[uu]_(.35){P} \AW[lld]^{ka}_{ka}
\ar@{-^>}@<.5ex>[rrd]^(.35){R} \ar@{_<-}@<-.5ex>[rrd]_(.35){Q} & & \\
\OX_1 \AW[rrrr]^{-a}_{-1}  & & & & \O_3 }\end{array}\text{  }
\begin{array}{l}
P=1-(k+1)a\\
Q=1-k\\
R=(1-k)a.
\end{array}
\end{equation*}
First suppose that $P=1-(k+1)a=0$. Then $\Gamma'_1\in\mathcal{S}$ implies $Q\in\{0,-1\}$ and so $k \in \{1,2\}$.  If $k=1$, then we are reduced to a previous case.  If $k=2$, then $P=0$ implies that $a=\frac{1}{3}$.  Reflecting at $v_4$ of $\Gamma'$ then yields $\Gamma''$
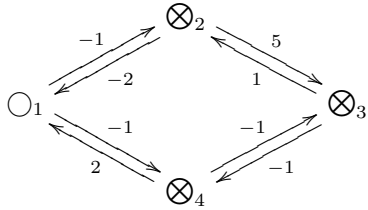
\begin{equation*}
\xymatrix{ & & \OX_2 \AW[lld]^{-2}_{-1} \AW[rrd]^{5}_{1} & & \\ \O_1 & & & & \OX_3 \\
& & \OX_4 \AW[llu]^{2}_{-1} \AW[rru]^{-1}_{-1} & & }.
\end{equation*}
But, $\Gamma''_4\not\in\mathcal{S}$.

Now we assume that $P\neq 0$.  Then reflecting at $v_2$ of $\Gamma'$ yields $\Gamma''$
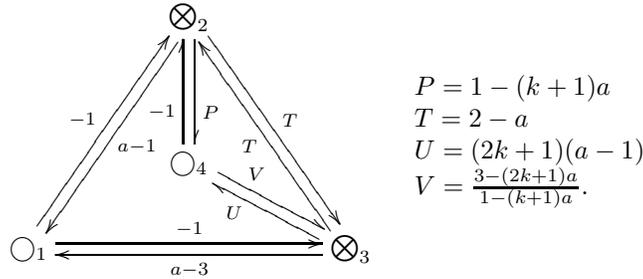
\begin{equation*}\begin{array}{c}
\xymatrix{ & & \OX_2 \AW[llddd]^{\! a-1}_{-1} \AW[rrddd]^{T}_{T} & & \\ & & & & \\
& & \O_4  \ar@{-^>}@<.5ex>[uu]^(.35){-1} \ar@{_<-}@<-.5ex>[uu]_(.35){P}
\ar@{-^>}@<.5ex>[rrd]^(.35){V} \ar@{_<-}@<-.5ex>[rrd]_(.35){U} & & \\
\O_1 \AW[rrrr]^{-1}_{a-3}  & & & & \OX_3 }\end{array}\text{  }
\begin{array}{l}
P=1-(k+1)a\\
T=2-a\\
U=(2k+1)(a-1)\\
V=\frac{3-(2k+1)a}{1-(k+1)a}.
\end{array}
\end{equation*}
Now $\Gamma''_2\in\mathcal{S}$ implies $V\in \{0,-1,-2\}$.   If $V=0$, then $U=(2k+1)(a-1)=0$ and so $a=1$, which contradicts  $a\not\in\mathbb{Z}$.
If $V=-2$, then $\Gamma''_1\in\mathcal{S}$ implies $P/T=-1/3$, and so $k=1$, $a=\frac{5}{7}$.
But then $\Gamma''_2\not\in\mathcal{S}$.
If $V=-1$, then $a=\frac{4}{3k+2}$.  Now $\Gamma''_1\in\mathcal{S}$ implies that either $\Gamma''$ is $S(1,2,\beta)$ and $P+U=-2T$, or  $\Gamma''$ is $C(3)$ and $P/T=-1/2$. If $P+U=-2T$, then $a=2$ which contradicts  $a\not\in\mathbb{Z}$. If $P/T=-1/2$, then $k=1$, $a=\frac{4}{5}$. Reflecting at $v_3$ of $\Gamma''$ yields $\Gamma'''$
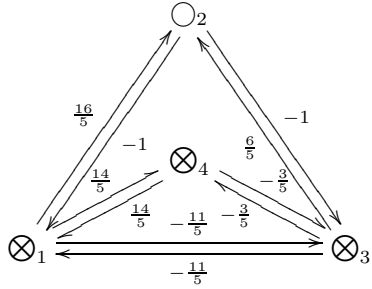
\begin{equation*}
\xymatrix{ & & \O_2 \AW[llddd]^{-1}_{\frac{16}{5}} \AW[rrddd]^{-1}_{\frac{6}{5}} & & \\ & & & & \\
& & \OX_4
\ar@{-^>}@<.5ex>[lld]^(.35){\!\!\frac{14}{5}} \ar@{_<-}@<-.5ex>[lld]_(.5){\frac{14}{5}\!\!}
\ar@{-^>}@<.5ex>[rrd]^(.5){\!\!\!\! -\frac{3}{5}} \ar@{_<-}@<-.5ex>[rrd]_(.35){-\frac{3}{5}\!\!} & & \\
\OX_1 \AW[rrrr]^{-\frac{11}{5}}_{-\frac{11}{5}} & & & & \OX_3 }.
\end{equation*}
But then $\Gamma'''_3\not\in\mathcal{S}$.

\noindent\textbf{Case 2:}
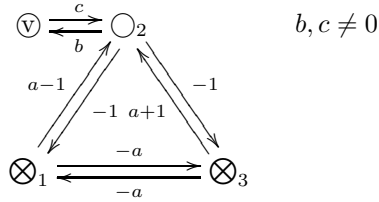
\begin{equation*}
\xymatrix{ & \OV \AW[r]^{c}_{b}  & \O_2 \AW[ldd]^{-1}_{a-1} \AW[rdd]^{-1}_{a+1} & \\ & & &  \\ &  \OX_1 \AW[rr]^{-a}_{-a}
& & \OX_3 }\hspace{.5cm}b,c\neq 0
\end{equation*}
By reflecting at $v_1$, we are reduced to Case 1.

\noindent\textbf{Case 3:}
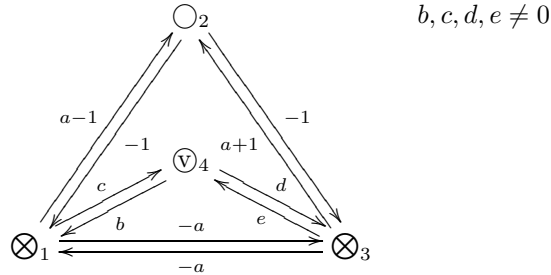
\begin{equation*}
\xymatrix{ & & \O_2 \AW[llddd]^{-1}_{a-1} \AW[rrddd]^{-1}_{a+1} & & \\ & & & & \\
& & \OV_4 \AW[lld]^{b}_{c} \AW[rrd]^{d}_{e} & & \\
\OX_1 \AW[rrrr]^{-a}_{-a} & & & & \OX_3 }\hspace{.5cm}b,c,d,e\neq 0
\end{equation*}
Now $\Gamma_2\in\mathcal{S}$ implies $v_4$ is either $\OX$ or $\O$.

\noindent\D If $v_4$ is $\OX$, then $\Gamma_2\in\mathcal{S}$ implies $b=c$, $d=e$, and $e=a-c$.  Also, $\Gamma_1,\Gamma_3\in\mathcal{S}$ implies $\frac{a-1}{c}, \frac{a+1}{a-c} \in \{-1,-2,-3\}$.
If $\frac{a-1}{c}=\frac{a+1}{a-c}$, then $c=\frac{a-1}{2}$ or $\frac{a-1}{c}=2$, which is a contradiction.  If either fraction equals $-1$, then a reflection at the corresponding vertex returns us to Case 1.  So without loss of generality by symmetry we have that $\frac{a-1}{c}=-2$ and $\frac{a+1}{a-c}=-3$. Then $a=\frac{1}{11}$ and $c=\frac{5}{11}$. By substituting and then reflecting at $v_3$, we obtain $\Gamma'$
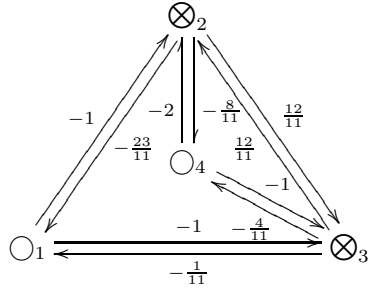
\begin{equation*}
\xymatrix{ & & \OX_2 \AW[llddd]^{\!\! -\frac{23}{11}}_{-1} \AW[rrddd]^{\frac{12}{11}}_{\frac{12}{11}} & & \\ & & & & \\
& & \O_4 \ar@{-^>}@<.5ex>[uu]^(.35){-2} \ar@{_<-}@<-.5ex>[uu]_(.35){-\frac{8}{11}}
\AW[rrd]^{\!\! -1}_{-\frac{4}{11}} & & \\
\O_1 \AW[rrrr]^{-1}_{-\frac{1}{11}} & & & & \OX_3 }.
\end{equation*}
But, $\Gamma'_3\not\in\mathcal{S}$.

\noindent\D If $v_4$ is $\O$, then $\Gamma_2\in\mathcal{S}$ implies that either $b,d=-1$ or $b=-1$, $d=-2$, without loss of generality by symmetry.

If $b=-1$, $d=-2$, then $\Gamma_2$ is a $G(3)$ diagram and $c=\frac{a}{3}$,
$e=\frac{2a}{3}$.  Then $\Gamma_1\in\mathcal{S}$ implies $e=-a-1$. Thus $a=-\frac{3}{5}$.  By substituting and then reflecting at $v_1$, we obtain $\Gamma'$
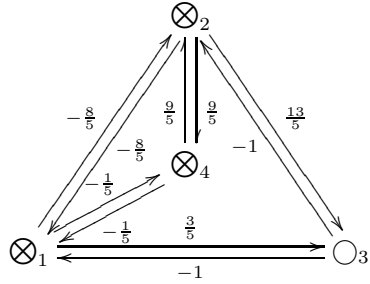
\begin{equation*}
\xymatrix{ & & \OX_2 \AW[llddd]^{\!\!-\frac{8}{5}}_{-\frac{8}{5}} \AW[rrddd]^{\frac{13}{5}}_{-1} & & \\ & & & & \\
& & \OX_4 \ar@{-^>}@<.5ex>[uu]^(.35){\frac{9}{5}} \ar@{_<-}@<-.5ex>[uu]_(.35){\frac{9}{5}}
\AW[lld]^{\!\!\! -\frac{1}{5}}_{-\frac{1}{5}\!\!} & & \\
\OX_1 \AW[rrrr]^{\frac{3}{5}}_{-1} & & & & \O_3 }.
\end{equation*} But then $\Gamma'_1\not\in\mathcal{S}$.

If $b,d=-1$, then $\Gamma_2\in\mathcal{S}$ implies that either $\Gamma_2$ is a $C(3)$ diagram and $-\frac{c}{a}=-\frac{e}{a}=-\frac{1}{2}$, or $\Gamma_2$ is a $S(1,2,\beta)$ diagram and $c+e=2a$.
Now by reflecting $\Gamma$ at $v_1$ we obtain $\Gamma'$
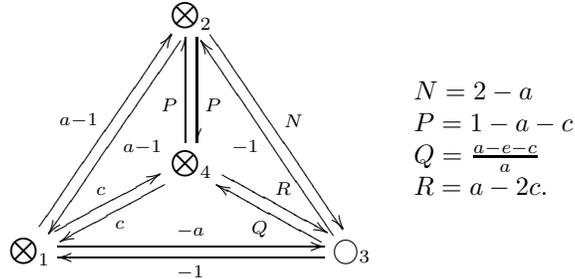
\begin{equation*}\begin{array}{c}
\xymatrix{ & & \OX_2 \AW[llddd]^{a-1}_{a-1} \AW[rrddd]^{N}_{-1} & & \\ & & & & \\
& & \OX_4 \ar@{-^>}@<.5ex>[uu]^(.4){P} \ar@{_<-}@<-.5ex>[uu]_(.4){P}
\AW[lld]^{c}_{c} \AW[rrd]^{R}_{Q} & & \\
\OX_1 \AW[rrrr]^{-a}_{-1} & & & & \O_3 }\end{array}\text{  }
\begin{array}{l}
N=2-a\\
P=1-a-c\\
Q=\frac{a-e-c}{a}\\
R=a-2c.
\end{array}
\end{equation*}
If $\Gamma_2$ is a $C(3)$ diagram with $-\frac{c}{a}=-\frac{e}{a}=-\frac{1}{2}$, then $R,Q=0$ and $\Gamma'$ reduces to the previous subcase. Thus $\Gamma_2$ is a $S(1,2,\beta)$ diagram with $e=2a-c$ and $Q=-1$.
Then by reflecting $\Gamma$ at $v_3$ we obtain $\Gamma''$
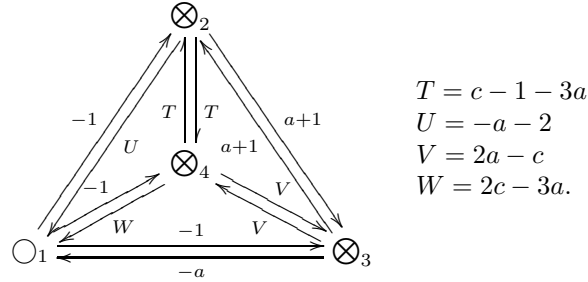
\begin{equation*}\begin{array}{c}
\xymatrix{ & & \OX_2 \AW[llddd]^{U}_{-1} \AW[rrddd]^{a+1}_{a+1} & & \\ & & & & \\
& & \OX_4 \ar@{-^>}@<.5ex>[uu]^(.35){T} \ar@{_<-}@<-.5ex>[uu]_(.35){T}
\AW[lld]^{W}_{-1} \AW[rrd]^{V}_{V} & & \\
\O_1 \AW[rrrr]^{-1}_{-a} & & & & \OX_3 }\end{array}\text{  }
\begin{array}{l}
T=c-1-3a\\
U=-a-2\\
V=2a-c\\
W=2c-3a.
\end{array}
\end{equation*}

If $P\neq 0$, then $\Gamma'_1\in\mathcal{S}$ implies that either $N=R=-P/2$ or $N+R=-2P$.  But if $N=R=-P/2$, then by reflecting $\Gamma'$ at $v_4$ we reduce to Case 1. Similarly, if $T\neq 0$, then $\Gamma''_3\in\mathcal{S}$ implies that either $U=W=-T/2$ or $U+W=-2T$. But if   $U=W=-T/2$, then by reflecting $\Gamma''$ at $v_4$ we reduce to Case 1.
\begin{itemize}
  \item If $P,T=0$, then $1-a=c=1+3a$ implies $a=0$, contradicting $a\not\in\mathbb{Z}$.
  \item If $P=0$, $T\neq 0$, then $c=1-a$. Then by substitution $\Gamma''_3\in\mathcal{S}$ implies that $-6a=8a$, which contradicts $a\not\in\mathbb{Z}$.
  \item If  $P\neq 0$, $T=0$, then $c=1+3a$. Then by substitution $\Gamma'_1\in\mathcal{S}$ implies that $-6a=8$, which contradicts $a\not\in\mathbb{Z}$.
  \item If $P,T\neq 0$, then $a=2+2c$ and $5a=2+2c$, contradicting $a\not\in\mathbb{Z}$.
\end{itemize}

\noindent\textbf{Case 4:}
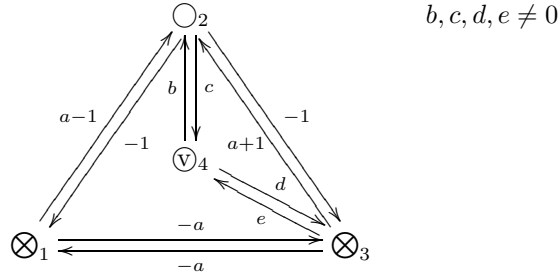
\begin{equation*}
\xymatrix{ & & \O_2 \AW[llddd]^{-1}_{a-1} \AW[rrddd]^{-1}_{a+1\!\!} & & \\ & & & & \\
& & \OV_4 \AW[uu]^{b}_{c} \AW[rrd]^{d}_{e} & & \\
\OX_1 \AW[rrrr]^{-a}_{-a} & & & & \OX_3 }\hspace{.5cm}\ b,c,d,e\neq 0
\end{equation*}
Now $\Gamma_1\in\mathcal{S}$ implies $v_4$ is $\OX$ and $\Gamma_2\in\mathcal{S}$ implies $e=a$.  By reflecting at $v_3$, we are reduced to Case 3.

\noindent\textbf{Case 5:}
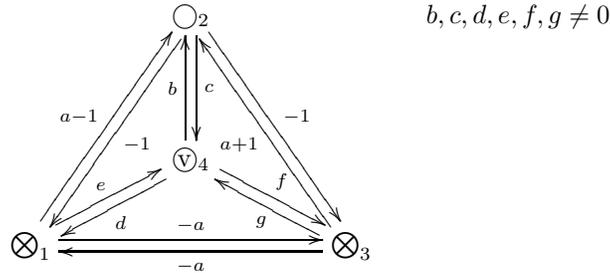
\begin{equation*}
\xymatrix{ & & \O_2 \AW[llddd]^{-1}_{a-1} \AW[rrddd]^{-1}_{a+1} & & \\ & & & & \\
& & \OV_4 \AW[uu]^{b}_{c} \AW[lld]^{d}_{e} \AW[rrd]^{f}_{g} & & \\
\OX_1 \AW[rrrr]^{-a}_{-a} & & & & \OX_3 }\hspace{.5cm}\ b,c,d,e,f,g\neq 0
\end{equation*}
If $v_4$ is $\O$, then $\Gamma_1\not\in\mathcal{S}$.
If $v_4$ is $\OX$ then $\Gamma_2$ is $D(2,1,\alpha)$.  But then a reflection at $v_1$ reduces to a previous case.

We conclude that $S(1,2,\alpha)$ is not extendable.
\end{proof}

\begin{lemma}
$C(n)^{(1)}$ is not extendable and hence $C(n)^{(1)}\not\in\mathcal{S}$.
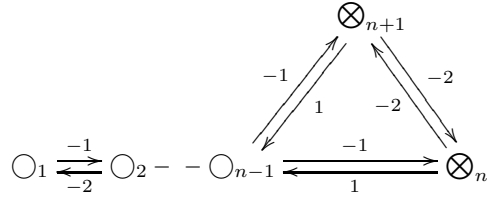
\begin{equation*}
\xymatrix{ & & & \OX_{n+1\!\!\!\!\!\!} \AW[ldd]^{1}_{-1} \AW[rdd]^{-2}_{-2} & \\ & & & & \\
    \O_1 \AW[r]^{-1}_{-2} & \O_2 \ar@{--}[r] & \O_{n-1} \AW[rr]^{-1}_{1} & & \OX_{n} }
\end{equation*}
\end{lemma}
\begin{proof}Let $v_{n+2}$ denote the additional vertex. Now $\Gamma_1\in\mathcal{S}$ implies $a_{n+2,n+1},a_{n+2,n}=0$, and $\Gamma_{n+1}\in\mathcal{S}$ implies $a_{n+2,j}=0$ for $j=1,\dots,n-1$.  Hence, $C(n)^{(1)}$ is not extendable.
\end{proof}

\begin{lemma}
$A(0,m)^{(1)}$ is not extendable and hence $A(0,m)^{(1)}\not\in\mathcal{S}$.
\end{lemma}
\begin{proof}First we show that $A(0,1)^{(1)}$ is not extendable.

\noindent\textbf{Case 1:}
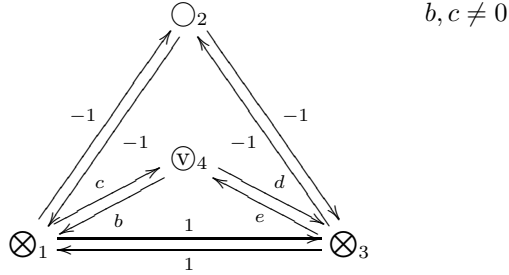
\begin{equation*}
\xymatrix{ & & \O_2 \AW[llddd]^{-1}_{-1} \AW[rrddd]^{-1}_{-1} & & \\ & & & & \\
& & \OV_4  \AW[lld]^{b}_{c} \AW[rrd]^{d}_{e} & & \\ \OX_1 \AW[rrrr]^{1}_{1} & & & & \OX_3 }\hspace{.5cm}\ b,c\neq 0
\end{equation*}
\D If $v_4$ is $\OX$, then $\Gamma_3\in\mathcal{S}$ implies $c>0$. But then the ratio of $v_1$ in $\Gamma_2$ is positive, so $\Gamma_2$ is a $D(2,1,\alpha)$ diagram and $d,e\neq 0$.  Then  $\Gamma_1\in\mathcal{S}$ implies $e>0$. But then the ratio of $v_3$ in  $\Gamma_2$ is also positive, so $\Gamma_2\not\in\mathcal{S}$ by Lemma~\ref{3rm}.

\noindent\D If $v_4$ is $\O$, then $\Gamma_2\in\mathcal{S}$ implies $c<0$.  Then $\Gamma_3\in\mathcal{S}$ implies $b=-1$, since the ratio at $v_1$ is positive.  By reflecting at $v_1$ we obtain $\Gamma'$
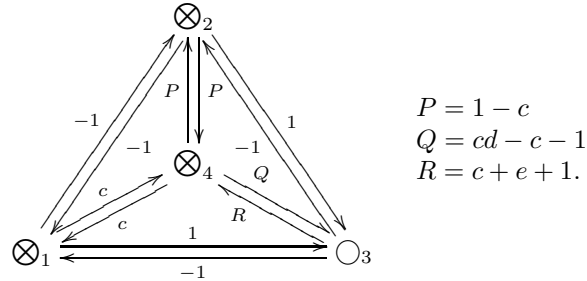
\begin{equation*}\begin{array}{c}
\xymatrix{ & & \OX_2 \AW[llddd]^{-1}_{-1} \AW[rrddd]^{1}_{-1} & & \\ & & & & \\
& & \OX_4 \AW[uu]^{P}_{P}  \AW[lld]^{c}_{c} \ar@{-^>}@<.5ex>[rrd]^(.35){Q} \ar@{_<-}@<-.5ex>[rrd]_(.35){R}
 & & \\ \OX_1 \AW[rrrr]^{1}_{-1} & & & & \O_3 }\end{array}\text{  }
 \begin{array}{l}
 P=1-c\\
 Q=cd-c-1\\
 R=c+e+1.
 \end{array}
\end{equation*}
But $\Gamma'_1\not\in\mathcal{S}$ since the ratio at $v_2$ is positive, namely $P=1-c>0$.

\noindent\textbf{Case 2:}
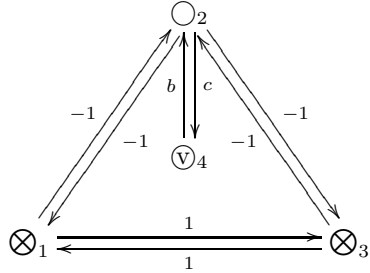
\begin{equation*}
\xymatrix{ & & \O_2 \AW[llddd]^{-1}_{-1} \AW[rrddd]^{-1}_{-1} & & \\ & & & & \\
& & \OV_4 \AW[uu]^{b}_{c}  & & \\ \OX_1 \AW[rrrr]^{1}_{1} & & & & \OX_3 }\hspace{.5cm}\ b,c\neq 0
\end{equation*}
By reflecting at $v_1$ we return to Case 1.

\noindent\textbf{Case 3:}
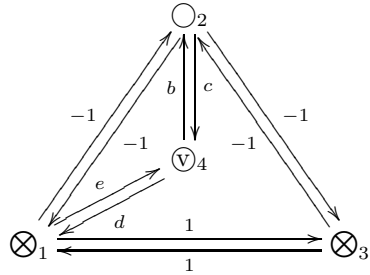
\begin{equation*}
\xymatrix{ & & \O_2 \AW[llddd]^{-1}_{-1} \AW[rrddd]^{-1}_{-1} & & \\ & & & & \\
& & \OV_4 \AW[uu]^{b}_{c} \AW[lld]^{d}_{e}  & & \\ \OX_1 \AW[rrrr]^{1}_{1} & & & & \OX_3 }\hspace{.5cm}\ b,c,d,e\neq 0
\end{equation*}
Now $\Gamma_2\in\mathcal{S}$ implies $e<0$. Then $\Gamma_3\not\in\mathcal{S}$ since the ratio at $v_1$ is positive.

\noindent\textbf{Case 4:}
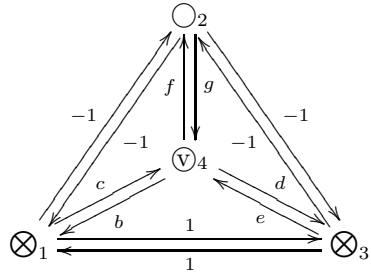
\begin{equation*}
\xymatrix{ & & \O_2 \AW[llddd]^{-1}_{-1} \AW[rrddd]^{-1}_{-1} & & \\ & & & & \\
& & \OV_4 \AW[uu]^{f}_{g} \AW[lld]^{b}_{c} \AW[rrd]^{d}_{e} & & \\ \OX_1 \AW[rrrr]^{1}_{1} & & & & \OX_3 }\hspace{.5cm}\ b,c,d,e,f,g\neq 0
\end{equation*}
Now $\Gamma_1,\Gamma_3\in\mathcal{S}$ implies $v_4$ is $\OX$ and $c,e>0$.  But then the vertex ratios of $v_1$ and $v_3$ in $\Gamma_2$ are positive.  Hence $\Gamma_2\not\in\mathcal{S}$ by Lemma~\ref{3rm}.  Therefore, $A(0,1)^{(1)}$ is not extendable.

Now we show that $A(0,2)^{(1)}$ is not extendable.
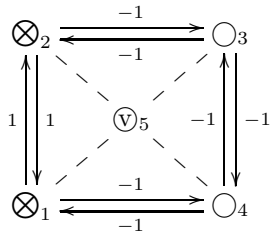
\begin{equation*}
\xymatrix{ \OX_2 \AW[dd]^{1}_{1} \AW[rr]^{-1}_{-1} & & \O_3 \AW[dd]^{-1}_{-1} \\
    & \OV_5 \ar@{--}[ld] \ar@{--}[lu] \ar@{--}[ru] \ar@{--}[rd] & \\ \OX_1 \AW[rr]^{-1}_{-1} & & \O_4 }
\end{equation*}
First suppose that $a_{15}\neq 0$.

\noindent\D If $v_5$ is $\OX$ then $\Gamma_{2,3}\in\mathcal{S}$ implies $a_{15}>0$. Then
$\Gamma_{3,4}$ is $D(2,1,\alpha)$, and so $a_{2,5}\neq 0$.  Then $\Gamma_{1,4}\in\mathcal{S}$ implies $a_{25}>0$.  So $\Gamma_{3,4}$ has two isotropic vertices with a positive ratio, which contradicts Lemma~\ref{3rm}.

\noindent\D If $v_5$ is $\O$ then  $\Gamma_{3,4}\in\mathcal{S}$ implies $a_{15}<0$. Then
then $\Gamma_{2,3}$ is $D(2,1,\alpha)$, and so $a_{51}=-1$. Then by reflecting at $v_1$
we return a previous case.

Therefore, $a_{15}=0$ and by symmetry $a_{25}=0$.  Now without loss of generality suppose that $a_{45} \neq 0$.  Then by reflecting at $v_1$ we have $v'_4$ is $\OX$ and $a'_{45} \neq 0$, but this is the previous case. Hence, $A(0,2)^{(1)}$ is not extendable.  \\

Finally we show that $A(0,m)^{(1)}$ is not extendable, for $m\geq 3$. Let $\Gamma$ be a diagram for $A(0,m)^{(1)}$.  Then $\Gamma$ has $m+2 \geq 5$ vertices.  Let $v_1$ denote the vertex being added to the diagram. First suppose $v_1$ is connected to an isotropic vertex, which we denote $v_2$. Let $v_3$ and $v_4$ denote the vertices adjacent to $v_2$ in $\Gamma'$.  Since $v_2$ is isotropic with degree 3, it must be contained in a
$D(2,1,\alpha)$ subdiagram.  Since the subdiagram $\{v_3,v_2,v_4\}$ is not $D(2,1,\alpha)$, we have without loss of
generality that the subdiagram $\{v_1,v_2,v_3\}$ is $D(2,1,\alpha)$.  Then $v_1$ is either $\OX$ or $\O$.  If $v_1$ is
$\OX$ then we have the following subdiagram:
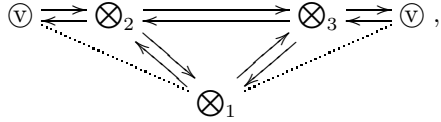
\begin{equation*}
    \xymatrix{ \OV \AW[r]^{}_{} & \OX_2 \AW[rr]^{}_{} \AW[rd]^{}_{} & &
    \OX_3 \AW[ld]^{}_{} \AW[r]^{}_{} & \OV\\
    & & \OX_1 \ar@{..}[llu] \ar@{..}[rru] & & \\},
    \end{equation*}
where the double lines are necessarily connected, and the dotted lines are possibly connected.  By the
$5$-vertex classification, this diagram is not subfinite regular Kac-Moody.  If $v_1$ is $\O$,
then by reflecting at $v_2$ we return to the this case.

Next suppose $v_1$ is not connected to an isotropic vertex.  Then let $\Gamma''$ be a minimal subdiagram containing $v_1$ and an isotropic vertex, which we denote $v_n$.  By minimality of $\Gamma''$ it is a chain such that $v_i$ is not
isotropic for $1<i<n$.  Thus by Lemma~\ref{lm12.2}, there is a sequence of odd reflections $R$ such that $R(v_2)$ is
isotropic and connected to $v_1$.  This reduces us to the previous case.
\end{proof}
This concludes the proof of Theorem~\ref{thm2}.
\end{proof}

\bibliographystyle{amsplain}

\begin{thebibliography}{}



\bibitem{EF03} S. Eswara Rao and V. Futorny, {\it Irreducible modules for affine Lie superalgebras},
e-print: math/0311207 (2003).

\bibitem{H07} C. Hoyt, {\it Kac-Moody superalgebras of finite growth}, Ph.D.
  thesis, University of California, Berkeley (2007).

\bibitem{HS07} C. Hoyt and V. Serganova, {\it Classification of finite-growth general Kac-
Moody superalgebras}, Communications in Algebra \textbf{35 no. 3} (2007), 851-874.

\bibitem{K77} V.G. Kac, {\it Lie superalgebras}, Advances in Math {\bf 26} (1977), 8-96.

\bibitem{K78} V.G. Kac, {\it Infinite-dimensional algebras, Dedekind's $\eta$-function, classical M\"{o}bius function and the very strange formula}, Advances in Math {\bf 30} (1978), 85-136.

\bibitem{K90} V.G. Kac, {\it Infinite dimensional Lie algebras, 3rd ed.}, Cambridge University Press, 1990.

\bibitem{KL89} V.G. Kac and J.W. van de Leur, {\it On classification of superconformal algebras}, Strings \textbf{88}, World Sci. (1989), 77-106.

\bibitem{KW94} V. G. Kac and M. Wakimoto, Integrable highest weight modules
over affine superalgebras and number theory, Lie Theory and Geometry, Progress in Math. 123, (1994),
415-456.

\bibitem{KW01} V.G. Kac and M. Wakimoto, {\it Integrable highest weight modules over affine superalgebras and Appell's function}, Comm. Math. Phys. {\bf 215} (2001), no. 3, 631-682.

\bibitem{LSS86} D. Leites, M. Savel'ev and V. Serganova, {\it Embeddings of Lie superalgebra $osp(1|2)$ and nonlinear
supersymmetric equations}, Group Theoretical Methods in Physics {\bf 1} (1986), 377-394.

\bibitem{S08} V. Serganova, {\it Kac-Moody superalgebras and integrability}, Perspectives and trends in infinite-dimensional Lie theory.

\bibitem{L86} J.W. van de Leur, {\it Contragredient Lie superalgebras of finite growth}, Thesis Utrecht University
(1986).

\bibitem{L89} J.W. van de Leur, {\it A classification of contragredient Lie superalgebras of finite growth},
Communications in Algebra {\bf 17} (1989), 1815-1841.


\end{thebibliography}

\end{document}